\documentclass[10pt,reqno]{amsart} 

\usepackage{appendix}
\usepackage{float} 

\usepackage{soul}

\usepackage{hhline}

\usepackage{yhmath}\usepackage{epsfig,amssymb,amsmath,version}
\usepackage{amssymb,version,graphicx,fancybox,mathrsfs}
\usepackage{mathtools}

\usepackage{enumitem}
\usepackage{bbm}
\usepackage{url,hyperref}

\usepackage{zref-clever}

\usepackage{subfigure}
\usepackage{color}
\usepackage{stmaryrd}
\usepackage{multirow}
\usepackage{booktabs,siunitx}
\usepackage{booktabs}
\usepackage{multicol,epstopdf}
\usepackage{xypic}
\usepackage[]{graphicx}
\usepackage[all]{xy}
\usepackage{tikz,ifthen}
\usetikzlibrary{positioning, math, decorations.markings, arrows.meta}
\usetikzlibrary{calc,shapes,cd}
\usetikzlibrary{knots} 
\usetikzlibrary{patterns}
\usepgfmodule{decorations}
\allowdisplaybreaks
\usetikzlibrary{arrows.meta,decorations.pathreplacing}

\tikzset{knotarrow/.pic={ \draw[edge, <-] (0,0) -- +(-.001,0);}}

\tikzset{edge/.style={line width=0.8}}
\tikzset{wall/.style={very thick}}

\tikzset{->-/.style n args={2}{decoration={markings, mark=at position #1 with {\arrow{#2}}}, postaction={decorate}}} 

\ExplSyntaxOn
\cs_new_eq:NN \ifstreqF \str_if_eq:nnF
\cs_new_eq:NN \ifstreqTF \str_if_eq:nnTF
\ExplSyntaxOff

\tikzset{-o-/.code 2 args={\ifstreqF{#2}{} 
{\ifstreqTF{#2}{>}
	{\pgfkeysalso{decoration={markings,mark=at position #1 with {\arrow[scale=0.8]{#2}}}
			,postaction={decorate}}
	}
	{\ifstreqTF{#2}{<}
		{\pgfkeysalso{decoration={markings,mark=at position #1 with {\arrow[scale=0.8]{#2}}}
				,postaction={decorate}}
		}
		{\pgfkeysalso{decoration={markings,
					mark=at position #1 with
					{\draw[black, fill={#2}] circle[radius=2pt];}}
				,postaction={decorate}}
		}
	}
}}}

\tikzset{
geod/.style={circle, thick, draw=black, fill=white, inner sep=1.5pt},
cusp/.style={circle, thick, draw=black, fill=white, inner sep=1.5pt},
bmark/.style={circle, thick, draw=black, fill=black, inner sep=1.5pt}
}

\allowdisplaybreaks[4]

\newtheorem{theorem}{Theorem}[section]
\newtheorem{lemma}[theorem]{Lemma}
\newtheorem{corollary}[theorem]{Corollary}
\newtheorem{proposition}[theorem]{Proposition}

\theoremstyle{definition}
\newtheorem{definition}{Definition}[section]

\newtheorem{remark}[definition]{Remark}

\newcommand{\Rlabel}[1]{\hyperref[R#1]{(R#1)}}

\newcommand{\bp}{\begin{proposition}}
\newcommand{\ep}{\end{proposition}}
\newcommand{\bpr}{\begin{proof}}
\newcommand{\epr}{\end{proof}}

\newcommand{\bt}{\begin{theorem}}
\newcommand{\et}{\end{theorem}}

\newcommand{\bl}{\begin{lemma}}
\newcommand{\el}{\end{lemma}}

\newcommand{\bcr}{\begin{corollary}}
\newcommand{\ecr}{\end{corollary}}

\newcommand{\be}{\begin{equation}}
\newcommand{\ee}{\end{equation}}

\newcommand{\bes}{\begin{equation*}}
\newcommand{\ees}{\end{equation*}}

\newcommand{\ba}{\begin{align}}
\newcommand{\ea}{\end{align}}

\newcommand{\bas}{\begin{align*}}
\newcommand{\eas}{\end{align*}}

\DeclareMathOperator{\im}{\mathrm{Im}}

\DeclareMathOperator{\rank}{\text{rank}}

\graphicspath{{./Figures/} } 

\newcommand{\vs}[0]{\vspace{2mm}}

\begin{document}

\bibliographystyle{amsalpha}

\title{Generalized quantum cluster structure of reduced L{\^e}--Roger--Yang skein algebras and their Frobenius maps}

\author[Min Huang]{Min Huang}
\address{Min Huang, School of Mathematics (Zhuhai), Sun Yat-sen University, Zhuhai, China.}
\email{huangm97@mail.sysu.edu.cn}

\author[Zhihao Wang]{Zhihao Wang}
\address{Zhihao Wang, School of Mathematics, Korea Institute for Advanced Study (KIAS), 85 Hoegi-ro, Dongdaemun-gu, Seoul 02455, Republic of Korea}
\email{zhihaowang@kias.re.kr}


\maketitle

\begin{abstract}
We establish a generalized quantum cluster algebra structure, in the sense
of Bai--Chen--Ding--Xu \cite{BCDX0}, on the reduced
L{\^e}--Roger--Yang skein algebra of every quasi-triangulable punctured
surface. We define Frobenius maps for generalized quantum cluster
algebras at roots of unity and use these maps, together with the cluster
algebra structure, to construct Frobenius maps for the (reduced)
L{\^e}--Roger--Yang skein algebras of all punctured surfaces. 
As applications of the Frobenius map, we determine the centers of both
the Muller--Roger--Yang skein algebras and the reduced L{\^e}--Roger--Yang skein algebras   at roots of unity, and
compute their ranks as modules over their centers.
\end{abstract}

\tableofcontents

\newcommand{\ca}{{\cev{a}  }}

\def\BZ{\mathbb Z}
\def\Id{\mathrm{Id}}
\def\Mat{\mathrm{Mat}}
\def\BN{\mathbb N}

\def \cb {\color{blue}}
\def \cred {\color{red}}
\def \cbf {\color{blue}\bf}
\def \credf {\color{red}\bf}
\definecolor{ligreen}{rgb}{0.0, 0.3, 0.0}
\def \cg {\color{ligreen}}
\def \cgf {\color{ligreen}\bf}
\definecolor{darkblue}{rgb}{0.0, 0.0, 0.55}
\def \dbf {\color{darkblue}\bf}
\definecolor{anti-flashwhite}{rgb}{0.55, 0.57, 0.68}
\def \afw {\color{anti-flashwhite}}
\def\cF{\mathcal F}
\def\cP{\mathcal P}
\def\embed{\hookrightarrow}
\def\pr{\mathrm{pr}}
\def\cV{\mathcal V}
\def\ot{\otimes}
\def\buu{{\mathbf u}}


\def \ri {{\rm i}}
\newcommand{\bs}[1]{\boldsymbol{#1}}
\newcommand{\cev}[1]{\reflectbox{\ensuremath{\vec{\reflectbox{\ensuremath{#1}}}}}}
\def\bS{\bar \fS}
\def\cE{\mathcal E}
\def\fB{\mathfrak B}
\def\cR{\mathcal R}
\def\cY{\mathcal Y}
\def\cS{\mathscr S}
\def\rS{\overline{\cS}_q}

\def\fS{\mathfrak{S}}

\def\MN {(M)}
\def\cN {\mathcal{N}}
\def\SL{{\rm SL}_n}

\def\bP{\mathbb P}
\def\bR{\mathbb R}

\def\rd{\overline{\cS}}

\def\SS{\cS_{q}(\fS)}
\def\rdS{\overline \cS_{q}(\fS)}
\def\rdP{\overline \cS_{q}(\mathbb{P}_4)}

\def\Vm{\mathcal V_{\text{mut}}}
\def\Vc{\mathcal V}

\def\Zq{\mathbb Z[q^{\pm \frac{1}{2}}]}
\def\Zpq{\mathbb Z_{\geq 0}[q^{\pm \frac{1}{2}}]}

\newcommand{\beq}{\begin{equation}}
\newcommand{\eeq}{\end{equation}}

\section{Introduction}\label{sec-intro}

Let $R$ be a commutative domain with a distinguished invertible element $q^{\frac{1}{2}}$.

\subsection{(Reduced) L{\^e}--Roger--Yang skein algebras}
The (Kauffman bracket) skein algebra of a surface was introduced independently by Przytycki and Turaev \cite{Przytycki1991,turaev1990conway}. It provides a quantization of the enhanced Teichm{\"u}ller space \cite{turaev1991skein,Bul97,przytycki1997skein,BW11}. L{\^e} subsequently generalized this construction by introducing the stated skein algebra, originally designed to construct the quantum trace map in \cite{le2018triangular}. It was later discovered that the stated skein algebra has many other interesting properties, including the existence of a splitting homomorphism \cite{le2018triangular}, a geometric realization of the quantum group $\mathcal O_q({\rm SL}_2)$ \cite{CLL}, and connections with cluster algebras \cite{muller2016skein,LY22}.

Roger and Yang introduced the Roger--Yang skein algebra \cite{roger2014skein} as a candidate quantization of the decorated Teichm{\"u}ller space. The results of \cite{moon2022consequences,li2025skein,bloomquist2025degenerations} provide further evidence supporting this interpretation. In the classical setting, the Roger--Yang skein algebra is closely related to cluster algebras \cite{karuo2025cluster,li2025skein,mandel2023bracelets}. By contrast, much less is known in the quantum setting, partly because a quantum cluster algebra structure in the sense of \cite{BZ} does not exist when the surface has interior punctures.

Bloomquist, Karuo, and L{\^e} unified the constructions of the stated skein algebra and the Roger--Yang skein algebra by introducing the L{\^e}--Roger--Yang skein algebra \cite{bloomquist2025degenerations}. This algebra is associated with a punctured surface $\fS$ as in Definition~\ref{def-punctured-surface}. An $\fS$-tangle $\alpha$ is a compact, unoriented, properly embedded one-dimensional submanifold of $\fS\times(-1,1)$ satisfying the conditions in Definition~\ref{def-tangle}. 
\begin{align}\label{intro-V}
\text{Let $V$ denote the set of boundary components of $\fS$ that are diffeomorphic to $S^1$.}
\end{align}
Define
\begin{align}\label{intro-R}
\widetilde{R}=R[v^{\frac{1}{2}}\mid v\in V].
\end{align}
A state of $\alpha$ is a map
$(\partial\alpha)\setminus\left(\bigcup_{v\in V}v\right)
\longrightarrow \{-,+\}.$
The \emph{L{\^e}--Roger--Yang skein algebra} $\cS_q(\fS)$ is defined as the quotient of the free $\widetilde{R}$-module generated by the isotopy classes of stated $\fS$-tangles by the relations \eqref{cross}--\eqref{unknot-RY}. Its multiplication is given by stacking stated $\fS$-tangles.
The \emph{reduced L{\^e}--Roger--Yang skein algebra} $\rd_q(\fS)$ is defined as the quotient of $\cS_q(\fS)$ by the ideal generated by the bad arcs depicted in Figure~\ref{fig:badarc}.

When $V=\emptyset$, the L{\^e}--Roger--Yang skein algebra $\cS_q(\fS)$ is the stated skein algebra in \cite{le2018triangular}. When $\fS$ contains no punctures, it is the Roger--Yang skein algebra in \cite{roger2014skein}.

\subsection{Generalized quantum cluster algebras}
Cluster algebras were introduced by Fomin and Zelevinsky \cite{FZ} with the primary goal of studying (dual) canonical bases. By definition, a cluster algebra is a commutative algebra endowed with a distinguished set of generators, known as \emph{cluster variables}. These variables are grouped into overlapping subsets called \emph{seeds}, which are interconnected via a fundamental operation termed \emph{mutation}. A cornerstone of this theory is the \emph{Laurent phenomenon}, which asserts that every cluster variable can be expressed as a Laurent polynomial in the variables of any given seed.

Motivated by the Teichm\"{u}ller theory of Riemann surfaces with orbifold points, Chekhov and Shapiro \cite{CS} introduced \emph{generalized cluster algebras}, a natural extension of the ordinary framework. In a broader context, Lam and Pylyavskyy \cite{LP} developed \emph{Laurent phenomenon algebras}, a further generalization designed specifically to preserve the Laurent phenomenon.

On the quantum side, Berenstein and Zelevinsky \cite{BZ} introduced quantum cluster algebras as quantizations of their classical counterparts. More recently, Bai, Chen, Ding, and Xu \cite{BCDX0} introduced the quantum analogue of generalized cluster algebras, termed \emph{generalized quantum cluster algebras} (see Definition \ref{def:qca}). Notably, both quantum cluster algebras and generalized quantum cluster algebras satisfy the quantum analogue of the Laurent phenomenon; see \cite{BZ,BCDX}. Additionally, fully non-commutative (generalized) cluster algebras arising from surfaces were introduced by Berenstein and Retakh in \cite{BR,BHR}.

Roughly speaking, a generalized quantum cluster algebra is constructed from an initial generalized quantum seed ${\bf s}=(A,{\bf h},\Lambda,\widetilde B)$ (see Definition \ref{def:generalizedQS}), where $A$ is a set of $m$ pairwise quasi-commuting variables, called \emph{cluster variables}, whose commutation relations are governed by $\Lambda$, and the \emph{exchange matrix} $\widetilde B$ is an $m\times n$ skew-symmetrizable matrix compatible with $\Lambda$ in the sense of Berenstein--Zelevinsky \cite{BZ}. New seeds are obtained through iterated applications of \emph{mutations}, which are determined by the string sequence ${\bf h}$ and the exchange matrix $\widetilde B$. The generalized quantum cluster algebra is then defined as the algebra generated by all cluster variables, coefficient strings and all frozen cluster variables with their inverses (see Definition \ref{def:qca}).

The quest to identify (quantum) cluster structures on significant algebraic and geometric objects has garnered substantial interest; see, for instance, \cite{FZ,S,SW21,CGGLS,GLSB,QY} and the references therein. Furthermore, cluster algebras and varieties have been extensively constructed within the context of Teichm\"{u}ller theory \cite{FG06,FST,GSV,GS19} and are intimately linked to skein theory \cite{muller2016skein,ishibashi2023skein,LY22,LY23,KimWang,Kim21,IY1,huang2025quantum,CHW2026}. Despite these advances, examples of generalized (quantum) cluster algebras arising from interesting algebraic and geometric objects remain scarce.

\subsection{Generalized quantum cluster structures on reduced L{\^e}--Roger--Yang skein algebras}
When $V=\emptyset$ (see \eqref{intro-V}) and $\fS$ has no interior punctures, a quantum cluster algebra structure on $\rd_q(\fS)$ was constructed in \cite{muller2016skein,LY22} using the so-called $\mathcal A$-quantum trace map; see \S\ref{sec-A-quantum-trace}.

Our goal is to construct an appropriate analogue of this quantum cluster structure on $\rd_q(\fS)$. 
When $V\neq\emptyset$ or $\fS$ has interior punctures, however, the standard quantum cluster algebra structure associated with $\fS$ is no longer available. To overcome this difficulty, we employ the generalized quantum cluster algebra framework (Definitions~\ref{def:generalizedQS} and \ref{def:qca}).

To construct a generalized quantum seed in $\rd_q(\fS)$, we begin with a quasi-triangulation of $\fS$, namely, a maximal collection of pairwise compatible admissible ideal arcs; see Definition~\ref{def-quasitri}. More precisely, an admissible ideal arc is either an embedding
$(0,1)\hookrightarrow \fS$
connecting two boundary punctures and not cutting out a disk, or an embedding
$[0,1)\hookrightarrow \fS$
connecting a boundary component in $V$ (see \eqref{intro-V}) to a boundary puncture. 

To each admissible ideal arc $\alpha$, we associate an element
$A_\alpha\in\rd_q(\fS);$
see \eqref{eq-D-X-q}. This element will serve as a cluster variable. For any two compatible admissible ideal arcs $\alpha$ and $\beta$ in $\fS$, there exists an integer $\Lambda_{\alpha,\beta}$, defined in \eqref{def--Lambda}, such that
\begin{equation}\label{intro-quasi-commutative}
A_\alpha A_\beta
=
q^{\Lambda_{\alpha,\beta}}A_\beta A_\alpha
\qquad\text{in }\rdS;
\end{equation}
see Proposition~\ref{prop:quasi-commutative}.

Let $\Delta$ be a quasi-triangulation of $\fS$, and set
$\Lambda^\Delta
=
\bigl(\Lambda_{\alpha,\beta}\bigr)_{\alpha,\beta\in\Delta}.$
We define the exchange matrix $B^\Delta$ by \eqref{def1-BB} and \eqref{def2-BB}. In Proposition~\ref{prop:compatible}, we prove that the pair $(\Lambda^\Delta,B^\Delta)$ is compatible in the sense of Definition~\ref{def:compatible pair}.

Following Definition~\ref{def:generalizedQS}, it remains to define the coefficient string
${\bf h}^\Delta
=
\bigl(\mathbf h_\gamma\mid\gamma\in\mathring{\Delta}\bigr),$
where $\mathring{\Delta}\subseteq\Delta$ consists of the admissible ideal arcs that are not isotopic to boundary components of $\fS$. Let $\mathring{\mathcal P}$ denote the set of interior punctures of $\fS$. For each $p\in\mathring{\mathcal P}$, let $\alpha_p$ denote the crossingless loop surrounding $p$, called the \emph{peripheral loop}. For each $\gamma\in\mathring{\Delta}$, define
\begin{equation}\label{intro-eq:h}
\mathbf h_\gamma
=
\begin{cases}
	(1,\alpha_p,1),
	& 
	\text{$\gamma$ is a loop enclosing an interior puncture }
	p\in\mathring{\mathcal P},\\
	(1,1),
	& \text{otherwise }.
\end{cases}
\end{equation}
We denote the skew-field of fractions of $\rdS$ by ${\rm Frac}(\rdS)$. We can now state the following result.

\begin{proposition}[Proposition~\ref{prop:mutation}]
\label{intro:mutation}
Let $\fS$ be a quasi-triangulable punctured surface, and let $\Delta$ be a quasi-triangulation of $\fS$ (Definition~\ref{def-quasitri}). Then
\[
{\bf s}(\Delta)
:=
\bigl((A_\alpha)_{\alpha\in\Delta},
{\bf h}^\Delta,
\Lambda^\Delta,
B^\Delta\bigr)
\]
is a generalized quantum seed in ${\rm Frac}(\rdS)$ (Definition~\ref{def:generalizedQS}). Moreover, suppose that two quasi-triangulations $\Delta$ and $\Delta'$ are related by a flip at an admissible ideal arc $\alpha\in\Delta$. Then
\begin{equation}\label{intro:seedmutation}
	{\bf s}(\Delta')
	=
	\mu_\alpha\bigl({\bf s}(\Delta)\bigr),
\end{equation}
where $\mu_\alpha$ denotes the mutation defined in Definition~\ref{def:mutation}.
\end{proposition}

The main difficulty in proving Proposition~\ref{intro:mutation} is verifying the mutation relation \eqref{eq:mutation-X}. This is accomplished through the diagrammatic computations in \eqref{eqqqq-1}--\eqref{eqqqq-9}.

Let $\mathscr A_q(\fS)$ denote the generalized quantum cluster algebra
(Definition \ref{def:qca}) associated with the generalized quantum seed
${\bf s}(\Delta)$ over the ground ring $\mathfrak R=\widetilde{R}$ 
(see \eqref{intro-R}). By Proposition~\ref{intro:mutation},
$\mathscr A_q(\fS)$ is independent of the choice of quasi-triangulation
$\Delta$.
Our first main result identifies $\mathscr A_q(\fS)$ with the reduced
L{\^e}--Roger--Yang skein algebra $\rdS$. Thus, every quasi-triangulable
punctured surface admits a generalized quantum cluster algebra structure
on its reduced L{\^e}--Roger--Yang skein algebra.

\begin{theorem}[Theorem~\ref{thm:skein=A}, Corollary~\ref{cor:bi}]\label{intro-thm1}
Let $\fS$ be a quasi-triangulable punctured surface (see Definitions~\ref{def-punctured-surface} and \ref{def-quasitri}). Then the following equality holds:
\[
\overline{\cS}_q(\fS)=\mathscr A_q(\fS).
\]
Furthermore, there exist the following canonical bijections:
\begin{enumerate}[label={\rm (\alph*)}]
	\item The exchangeable quantum cluster variables of $\mathscr{A}_q(\fS)$ correspond to the interior admissible ideal arcs in $\fS$;
	\item The frozen quantum cluster variables of $\mathscr{A}_q(\fS)$ correspond to the boundary ideal arcs in $\fS$;
	\item The coefficients of $\mathscr{A}_q(\fS)$ correspond to the peripheral loops in $\fS$;
	\item The quantum clusters of $\mathscr{A}_q(\fS)$ correspond to the quasi-triangulations of $\fS$.
\end{enumerate}
In particular, these bijections are compatible with mutations and flips, respectively.
\end{theorem}

As we mentioned before, when $V=\emptyset$ (see \eqref{intro-V}) and $\fS$ has no interior punctures, Theorem~\ref{intro-thm1} was proved in \cite{muller2016skein,LY22}.

In the classical setting ($q^{\frac{1}{2}}=1$), when $\fS$ has no punctures, it was proved in \cite{moon2022consequences} that $\overline{\cS}_1(\fS)$ contains a copy of a certain commutative cluster algebra associated with $\fS$; see also \cite{karuo2025cluster,li2025skein} for related results. The generalized quantum cluster algebra considered here is different from the commutative cluster algebras studied in \cite{karuo2025cluster,moon2022consequences,li2025skein}; in particular, even the extended exchange matrices differ.

\vs

For readers more familiar with cluster algebras arising from surfaces, we make the following remark.

\begin{remark}\label{rem:introthesame}(Remark \ref{rem:thesame})
Following the terminology of \cite{FST1} (see also \cite{BHR,H2,H3}), we associate to each quasi-triangulable surface $\fS$ an unpunctured orbifold whose orbifold points carry weights. The correspondence is given by converting boundary punctures into boundary marked points, boundary circles into orbifold points of weight $\frac{1}{2}$, and each interior puncture $p$ into an orbifold point of weight $\alpha_p$. Under this correspondence, quasi-triangulations of $\fS$ are identified with triangulations of the resulting unpunctured orbifold. Hence the generalized quantum cluster algebra in Theorem~\ref{intro-thm1} may be regarded as arising from an unpunctured orbifold.
\end{remark}

The generalized quantum cluster algebra structure has several applications. In
\S\ref{intro-sec4}, we use Theorem~\ref{intro-thm1} to construct a Frobenius
map for the reduced L{\^e}--Roger--Yang skein algebra $\rdS$. We also
expect that, together with the results of \cite{muller2024poisson}, this
theorem will help determine the Azumaya locus of $\rdS$ at roots of unity
and thereby classify its irreducible representations of maximal dimension.
As in quantum cluster algebra theory, we define the generalized upper
quantum cluster algebra $\mathscr U_q$ as the intersection of the
associated quantum tori \cite{BCDX0} (Definition~\ref{def:qca}). In forthcoming
work with Lang Mou, we prove that
$\overline{\cS}_q=\mathscr A_q=\mathscr U_q$.

\subsection{Frobenius maps for (reduced) L{\^e}--Roger--Yang skein algebras}\label{intro-sec4}
The representation theory of the (reduced) L{\^e}--Roger--Yang skein algebra is particularly rich when the quantum parameter $q^{\frac12}$ is a root of unity. The Frobenius map plays a central role in its study: its image, or a suitable subalgebra of its image, lies in the center of the (reduced) L{\^e}--Roger--Yang skein algebra.
\begin{equation}\label{intro-eq-unity}
\text{ Suppose that $q^2$ has odd order $\ell$, and set
	$\eta^{\frac12}=q^{\frac{\ell^2}{2}}.$}
	\end{equation}
	We seek an algebra embedding
	\begin{align}\label{intro-eq-Fro}
\cF\colon \rd_\eta(\fS)\longrightarrow\rdS.
\end{align}
We will explain why the requirement ``odd order'' is necessary below. 

We first construct a Frobenius map for generalized quantum cluster algebras (Theorem \ref{thm-Fro-cluster}). Using Theorem~\ref{intro-thm1}, we then obtain a Frobenius map 
\begin{align}\label{eq-intro-tri}
\cF\colon \rd_\eta(\fS)\longrightarrow\rdS\quad
\text{for a quasi-triangulable punctured surface $\fS$}.
\end{align}
Finally, we extend the construction to arbitrary punctured surfaces.

\vs

To be precise, on a quantum torus, the Frobenius map sends each variable to its $\ell$-th power. To extend this map to a generalized quantum cluster algebra, we must also determine its action on the coefficients $h_{i,k}$ (Definition~\ref{def:generalizedQS}) and verify compatibility with mutation. We find that the required image of $h_{i,k}$ is
\[
T_{\ell,\ell_i}^{(k)}
\bigl(h_{i,1},\dots,h_{i,\ell_i-1}\bigr) \qquad (1\leq k\leq \ell_i-1),
\]
where $T_{\ell,\ell_i}^{(k)}$ is the polynomial defined in Definition~\ref{def:generalizedCP}. Although its coefficients are initially defined as complex numbers using a primitive $\ell$-th root of unity, Theorem~\ref{thm:expan} shows that they are all integers. With this choice of images for the coefficients, we prove that the quantum torus Frobenius map is compatible with mutation (Proposition~\ref{pro:compatible}). This yields the Frobenius map for generalized quantum cluster algebras (Theorem~\ref{thm-Fro-cluster}).

\begin{remark}
We show that $T_{\ell,n}^{(k)}$ coincides with the \emph{reduced power elementary polynomial} $\overline{P}_{\ell,n}^{(k)}$ of \cite{bonahon2024central}. A higher-rank analogue of the stated skein algebra is the stated ${\rm SL}_n$-skein algebra \cite{LS21}, defined using oriented loops and graphs whose vertices are either $1$-valent or $n$-valent. It was conjectured in \cite{bonahon2024central} and proved in \cite{kim2025frobenius} that the Frobenius image of a loop is obtained by threading it with the polynomial $\overline{P}_{\ell,n}^{(k)}$.
It is well-known that $\overline{P}_{\ell,2}^{(1)}$ equals the Chebyshev polynomial of the first kind \cite{bonahon2024central};
see \eqref{Chebyshev-polynomial}.
So $T_{\ell,n}^{(k)} = \overline{P}_{\ell,2}^{(k)}$ is a generalization of the first kind Chebyshev polynomial.
\end{remark}

Proposition~\ref{pro:compatible} and Theorem~\ref{thm-Fro-cluster} impose a condition on the order of the root of unity to ensure that the quantum torus Frobenius map is compatible with mutation. For a generic quasi-triangulable punctured surface, this condition requires both $q^2$ and $q^4$ to have order $\ell$. Thus, under assumption~\ref{intro-eq-unity}, Theorem~\ref{intro-thm1} yields the Frobenius map in \eqref{eq-intro-tri}. The construction shows that 
\begin{align}\label{intro-eq-arcs}
\text{$\cF(\alpha)=\alpha^\ell$ whenever $\alpha$ is a cluster variable.}
\end{align}
We next determine $\cF(\alpha)$ for a general stated $\fS$-tangle $\alpha$.

\vs

Let $\alpha$ be a connected (stated) $\fS$-tangle. For each positive
integer $m$, let $\alpha^{[m]}$ be the disjoint union of $m$ parallel
copies of $\alpha$ in the direction of its framing, contained in a
sufficiently small open tubular neighborhood of $\alpha$. Define
\[
\alpha^{(m)}
:=
\left(\prod_{v\in C(\alpha)}v^{\frac{m-1}{2}}\right)\alpha^{[m]},
\]
where $C(\alpha)$ is defined in Definition~\ref{def-arcs}. For a
polynomial $P(x)=\sum_{i=1}^{n}c_i x^i\in\mathbb Z[x]$, set
\[
\alpha^{(P)}:=\sum_{i=1}^{n}c_i\alpha^{(i)}.
\]

Let $T_\ell$ be the $\ell$-th Chebyshev polynomial of the first kind;
see \eqref{Chebyshev-polynomial}. For a framed knot $\alpha_1$, it is
well known that \cite{BL22,BW16}
\begin{align}\label{intro-image1}
\cF(\alpha_1)=\alpha_1^{(T_\ell)}.
\end{align}
Using \eqref{intro-eq-arcs}, we also prove that
\begin{align}\label{intro-image2}
\cF(\alpha_2)=\alpha_2^{(T_\ell)}
\end{align}
for a crossingless arc $\alpha_2$ connecting two distinct 
boundary circles of $\fS$ (Definition~\ref{def-punctured-surface}). The proof proceeds by cutting $\alpha_2$
along an ideal arc; see Proposition~\ref{prop-Fro-arc-T}.

We then show that the Frobenius map in \eqref{eq-intro-tri} is compatible
with the functoriality of Lemma~\ref{lem-functoriality}; see
\eqref{eq-f-Fro-com}. This compatibility, together with
\eqref{intro-eq-arcs}, \eqref{intro-image1}, and \eqref{intro-image2},
gives the following formula. Let
$\beta=
\left(\bigcup_{i=1}^{m}\beta_i\right)
\cup
\left(\bigcup_{j=1}^{k}\gamma_j\right)$
be a stated $\fS$-tangle, where each $\beta_i$ is an essential arc and
each $\gamma_j$ is either a loop or a circle arc
(Definition~\ref{def-arcs}). Then
\begin{align}\label{intro-eq-image-all-diagrams}
\cF(\beta)
=
\left(\bigcup_{i=1}^{m}\beta_i^{(\ell)}\right)
\cup
\left(\bigcup_{j=1}^{k}\gamma_j^{(T_\ell)}\right).
\end{align}
Finally, we use \eqref{intro-eq-image-all-diagrams} and the
functoriality described in Remarks~\ref{rem-blow-map}
and~\ref{rem-adding-puncture} to extend the Frobenius map in
\eqref{eq-intro-tri} to arbitrary punctured surfaces.
Then we obtain our second main result, which establishes the Frobenius map for the reduced L{\^e}--Roger--Yang skein algebra.

\begin{theorem}\label{intro-thm-Fro-main}(Theorem \ref{thm-Fro-main})
Let $\fS$ be a punctured surface, and suppose that $q^2$ is a root
of unity of odd order $\ell$. Set
$\eta^{\frac12}=q^{\frac{\ell^2}{2}}$. Then there is a unique
algebra embedding
\[
\cF\colon\rd_\eta(\fS)\longrightarrow\rdS
\]
satisfying \eqref{intro-eq-image-all-diagrams}. Moreover, $\cF$
has the following properties:
\begin{enumerate}
	\item[\rm (a)] It is compatible with the cutting homomorphism
	(see \S\ref{sec-cutting}); equivalently, the diagram
	\eqref{diag-Fro-cut} commutes.
	
	\item[\rm (b)] If $\fS$ is triangulable with triangulation
	$\lambda$, then it is compatible with the $\mathcal X$-quantum
	trace map (see \S\ref{sec:xqtm}); equivalently, the diagram
	\eqref{diag-Fro-trace} commutes.
\end{enumerate}
\end{theorem}

When $\fS$ has no  boundary circles (Definition~\ref{def-punctured-surface}), the Frobenius map of
Theorem~\ref{intro-thm-Fro-main} was established in \cite{BL22,BW16}.
Extending this result to the reduced L{\^e}--Roger--Yang skein algebra
requires accounting for relation~\eqref{RY-relation}.
Our construction circumvents this difficulty without dealing with the relation directly.

\vs

The Frobenius map for the reduced L{\^e}--Roger--Yang skein algebra
induces a Frobenius map for the L{\^e}--Roger--Yang skein algebra through
an algebra embedding
$ \cS_q(\fS)\hookrightarrow\rd_q(\fS^\ast);$
see \S\ref{sec-Fro-stated}. Here $\fS^\ast$ is obtained from $\fS$ by
attaching ideal triangles along its boundary edges; see
Figure~\ref{fig:attaching-triangle}.

\subsection{Centers and ranks of the Muller--Roger--Yang and reduced L{\^e}--Roger--Yang skein algebras}

Throughout this subsection, assume that $R=\mathbb C$.
Let $\cS_q^+(\fS)$ be the $\widetilde{\mathbb{C}}$-subalgebra (see \eqref{intro-R}) of $\cS_q(\fS)$ generated by positively stated $\fS$-tangles. We call it the \emph{Muller--Roger--Yang skein algebra} \cite{bloomquist2025degenerations}. The canonical projection
$\cS_q(\fS)\longrightarrow \overline{\cS}_q(\fS)$
restricts to an algebra embedding
$\cS_q^+(\fS)\hookrightarrow\overline{\cS}_q(\fS)$, and $\overline{\cS}_q(\fS)$ is obtained from $\cS_q^+(\fS)$ by localizing at positively stated boundary arcs \cite{karuo2025cluster}; see Proposition~\ref{lem-positive-reduced}.

Let $\fS$ be a triangulable punctured surface with an ideal triangulation $\lambda$ (Definition~\ref{def-ideal-arc}). The algebra $\cS_q^+(\fS)$ has a basis parametrized by a submonoid of $\mathbb N^\lambda$ (Lemma~\ref{lem-positive-reduced}). This parametrization is compatible with the $\mathcal X$-quantum trace map and with multiplication, as described in Lemmas~\ref{lem-leading-term} and \ref{lem-leading-deg}, respectively. Under assumption \ref{intro-eq-unity} that $q^2$ is a root of unity of odd order, using these lemmas and the Frobenius map, we determine the center of $\cS_q^+(\fS)$; see Theorem~\ref{thm-center} and Remark~\ref{rem-center}. The center plays a central role in the representation theory of $\cS_q^+(\fS)$.

When $q^2$ is a root of unity of odd order, we show that $\cS_q^+(\fS)$ is finitely generated as a module over its center $\mathcal Z(\cS_q^+(\fS))$; see Lemma~\ref{lem-finite-center}. It follows that $\cS_q^+(\fS)$ is an affine almost Azumaya algebra (Definition~\ref{def-Azu} and Proposition~\ref{prop-Azumaya}). The unicity theorem of \cite{unicity} then classifies its irreducible representations of maximal dimension by the Azumaya locus, a Zariski open dense subset of the maximal ideal spectrum of its center. Their maximal dimension is the square root of the rank of $\cS_q^+(\fS)$ over its center (Definition~\ref{def-rank}). We compute this rank in Theorem~\ref{prop-rank}.

\vs

We now introduce notation for a summary of these results. For an $\fS$-tangle diagram $\alpha$ and a boundary component $e$ of $\fS$, let $e(\alpha)$ be the number of endpoints of $\alpha$ on $e$. Define
\[
\cS^{\rm ev}_q(\fS)
:=
{\rm Span}_{\widetilde{\mathbb{C}}}
\left\{
\alpha \,\middle|\,
\begin{array}{l}
\alpha \text{ is a positively stated }\partial\fS\text{-tangle diagram, and}\\
e(\alpha)\equiv c(\alpha)\pmod 2
\text{ for all boundary components }c,e\text{ of }\fS
\end{array}
\right\},
\]
and set
\[
\im_q\cF
:=
\begin{cases}
\cF\bigl(\cS^+_\eta(\fS)\bigr),
& \text{if the order of $q$ is odd},\\[2pt]
\cF\bigl(\cS^{\rm ev}_\eta(\fS)\bigr),
& \text{if the order of $q$ is even}.
\end{cases}
\]

Let $\mathfrak M$ be the punctured surface obtained from a closed annulus by removing one point from a boundary component, and let $\mathfrak M_1$ be the punctured surface obtained by removing one interior point from a closed annulus. The algebras $\cS_q^+(\mathfrak M)$ and $\cS_q^+(\mathfrak M_1)$ are commutative.

\begin{theorem}[Theorems~\ref{thm-center} and \ref{prop-rank}, and Proposition~\ref{prop-Azumaya}]\label{into-thm4}
Let $\fS\neq\mathfrak M,\mathfrak M_1$ be a triangulable punctured surface, and suppose that $q^2$ is a root of unity of odd order.
\begin{enumerate}
	\item[\rm (a)] The center $\mathcal Z\bigl(\cS_q^+(\fS)\bigr)$, viewed as a $\widetilde{\mathbb{C}}$-subalgebra (see \eqref{intro-R}) of $\cS_q^+(\fS)$, is generated by $\im_q\cF$, the peripheral loops, and the boundary central elements defined in \eqref{boundary-central}.
	
	\item[\rm (b)] The algebra $\cS_q^+(\fS)$ is affine almost Azumaya (Definition~\ref{def-Azu}).
	
	\item[\rm (c)] The rank of $\cS_q^+(\fS)$ over its center (Definition~\ref{def-rank}) is
	\[
	\begin{cases}
		\ell^{\,2a+2b+6g+2d-6},
		& \text{if the order of $q$ is odd},\\[2pt]
		2^{2\left\lfloor\frac{b-1}{2}\right\rfloor}
		\ell^{\,2a+2b+6g+2d-6},
		& \text{if the order of $q$ is even},
	\end{cases}
	\]
	where $a$ is the number of interior punctures, $g$ is the genus of $\fS$, $b$ is the number of connected components of $\partial\fS$, and $d$ is the number of components of $\partial\overline{\fS}$ (Definition~\ref{def-punctured-surface}) containing at least one puncture. The notation $\left\lfloor\frac{b-1}{2}\right\rfloor$ is defined in \eqref{eq-def-lrfloor}.
\end{enumerate}
\end{theorem}

When $\fS$ has no interior punctures and $q$ has odd order, Theorem~\ref{into-thm4}(a) and (b) were established in \cite{karuo2026center}. Their description of the center did not use a Frobenius map, whose existence was not known at the time. Their methods do not extend to the case where $q$ has even order.

\vs

The center and the rank in Theorem~\ref{into-thm4} reflect the center and the rank of $\rdS$ (Remarks~\ref{rem-center-reduced} and \ref{rem-rankkk}).

\medskip

{\bf Acknowledgements:}
This work was partially supported by the National Natural Science Foundation
of China (No.12471023) (M.H.).  
Z.W. was supported by a KIAS Individual Grant (MG104701) at the Korea Institute for Advanced Study. We would like to thank Lang Mou, Helen Wong, and Travis Mandel for inspirit discussions.

\section{(Reduced) L{\^e}--Roger--Yang skein algebras}\label{sec}

In this section, we review the (reduced) Lê--Roger--Yang skein algebra introduced in \cite{bloomquist2025degenerations}, which generalizes the Roger--Yang skein algebra of \cite{roger2014skein} to the stated setting. We also recall the quantum trace and cutting homomorphisms constructed in \cite{bloomquist2025degenerations}, which will be used in subsequent sections.

\subsection{Punctured surfaces and stated tangles}

The (reduced) Lê--Roger--Yang skein algebra is defined for the following punctured surface.

\begin{definition}\label{def-punctured-surface}
A punctured surface is a surface of the form
$\mathfrak{S}=\overline{\fS}\setminus\mathcal{P},$
where $\overline\fS$ is a compact oriented surface with possibly empty boundary
$\partial(\overline\fS)$, and $\mathcal{P}\subset\overline\fS$ is a finite set. An element of
$\mathcal{P}$ is called a \textbf{puncture}, or an \textbf{ideal point}, of
$\mathfrak{S}$. A puncture is called a \textbf{boundary puncture} if it lies in
$\partial(\overline\fS)$, and an \textbf{interior puncture} otherwise. The set of interior punctures is denoted by $\mathring{\mathcal P}$.

Each connected
component of $\partial(\overline\fS)\setminus\mathcal{P}$ that is homeomorphic to an
open interval (resp. a circle) is called a \textbf{boundary edge} (resp. {\bf boundary circle}) of $\mathfrak{S}$.
\end{definition}

Note that it is not required for each connected component of $\partial(\overline{\fS})$ to contain at least one puncture.

\begin{definition}\label{def-ideal-arc}
An \textbf{ideal arc} in $\fS$ is a smooth embedding
$e\colon (0,1)\hookrightarrow \fS$
that extends to an immersion
$\overline{e}\colon [0,1]\to \fS$
such that $\overline{e}(0),\overline{e}(1)\in\mathcal{P}$. An ideal arc is called \textbf{trivial} if the image of $\overline{e}$ bounds a disk in $\fS$. 

\begin{enumerate}
	\item[(a)] An $n$-gon is the surface obtained by removing $n$ points from the boundary of the standard oriented closed disk.
	
	\item[(b)] A punctured surface $\fS$ is called \textbf{triangulable} if every connected component of $\overline\fS$ contains at least one puncture and is not one of the following exceptional surfaces:
	\begin{enumerate}
		\item[(1)] a sphere with at most two punctures;
		\item[(2)] an $n$-gon with $n=1$ or $n=2$.
	\end{enumerate}
	
	\item[(c)] An \textbf{ideal triangulation} of a triangulable surface $\fS$ is a maximal collection $\lambda$ of pairwise disjoint, pairwise non-isotopic, non-trivial ideal arcs.
\end{enumerate}

\end{definition}

Define
\[
\widetilde{\fS}:=\fS\times(-1,1),
\]
equipped with the product orientation determined by the orientation of $\fS$
and the standard orientation of $(-1,1)$.

The (reduced) Lê--Roger--Yang skein algebra of $\fS$ is defined by stated $\fS$-tangles.

\begin{definition}[\cite{bloomquist2025degenerations}]\label{def-tangle}
\begin{enumerate}
	\item[(a)] A \textbf{$ \fS$-tangle} is a compact, unoriented, properly embedded $1$-dimensional submanifold
	\[
	\alpha\hookrightarrow \fS\times(-1,1),
	\]
	equipped with a normal vector field, called its \textbf{framing}, satisfying the following conditions:
	\begin{itemize}
		\item at every boundary point of $\alpha$, the framing vector is vertical;
		\item any two boundary points of $\alpha$ lying on the same connected component of $\partial \fS\times(-1,1)$ have distinct heights.
	\end{itemize}
	The set of endpoints of $\alpha$ is denoted by $\partial\alpha$.
	
	\item[(b)] The \textbf{height order} on $\partial\alpha$ is the partial order obtained by comparing the heights of endpoints lying on the same connected component of $\partial \fS\times(-1,1)$.
	
	\item[(c)] Two endpoints of $\alpha$ are called \textbf{consecutive} if no endpoint of $\alpha$ lies strictly between them with respect to the height order.
	
	\item[(d)] By convention, the empty set is regarded as a $\fS$-tangle.
	
	\item[(e)] Isotopies of $\fS$-tangles are required to remain within the class of $ \fS$-tangles throughout the isotopy.
\end{enumerate}
\end{definition}

As usual, $\fS$-tangles are represented by diagrams on $\fS$, as described below. Every $\fS$-tangle is isotopic to one with vertical framing. Let $\alpha$ be a vertically framed $\fS$-tangle in general position with respect to the standard projection
$\pi\colon \fS\times(-1,1)\longrightarrow \fS.$
Thus, the restriction $\pi|_\alpha\colon\alpha\to\fS$ is an immersion whose only singularities are transverse double points, none of which lies on $\partial\fS$. The image
$D=\pi(\alpha),$
together with the overcrossing and undercrossing information at each double point, is called a \textbf{$\fS$-tangle diagram}, or simply a \textbf{tangle diagram on $\fS$}.

\begin{definition}\label{def-CCCC}
Let $\alpha$ be an $\fS$-tangle represented by a diagram $D(\alpha)$. We define $C(\alpha)$ to be the multiset of boundary circles of $\fS$ on which the endpoints of $D(\alpha)$ lie. 
\end{definition}

A tangle diagram $D$ is said to be \textbf{boundary ordered} if, for every boundary component $b$ of $\fS$, the set $D\cap b$ is equipped with a linear order. Isotopies of $\fS$-tangle diagrams are ambient isotopies in $\fS$. A diagram $D$ of a $\fS$-tangle $\alpha$ inherits a boundary order from the height order of the endpoints of $\alpha$. The boundary ordered diagram $D$ uniquely determines the isotopy class of $\alpha$. Therefore, when no confusion can arise, we identify a boundary ordered $\fS$-tangle diagram with the corresponding isotopy class of $\fS$-tangles.

A \textbf{marked orientation} $o$ consists of an orientation of $\partial\fS$ together with a distinguished point, called the \textbf{initial point}, on each circular component of $\partial\fS$. A $\fS$-tangle diagram $D$ is said to have the \textbf{$o$-order} if its boundary order agrees with the direction prescribed by $o$. More precisely, on each boundary component $b$, the height order increases as one travels along $b$ in the direction of $o$, starting from the initial point when $b$ is a circle. If the orientation of every boundary component is induced by the orientation of $\fS$, then the corresponding $o$-order is called the \textbf{positive order}.

\begin{definition}\label{def-crossingless}
A (boundary ordered) $\fS$-tangle diagram $\alpha$ is called
\textbf{crossingless} if it has no crossings in the interior of
$\fS$ and has at most one endpoint on each boundary circle of $\fS$.
\end{definition}

\begin{definition}\label{def-arcs}

A connected $\fS$-tangle is called a \textbf{loop} if it is
diffeomorphic to $S^1$, and an \textbf{arc} if it is diffeomorphic to
$[0,1]$.
It is called {\bf simple} if it is represented by a crossingless diagram. 

An arc $\alpha$ is called an \textbf{edge arc} if both endpoints of its diagram lie on boundary edges, and an \textbf{edge-circle arc} if one endpoint lies on a boundary edge and the other on a boundary circle. It is called a \textbf{circle arc} if both endpoints lie on boundary circles, which are not necessarily distinct.

An edge arc or an edge--circle arc is called an \textbf{essential arc}.

\end{definition}

\begin{definition}\label{def-state}
For a $\fS$-tangle $\alpha$, or a $\fS$-tangle diagram $\alpha$, define the \textbf{wall boundary} $\partial_w(\alpha)$ to be the set of all endpoints of $\alpha$ lying on the boundary edges of $\fS$. A \textbf{state} of $\alpha$ is a map
$s\colon \partial_w(\alpha)\longrightarrow \{+,-\}.$
Thus, no states are assigned to the endpoints of $\alpha$ lying on the boundary circles of $\fS$.
\end{definition}

\begin{definition}
A state on a boundary ordered $\partial\mathfrak S$-tangle diagram
$\alpha$ is \textbf{increasing} if the state function
$s:\partial_{ w}(\alpha)\longrightarrow\{-,+\}$
is increasing with respect to the height order. That is, $x\leq y$
implies $s(x)\leq s(y)$, where $-<+$.
\end{definition}

\subsection{(Reduced) L{\^e}-Roger-Yang skein algebras}

Let $R$ be a commutative domain with a distinguished invertible element $q^{\frac{1}{2}}$.
We use $V$ to denote the set of all the boundary circles of $\fS$.  Define 
\begin{align}\label{eq-Rtilde}
\widetilde{R}:= R[v^{\pm \frac{1}{2}}\mid v\in V].
\end{align}

We use 
$\begin{tikzpicture}
\filldraw[draw=black,fill=white] (0.4,-0.1) rectangle (0.6, 0.1);
\end{tikzpicture}$ and $\begin{tikzpicture}
\draw[draw=black,fill=white]  (0.5,0) circle (0.1);
\end{tikzpicture}$ to represent the punctures and boundary circles respectively.

The \textbf{L{\^e}-Roger-Yang skein algebra} (or \textbf{LRY skein algebra}) $\cS_{q}(\fS)$ of $\fS$ is
the quotient module of the $\widetilde R$-module freely generated by the set 
of all isotopy classes of stated 
$\fS$-tangles subject to  relations \eqref{cross}-\eqref{unknot-RY} \cite{bloomquist2025degenerations}. 
\begin{equation}\label{cross}
\raisebox{-.20in}{
	\begin{tikzpicture}
		\filldraw[draw=white,fill=gray!20] (-0,-0.2) rectangle (1, 1.2);
		\draw [line width =1pt](0.6,0.6)--(1,1);
		\draw [line width =1pt](0.6,0.4)--(1,0);
		\draw[line width =1pt] (0,0)--(0.4,0.4);
		\draw[line width =1pt] (0,1)--(0.4,0.6);
		\draw[line width =1pt] (0.6,0.6)--(0.4,0.4);
	\end{tikzpicture}
}=
q
\raisebox{-.20in}{
	\begin{tikzpicture}
		\filldraw[draw=white,fill=gray!20] (-0,-0.2) rectangle (1, 1.2);
		\draw[line width =1pt] plot[smooth] coordinates {(0,1) (0.4,0.5) (0,0)};
		\draw[line width =1pt] plot[smooth] coordinates {(1,1) (0.6,0.5) (1,0)};
	\end{tikzpicture}
}
+
q^{-1}
\raisebox{-.20in}{
	\begin{tikzpicture}
		\filldraw[draw=white,fill=gray!20] (-0,-0.2) rectangle (1, 1.2);
		\draw[line width =1pt] plot[smooth] coordinates {(0,1) (0.5,0.6) (1,1)};
		\draw[line width =1pt] plot[smooth] coordinates {(0,0) (0.5,0.4) (1,0)};
	\end{tikzpicture}
}, 
\end{equation}
\begin{equation}\label{unknot}
\raisebox{-.15in}{
	\begin{tikzpicture}
		\filldraw[draw=white,fill=gray!20] (-0,-0) rectangle (1, 1);
		\draw [line width =1pt] (0.5,0.5) circle (0.3);
	\end{tikzpicture}
}=-(q^2+q^{-2})
\raisebox{-.15in}{
	\begin{tikzpicture}
		\filldraw[draw=white,fill=gray!20] (-0,-0) rectangle (1, 1);
	\end{tikzpicture}
},
\end{equation}
\begin{equation}\label{arc}
\raisebox{-.25in}{
	\begin{tikzpicture}
		\filldraw[draw=white,fill=gray!20] (-0.2,-0) rectangle (1.2, 0.6);
		\draw [line width =1.5pt,decoration={markings, mark=at position 1 with {\arrow{>}}},postaction={decorate}](-0.2,-0)--(1.2,0);
		\draw [line width =1pt]  (0.8 ,0) arc (0:180:0.4 and 0.4);
		\node[below] at (0,0) {\small $-$};
		\node [below]  at (0.8,0) {\small $+$};
	\end{tikzpicture}
}=q^{-\frac{1}{2}}
\raisebox{-.15in}{
	\begin{tikzpicture}
		\filldraw[draw=white,fill=gray!20] (-0.2,-0) rectangle (1.2, 0.6);
		\draw [line width =1.5pt,decoration={markings, mark=at position 1 with {\arrow{>}}},postaction={decorate}](-0.2,-0)--(1.2,0);
	\end{tikzpicture}
},\;\;
\raisebox{-.25in}{
	\begin{tikzpicture}
		\filldraw[draw=white,fill=gray!20] (-0.2,-0) rectangle (1.2, 0.6);
		\draw [line width =1.5pt,decoration={markings, mark=at position 1 with {\arrow{>}}},postaction={decorate}](-0.2,-0)--(1.2,0);
		\draw [line width =1pt]  (0.8 ,0) arc (0:180:0.4 and 0.4);
		\node[below] at (0,0) {\small $+$};
		\node [below]  at (0.8,0) {\small $+$};
	\end{tikzpicture}
}=\raisebox{-.25in}{
	\begin{tikzpicture}
		\filldraw[draw=white,fill=gray!20] (-0.2,-0) rectangle (1.2, 0.6);
		\draw [line width =1.5pt,decoration={markings, mark=at position 1 with {\arrow{>}}},postaction={decorate}](-0.2,-0)--(1.2,0);
		\draw [line width =1pt]  (0.8 ,0) arc (0:180:0.4 and 0.4);
		\node[below] at (0,0) {\small $-$};
		\node [below]  at (0.8,0) {\small $-$};
	\end{tikzpicture}
} =0,
\end{equation}
\begin{equation}\label{hight}
\raisebox{-.23in}{
	\begin{tikzpicture}
		\filldraw[draw=white,fill=gray!20] (-0.2,-0) rectangle (1.2, 0.6);
		\draw [line width =1.5pt,decoration={markings, mark=at position 1 with {\arrow{>}}},postaction={decorate}](-0.2,-0)--(1.2,0);
		\draw [line width =1pt](0.1,0.6)--(0.1,0);
		\draw [line width =1pt](0.9,0.6)--(0.9,0);
		\node [below] at (0.1,0) {$+$};
		\node [below] at (0.9,0) {\small $-$};
	\end{tikzpicture}
}=q^{2}
\raisebox{-.23in}{
	\begin{tikzpicture}
		\filldraw[draw=white,fill=gray!20] (-0.2,-0) rectangle (1.2, 0.6);
		\draw [line width =1.5pt,decoration={markings, mark=at position 1 with {\arrow{>}}},postaction={decorate}](-0.2,-0)--(1.2,0);
		\draw [line width =1pt](0.1,0.6)--(0.1,0);
		\draw [line width =1pt](0.9,0.6)--(0.9,0);
		\node [below] at (0.1,0) {$-$};
		\node [below] at (0.9,0) {\small $+$};
	\end{tikzpicture}
} + q^{-\frac{1}{2}}
\raisebox{-.16in}{
	\begin{tikzpicture}
		\filldraw[draw=white,fill=gray!20] (-0.2,-0) rectangle (1.2, 0.6);
		\draw [line width =1.5pt,decoration={markings, mark=at position 1 with {\arrow{>}}},postaction={decorate}](-0.2,-0)--(1.2,0);
		\draw[line width =1pt] plot[smooth] coordinates {(0.1,0.6) (0.5, 0.15) (0.9,0.6)};
	\end{tikzpicture}
},
\end{equation}
\begin{equation}\label{RY-relation}
\raisebox{-.16in}{
	\begin{tikzpicture}
		\filldraw[draw=white,fill=gray!20] (-0.2,-0) rectangle (1.2, 1);
		\draw [line width =1pt](0,1)--(0.35,0.65);
		\draw [line width =1pt](1,1)--(0.5,0.5);
		\filldraw[draw=black,fill=white,thick]  (0.5,0.5) circle (0.2);
		\draw[->, thick]
		(0.5,0.7)
		arc[start angle=90, end angle=-200, radius=0.2];
		\filldraw[draw=black]  (0.7,0.5) circle (0.05);
	\end{tikzpicture}
}=v^{-1} \left( q^{-\frac{1}{2}}
\raisebox{-.16in}{
	\begin{tikzpicture}
		\filldraw[draw=white,fill=gray!20] (-0.2,-0) rectangle (1.2, 1);
		\draw[line width =1pt] plot[smooth] coordinates {(0,1)  (0.5, 0.75) (1,1)};
		\filldraw[draw=black,fill=white]  (0.5,0.5) circle (0.1);
	\end{tikzpicture}
} + q^{\frac{1}{2}}
\raisebox{-.16in}{
	\begin{tikzpicture}
		\filldraw[draw=white,fill=gray!20] (-0.2,-0) rectangle (1.2, 1);
		\draw[line width =1pt] plot[smooth] coordinates {(0,1)  (0.5, 0.15) (1,1)};
		\filldraw[draw=black,fill=white]  (0.5,0.5) circle (0.1);
	\end{tikzpicture}
} \right) \quad \text{the boundary circle is $v$,}
\end{equation}
\begin{equation}\label{unknot-RY}
\raisebox{-.15in}{
	\begin{tikzpicture}
		\filldraw[draw=white,fill=gray!20] (-0,-0) rectangle (1, 1);
		\draw [line width =1pt] (0.5,0.5) circle (0.3);
		\filldraw[draw=black,fill=white]  (0.5,0.5) circle (0.1);
	\end{tikzpicture}
}=q +q^{-1}
\raisebox{-.15in}{
	\begin{tikzpicture}
		\filldraw[draw=white,fill=gray!20] (-0,-0) rectangle (1, 1);
		\filldraw[draw=black,fill=white]  (0.5,0.5) circle (0.1);
	\end{tikzpicture}
},
\end{equation}
where multiplication is defined by stacking with respect to $(-1,1)$, i.e., for any two stated $\fS$-tangles $\alpha,\alpha'
\subset\fS\times(-1,1)$, the product $\alpha\alpha'$ is defined by stacking $\alpha$ above $\alpha'$. That is, if $\alpha \subset \fS\times(0,1)$ and $\alpha' \subset \fS\times (-1,0)$, we have $\alpha \alpha' = \alpha \cup \alpha'$.
Note that the definition of $\cS_{q}(\fS)$ involves $q^{\frac{1}{2}}$. 
For simplicity, however, we write $\cS_{q}(\fS)$ instead of $\cS_{q^{\frac{1}{2}}}(\fS)$.

The figures in this paper are interpreted according to the following convention. Each shaded region $S\subset\fS$ is a closed square, except in Relations~\eqref{RY-relation} and~\eqref{unknot-RY}, where $S$ is a closed square with an open disk removed. The diagram shown in $S$ is the projection of
$\alpha\cap(S\times J),$
where $J\subset(-1,1)$ is a closed interval such that
$\alpha\cap(S\times\partial J)=\varnothing.$
In Relations~\eqref{arc} and~\eqref{hight}, each vertical line represents a portion of $\partial\fS$. The endpoints lying on this line are consecutive in the height order, which increases in the indicated direction.

\def\wR{\widetilde{R}}

Let $B(\fS)$ denote the set of increasingly stated, positively ordered crossingless $\fS$-tangle diagrams. 
\cite[Theorem~3.6]{bloomquist2025degenerations} shows that
$\cS_q(\fS)$ is a free $\wR$-module with basis $B(\fS)$.

\def\OB{\overline{B}(\fS)}
\def\WB{\widetilde{B}(\fS)}

\vs

We denote by $I$ the two-sided ideal of $\SS$ generated by all bad arcs (see Figure~\ref{fig:badarc}).
The \textbf{reduced L{\^e}–Roger–Yang skein algebra} (or simply the \textbf{reduced LRY skein algebra}) of $\fS$ is defined as
\begin{align*}
\rdS \;:=\; \cS_q(\fS) \big/ I.
\end{align*}

\begin{figure}[h]
\centering
\raisebox{-.15in}{
\begin{tikzpicture}
\filldraw[draw=white,fill=gray!20] (-1.5,-0) rectangle (1.5, 1);
\draw [line width =1.5pt](-1.5,-0)--(1.5,0);
\draw [line width =1pt]  (1,0) arc (0:180:1 and 0.7);
\filldraw[draw=black,fill=white] (-0.1,-0.1) rectangle (0.1, 0.1);
\node at (-1,-0.2) {\small $+$};
\node at (1,-0.2) {\small $-$};
\end{tikzpicture}
}
\caption{A bad arc.}
\label{fig:badarc}
\end{figure}
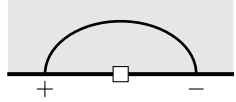

For an element $X\in \cS_q(\fS)$, we will, by abuse of notation, continue to write $X$ for its image under the projection $\cS_q(\fS)\twoheadrightarrow \rdS$.

A simple loop in $\fS$ that surrounds an interior puncture is called a 
\textbf{peripheral loop}. Define
\begin{align}
\overline{B}(\fS)&:=\{\,D\in B(\fS)\mid D \text{ contains no bad arcs}\,\}\label{def-OB},\\
\WB&:=\{\,D\in \OB\mid D \text{ contains no peripheral loops}\,\}\label{def-WB}.
\end{align}

\begin{theorem}[{\cite[Theorem~3.11]{bloomquist2025degenerations}}]\label{thm-basis-reduced}
$\overline{\cS}_q(\fS)$ is a free $R$-module with basis $\overline{B}(\fS)$. 
\end{theorem}

The following lemma follows directly from \eqref{hight}.  

\begin{lemma}\label{lem-m-arc}
In $\rdS$, we have the following relation 
$$
\begin{array}{c}
	\includegraphics[scale=0.3]{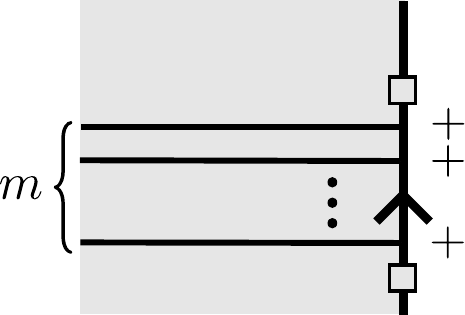} 
\end{array}
= q^{2m} q^{\frac{m(m-1)}{2}}
\begin{array}{c}
	\includegraphics[scale=0.3]{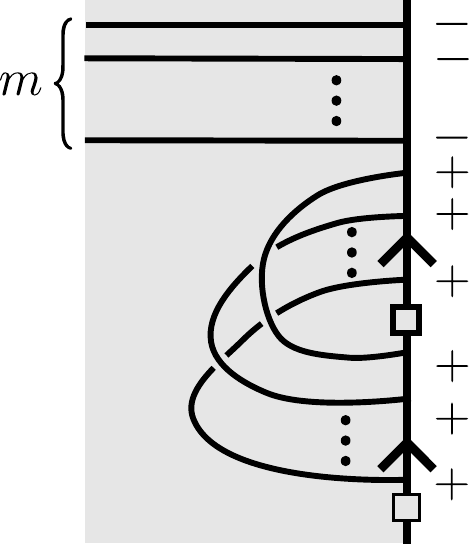}
\end{array}.
$$
\end{lemma}

\def\Cut{{\rm Cut}}

\subsection{Cutting homomorphism}\label{sec-cutting}
Let $c$ be an ideal arc contained in the interior of a punctured
surface $\fS$. Cutting $\fS$ along $c$ produces a punctured
surface $\Cut_c\fS$ with two boundary edges $c_1$ and $c_2$
such that
\[
\fS=\Cut_c\fS/(c_1=c_2),
\qquad c=c_1=c_2.
\]
Suppose that a tangle diagram $D$ is transverse to
$c$, and let $h$ be a linear order on $D\cap c$. Denote by
$p\colon \Cut_c\fS\longrightarrow \fS$
the natural projection. For each map
$s\colon D\cap c\longrightarrow\{+,-\},$
let $(D,h,s)$ denote the stated tangle diagram over
$\Cut_c\fS$ whose underlying diagram is $p^{-1}(D)$, and whose
height orders and states along $c_1\cup c_2$ are induced, via $p$,
by $h$ and $s$, respectively.

\begin{theorem}[\cite{bloomquist2025degenerations}]\label{thm-splitting-map}
Let $c$ be an interior ideal arc of a punctured surface $\fS$.
\begin{enumerate}
	\item[\textup{(a)}]
	There exists a unique $\wR$-algebra embedding
	\[
	\Theta_c\colon
	\cS_q(\fS)\hookrightarrow \cS_q(\Cut_c\fS)
	\]
	with the following property. If $D$ is a stated tangle diagram
	representing $\alpha\in\cS_q(\fS)$, transverse to $c$, and
	$h$ is any linear order on $D\cap c$, then
	\begin{equation}\label{eq-splitting-map}
		\Theta_c(\alpha)
		=
		\sum_{s\colon D\cap c\to\{+,-\}}(D,h,s).
	\end{equation}
	
	\item[\textup{(b)}]
	The homomorphism $\Theta_c$ descends to an algebra embedding
	between the reduced LRY skein algebras:
	\[
	\overline{\Theta}_c\colon
	\overline{\cS}_q(\fS)
	\hookrightarrow
	\overline{\cS}_q(\Cut_c\fS).
	\]
\end{enumerate}
\end{theorem}

\subsection{Reflections}\label{sub-reflection}

Let $A$ be a $\mathbb Z[q^{\pm\frac{1}{2}}]$-algebra that is torsion-free as a
$\mathbb Z[q^{\pm\frac{1}{2}}]$-module, and let
$x_1,\ldots,x_n\in A$ be pairwise $q$-commuting elements. Since $A$ is torsion-free over
$\mathbb Z[q^{\pm\frac{1}{2}}]$, whenever $x_ix_j\neq 0$, there is a unique integer $c_{ij}$ such that
$x_ix_j=q^{c_{ij}}x_jx_i.$
The \textbf{Weyl normalization} of the product $x_1x_2\cdots x_n$ is defined by
\begin{align}\label{eq-weyl}
[x_1x_2\cdots x_n]_{\mathrm{norm}}
:=
q^{-\frac12\sum_{i<j}c_{ij}}x_1x_2\cdots x_n.
\end{align}

If $R$ is a general ground ring, that is, a commutative
$\mathbb Z[q^{\pm\frac{1}{2}}]$-domain, we also use the notation
$[ab]_{\mathrm{norm}}$ for the element
\begin{align}\label{eq-weyl1}
[ab]_{\mathrm{norm}}\otimes 1
\in A\otimes_{\mathbb Z[q^{\pm\frac{1}{2}}]}R.
\end{align}

A $\mathbb Z[q^{\pm\frac{1}{2}}]$-algebra \textbf{with reflection} is an algebra $A$ equipped with a
$\mathbb Z$-linear anti-involution
${\bf R}\colon A \rightarrow A$,
called the \textbf{reflection}, satisfying
${\bf R}(q^{1/2})=q^{-1/2}$. Equivalently, for all $x,y\in A$,
\[
{\bf R}(xy)={\bf R}(y){\bf R}(x),
\qquad
{\bf R}(q^{1/2}x)=q^{-1/2}{\bf R}(x).
\]
An element $z\in A$ is called \textbf{reflection invariant} if
${\bf R}(z)=z$.

A nonzero element $z\in A$ is called \textbf{reflection normalizable} if
${\bf R}(z)=q^m z$
for some $m\in\mathbb Z$. Since $A$ is torsion-free over
$\mathbb Z[q^{\pm\frac12}]$, the integer $m$ is uniquely determined. We define the
\textbf{reflection normalization} of $z$ by
\begin{align}\label{eq-norm}
[z]_{\mathrm{norm}}:=q^{\frac{m}{2}}z.
\end{align}
This notation agrees with the Weyl normalization defined in
\eqref{eq-weyl}. In particular, $[z]_{\mathrm{norm}}$ is reflection invariant.
As in \eqref{eq-weyl1}, if $R$ is a general ground ring, we also use
$[z]_{\mathrm{norm}}$ to denote the element
\begin{align}\label{eq-norm1}
[z]_{\mathrm{norm}}\otimes 1
\in A\otimes_{\mathbb Z[q^{\pm\frac12}]}R.
\end{align}

The algebra $\cS_q(\fS)$ admits a natural reflection ${\bf R}$. 
For each $v\in V$, define ${\bf R}(v^{\pm \frac{1}{2}})=v^{\pm \frac{1}{2}}$.
For any stated, boundary-ordered $\fS$-tangle diagram $D$, the diagram ${\bf R}(D)$ is obtained by switching every crossing and reversing the height order of the endpoints on each boundary edge \cite{bloomquist2025degenerations}.

It is straightforward to verify that ${\bf R}$ descends to a reflection for $\rdS$, which is still denoted by ${\bf R}$.

\subsection{Functoriality}

The following lemma follows from the proof of \cite[Theorem 7.1]{CLL} (see \cite[Figures 27 and 29]{CLL}).

\begin{lemma}\label{lem-reduced-module}
Let $\fS$ be punctured surface.
Then $\overline{\cS}_{q}(\fS)$ of $\fS$ is
the quotient module of the $\wR$-module freely generated by the set 
of all isotopy classes of stated 
$\fS$-tangles subject to  relations \eqref{cross}-\eqref{unknot-RY}, and 
\begin{align}\label{lem-bad-arc-zero}
	\text{bad arc}=0.
\end{align}

\end{lemma}

The following straightforward lemma will be used later

\begin{lemma}\label{lem-functoriality}
Let $\fS_1$ and $\fS_2$ be punctured surfaces, and let
\[
f\colon \widetilde{\fS}_1\longrightarrow \widetilde{\fS}_2
\]
be an orientation-preserving embedding. Suppose that, for every boundary
component $b_1$ of $\overline{\fS}_1$, there exist a boundary component
$b_2$ of $\overline{\fS}_2$ and open regular neighborhoods $U_1$ and $U_2$
of $b_1\times (-1,1)$ and $b_2\times (-1,1)$, respectively, such that
\[
f\bigl(U_1\cap\widetilde{\fS}_1\bigr)
\subset
U_2\cap\widetilde{\fS}_2.
\]
If $b_1$ contains at least one puncture, assume further that the restriction
\[
f\big|_{U_1\cap\widetilde{\fS}_1}\colon
U_1\cap\widetilde{\fS}_1
\longrightarrow
U_2\cap\widetilde{\fS}_2
\]
is a diffeomorphism. Then $f$ induces an $R[v^{\frac{1}{2}}\mid v\in V_1]$-module homomorphism
\[
f_*\colon
\overline{\cS}_q(\fS_1)
\longrightarrow
\overline{\cS}_q(\fS_2),
\]
where $V_1$ is the set of boundary circles of $\fS_1$.
\end{lemma}

\begin{proof}
It is straightforward to check that $f$ preserves 
relations \eqref{cross}-\eqref{unknot-RY}, and \eqref{lem-bad-arc-zero}.
\end{proof}

\begin{remark}\label{rem-blow-map}
Let $\fS$ be a punctured surface with an interior puncture $p$, and let
$c_p$ be a trivial ideal arc based at $p$. Thus, both endpoints of
$c_p$ are equal to $p$, and $c_p$ bounds an embedded monogon $D_p$; see
Figure~\ref{fig_cp}. Removing the interior of $D_p$ from $\fS$ produces
a new punctured surface $\fS'$ in which $c_p$ becomes a boundary edge.

After identifying
$\fS=\fS'\setminus c_p,$
there is, up to isotopy, a unique embedding
$f\colon\fS\hookrightarrow\fS'$
that agrees with the identity outside a small neighborhood of $D_p$.
Intuitively, the embedding $f$ enlarges the puncture $p$ into a boundary
edge.

As in Lemma~\ref{lem-functoriality}, the embedding $f$ induces an
$\wR$-module homomorphism
\[
f_*\colon\rdS\longrightarrow\rd_q(\fS').
\]
It is straightforward to verify that $f_*$ preserves multiplication and
maps $\overline{B}(\fS)$ injectively into $\overline{B}(\fS')$.
Consequently, $f_*$ is an algebra embedding.
\end{remark}

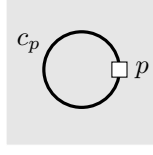
\begin{figure}
\centering
\begin{tikzpicture}
\draw[draw=white,fill=gray!20]  (-1,-1) rectangle (1,1);
    \draw[wall,fill=gray!20] (0,0) circle[radius=0.5];
    \draw[fill=white] (0.4,-0.1) rectangle (0.6,0.1);
    \node at (0.8,0) {$p$};
    \node [above] at (-0.7,0.1) {$c_p$};
\end{tikzpicture}
\caption{The arc $c_p$ and the enclosed monogon $D_p$.}\label{fig_cp}
\end{figure}

\begin{remark}\label{rem-adding-puncture}
Let $\fS$ be a punctured surface, and let $p$ be an interior point of
$\fS$. Set
$\fS'=\fS\setminus\{p\}.$
The natural embedding $\fS'\hookrightarrow\fS$ induces a surjective
$\wR$-algebra homomorphism
\[
\mathbf{P}\colon\rd_q(\fS')\longrightarrow\rdS.
\]
Its kernel is generated by elements of the form $a-a'$, where $a$ and
$a'$ are stated $\fS'$-tangle diagrams related by sliding a
strand across the puncture $p$; see Figure~\ref{eq-relation-point-p}.

\end{remark}

\begin{figure}
\centering
\raisebox{-.35\height}{
\begin{tikzpicture}
\draw[draw=white,fill=gray!20]  (-1,-1) rectangle (1,1);
    \draw[fill=white] (-0.1,-0.1) rectangle (0.1,0.1);
    \node [right] at (0.1,0) {$p$};
    \node [left] at (-1,0) {$a=$};
    \draw (1,0) parabola bend (0,-0.5) (-1,0);
\end{tikzpicture}}
\;\;\;\;\;\;\;\;\;\;
\raisebox{-.35\height}{
\begin{tikzpicture}
\draw[draw=white,fill=gray!20]  (-1,-1) rectangle (1,1);
    \draw[fill=white] (-0.1,-0.1) rectangle (0.1,0.1);
    \node [right] at (0.1,0) {$p$};
    \node [left] at (-1,0) {$a'=$};
    \draw (1,0) parabola bend (0,0.5) (-1,0);
\end{tikzpicture}}
\caption{$n$-web diagrams $a$ and $a'$ in  
	$\Sigma_{g,1}$ related by sliding over the puncture $p$}
\label{eq-relation-point-p}
\end{figure}
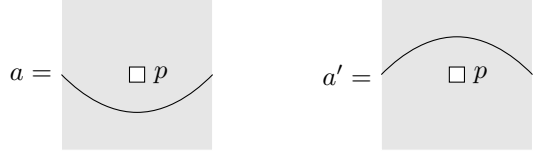

\subsection{$\mathcal A$-quantum trace maps}\label{sec-A-quantum-trace}

In this section, we review the $\mathcal A$-quantum trace map constructed in \cite{bloomquist2025degenerations} and extend this construction to a more general setting; see Proposition~\ref{prop:quantumtrace}. In \S\ref{sec-skein=A=U}, we use the quantum trace map from Proposition~\ref{prop:quantumtrace} to construct a generalized quantum cluster algebra structure on the reduced LRY skein algebra; see \S\ref{subsec-generalized}.

\begin{definition}\label{def-quasi}
\begin{enumerate}
	\item[(a)] An \textbf{admissible ideal arc} in $\fS$ is either an ideal arc connecting two boundary punctures or an embedding
	$a\colon [0,1)\hookrightarrow \fS$
	that extends to a proper embedding
	$\overline{a}\colon [0,1]\hookrightarrow \overline{\fS}$
	such that $\overline{a}(1)$ is a boundary puncture and $a(0)$ lies on a boundary circle of $\fS$. In the latter case, we call $a$ a \textbf{pending arc}. For $\epsilon\in(0,\tfrac{1}{2})$ sufficiently small, the subpaths $\alpha([0,\epsilon])$ and $\alpha([1-\epsilon,1])$ are called the {\bf half-edges} of $\alpha$.
	\item[(b)] An \textbf{admissible ideal multi-arc} is a disjoint union of admissible ideal arcs such that it intersects each boundary circle at most once. Given an admissible ideal multi-arc $\alpha$, the $\fS$-tangle diagram $D(\alpha)$ is obtained by moving all branches of $\alpha$ near each ideal point to the left, as illustrated in Figure \ref{fig:move-left}.
\end{enumerate}
\end{definition}

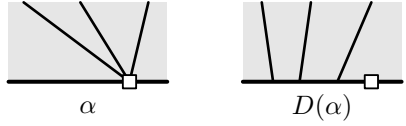
\begin{figure}[ht]
\centering
\begin{tikzpicture}[
	line cap=round,
	line join=round,
	arc/.style={line width=1pt},
	boundary/.style={line width=1.5pt},
	ideal point/.style={
		draw,
		fill=white,
		line width=0.7pt,
		minimum size=5pt,
		inner sep=0pt
	}
	]
	
	\begin{scope}
		\fill[gray!20] (-1.05,0) rectangle (1.05,1.05);
		
		\draw[boundary] (-1.05,0) -- (1.05,0);
		
		\draw[arc] (-0.85,1.05) -- (0.55,0);
		\draw[arc] (-0.10,1.05) -- (0.55,0);
		\draw[arc] ( 0.80,1.05) -- (0.55,0);
		
		\node[ideal point] at (0.55,0) {};
		
		\node at (0,-0.34) {$\alpha$};
	\end{scope}
	
	\begin{scope}[xshift=3.1cm]
		\fill[gray!20] (-1.05,0) rectangle (1.05,1.05);
		
		\draw[boundary] (-1.05,0) -- (1.05,0);
		
		\draw[arc] (-0.80,1.05) -- (-0.65,0);
		\draw[arc] (-0.15,1.05) -- (-0.30,0);
		\draw[arc] ( 0.65,1.05) -- ( 0.20,0);
		
		\node[ideal point] at (0.65,0) {};
		
		
		\node at (0,-0.34) {$D(\alpha)$};
	\end{scope}
	
\end{tikzpicture}

\caption{Moving left: from the admissible ideal multiarc
	$\alpha$ to $D(\alpha)$. If a component of $\alpha$ is a
	boundary edge, we first isotope it slightly into the interior and
	then apply the move-left operation.}
\label{fig:move-left}
\end{figure}

For each boundary puncture $p$, suppose that the endpoints of $\alpha$ incident to $p$ are equipped with a linear order. These orders determine a boundary ordered $\fS$-tangle diagram $D(\alpha)$. If, in addition, each endpoint of $\alpha$ incident to a boundary puncture is assigned a state in $\{-,+\}$, then $D(\alpha)$ represents a stated $\fS$-tangle. In this case, we identify $\alpha$ with the corresponding stated $\fS$-tangle.
We use $D_+(\alpha)\in \WB$ (see \eqref{def-WB}) to denote the positively stated and positive ordered $\fS$-tangle diagram obtained from $\alpha$. 

When only two endpoints of $\alpha$ incident to a boundary puncture are depicted, we indicate their linear order by breaking one of the strands, as in
\raisebox{-.16in}{
\begin{tikzpicture}
	\fill[gray!20] (-0.2,0) rectangle (1.2,1);
	\draw[line width=1.5pt] (-0.2,0)--(1.2,0);
	\draw[line width=1pt] (0,1)--(0.4,0.2);
	\draw[line width=1pt] (1,1)--(0.6,0.2);
	\draw[line width=1pt] (0.5,0)--(0.6,0.2);
	\filldraw[draw=black,fill=white]
	(0.4,-0.1) rectangle (0.6,0.1);
\end{tikzpicture}
}.
Here the broken strand is immediately below the unbroken strand in the height order.

If no states are explicitly assigned to the endpoints of $\alpha$, then, whenever $\alpha$ is regarded as a stated $\fS$-tangle diagram, all its endpoints are understood to have positive states.

We adopt this notation for stated $\fS$-tangle diagrams because it is well suited to the $\mathcal A$-quantum trace map. Moreover, in \S\ref{sec-skein=A=U}, it makes the correspondence between cluster variables and positively stated  simple essential arcs (Definition~\ref{def-arcs}) transparent.

\begin{definition}\label{def-quasitri}
\begin{enumerate}
	\item[(a)] We call $\fS$ \emph{quasi-triangulable} if, for every connected component $\fS_0$ of $\overline\fS$, the following hold:
	\begin{itemize}
		\item $\fS_0$ has at least one boundary puncture;
		\item the total number of punctures and boundary circles is at least two;
		\item $\fS_0$ is neither a $2$-gon (Definition~\ref{def-ideal-arc}) nor an annulus with one boundary puncture. 
	\end{itemize}
	
	\item[(b)] Two admissible ideal arcs are called {\bf compatible} if they do not cross each other,  do not connect to a same boundary circle, and are not isotopic to each other.
	
	\item[(c)] A \textbf{quasi-ideal triangulation} $\Delta$ of a quasi-triangulable surface $\fS$ is a maximal collection of pairwise compatible admissible ideal arcs. Note that $\Delta$ contains every boundary edge of $\fS$. An arc in $\Delta$ that is not a boundary edge is called an \textbf{interior edge}, and we denote the set of all interior edges of $\Delta$ by $\mathring{\Delta}$.
\end{enumerate} 
\end{definition}

\begin{remark}
Note that every quasi-triangulable surface is triangulable. For each quasi-triangulation $\Delta$ of such a surface, we will construct an ideal triangulation $\lambda(\Delta)$ in Section \ref{sec:xqtm}, as illustrated in Figure~\ref{fig:triangulatio}.  
\end{remark}

\begin{remark}\label{rem:thesame}
Following the terminology of \cite{FST1} (see also \cite{BHR,H2,H3}), we associate to each quasi-triangulable surface $\fS$ an unpunctured orbifold whose orbifold points carry weights. The correspondence is given by converting boundary punctures into boundary marked points, boundary circles into orbifold points of weight $\frac{1}{2}$, and each interior puncture $p$ into an orbifold point of weight $\alpha_p$. Under this correspondence, quasi-triangulations of $\fS$ are identified with triangulations of the resulting unpunctured orbifold. Hence the generalized quantum cluster algebra in Theorem~\ref{intro-thm1} may be regarded as arising from an unpunctured orbifold.
\end{remark}

Cutting $\fS$ along all the interior edges of $\Delta$ decomposes it into regions of the following types:

\begin{equation*}
\begin{tikzpicture}
	\filldraw[fill=gray!20, line width=1pt]
	(0,0)--(2,0)--(1,1.73)--cycle;
	\filldraw[draw=black,fill=white] (-0.1,-0.1) rectangle (0.1, 0.1);
	\filldraw[draw=black,fill=white] (1.9,-0.1) rectangle (2.1, 0.1);
	\filldraw[draw=black,fill=white] (0.9,1.63) rectangle (1.1, 1.83);
\end{tikzpicture}
\hspace{1.5cm}
\begin{tikzpicture}
	\filldraw[fill=gray!20, line width=1pt]
	(1,0) arc[start angle=-90, end angle=270, radius=1] -- cycle;
	\draw [line width =1pt](1,0)--(1,1.2);
	\filldraw[fill=white]  (1,1.2) circle (0.1);
	\filldraw[draw=black,fill=white] (0.9,-0.1) rectangle (1.1, 0.1);
	\filldraw[draw=black,fill=white] (0.9,1.9) rectangle (1.1, 2.1);
\end{tikzpicture}
\hspace{1.5cm}
\begin{tikzpicture}
	\filldraw[fill=gray!20, line width=1pt]
	(1,0) arc[start angle=-90, end angle=270, radius=1] -- cycle;
	\draw [line width =1pt](1,0)--(0.4,1.2);
	\draw [line width =1pt](1,0)--(1.6,1.2);
	\filldraw[fill=white]  (0.4,1.2) circle (0.1);
	\filldraw[fill=white]  (1.6,1.2) circle (0.1);
	\filldraw[draw=black,fill=white] (0.9,-0.1) rectangle (1.1, 0.1);
\end{tikzpicture}
\hspace{1.5cm}
\begin{tikzpicture}
	\filldraw[fill=gray!20, line width=1pt]
	(1,0) arc[start angle=-90, end angle=270, radius=1] -- cycle;
	\draw[ draw=black, fill=white]
	(0.9,1.1) rectangle (1.1,1.3);
	\filldraw[draw=black,fill=white] (0.9,-0.1) rectangle (1.1, 0.1);
\end{tikzpicture}
\end{equation*}

We denote the two half-edges (Definition \ref{def-quasi}) of an admissible ideal arc $\gamma$ by $\partial_1(\gamma)$ and $\partial_2(\gamma)$.

Following \cite{muller2016skein,LY22,bloomquist2025degenerations}, for two compatible admissible deal arcs $\alpha,\beta$, denote 
\begin{equation}\label{def--Lambda}
\Lambda_{\alpha,\beta}=
	\sum_{i,j\in \{1,2\}} \left\{
	\begin{array}{c|c}
		0 & \text{if $\partial_i(\alpha)$ and $\partial_j(\beta)$ are not incident} \\
		1 & \text{if $\partial_i(\alpha)$ and $\partial_j(\beta)$ are incident, and $\partial_i(\alpha)$ is clockwise to $\partial_j(\beta)$} \\
		-1 & \text{if $\partial_i(\alpha)$ and $\partial_j(\beta)$ are incident, and $\partial_j(\beta)$ is clockwise to $\partial_i(\alpha)$} \\
	\end{array}\right\}.
\end{equation}

\vs
As we will see in Proposition~\ref{prop:quasi-commutative}, the definition of $\Lambda_{\alpha,\beta}$ is motivated by the following lemma.

\begin{lemma}[{\cite[Lemma~2.4]{le2018triangular}}]\label{lem-height-change}
	In $\SS$, we have 
	\begin{align*}
		\raisebox{-.16in}{
			\begin{tikzpicture}
				\filldraw[draw=white,fill=gray!20] (-0.2,-0) rectangle (1.2, 1);
				\draw [line width =1.5pt](-0.2,-0)--(1.2,0);
				\draw [line width =1pt](0,1)--(0.4,0.2);
				\draw [line width =1pt](1,1)--(0.6,0.2);
				\draw [line width =1pt](0.5,0)--(0.6,0.2);
				\filldraw[draw=black,fill=white] (0.4,-0.1) rectangle (0.6, 0.1);
				\node at (0.1,0.2) {$+$};
				\node at (0.9,0.2) {\small $+$};
			\end{tikzpicture}
		}=q
		\raisebox{-.16in}{
			\begin{tikzpicture}
				\filldraw[draw=white,fill=gray!20] (-0.2,-0) rectangle (1.2, 1);
				\draw [line width =1.5pt](-0.2,-0)--(1.2,0);
				\draw [line width =1pt](0,1)--(0.5,0);
				\draw [line width =1pt](1,1)--(0.6,0.2);
				\filldraw[draw=black,fill=white] (0.4,-0.1) rectangle (0.6, 0.1);
				\node at (0.1,0.2) {$+$};
				\node at (0.9,0.2) {\small $+$};
		\end{tikzpicture}}.
	\end{align*}
\end{lemma}

Denote
\begin{align}\label{eq-same-height}
	\raisebox{-.16in}{
		\begin{tikzpicture}
			\filldraw[draw=white,fill=gray!20] (-0.2,-0) rectangle (1.2, 1); 
			\draw [line width =1.5pt](-0.2,-0)--(1.2,0);
			\draw [line width =1pt](0,1)--(0.5,0);
			\draw [line width =1pt](1,1)--(0.6,0.2);
			\draw [line width =1pt](0.5,0)--(0.6,0.2);
			\filldraw[draw=black,fill=white] (0.4,-0.1) rectangle (0.6, 0.1);
			\node at (0.1,0.2) {$+$};
			\node at (0.9,0.2) {\small $+$};
		\end{tikzpicture}
	}:=
	q^{-\tfrac{1}{2}}
	\raisebox{-.16in}{
		\begin{tikzpicture}
			\filldraw[draw=white,fill=gray!20] (-0.2,-0) rectangle (1.2, 1); 
			\draw [line width =1.5pt](-0.2,-0)--(1.2,0);
			\draw [line width =1pt](0,1)--(0.4,0.2);
			\draw [line width =1pt](1,1)--(0.6,0.2);
			\draw [line width =1pt](0.5,0)--(0.6,0.2);
			\filldraw[draw=black,fill=white] (0.4,-0.1) rectangle (0.6, 0.1);
			\node at (0.1,0.2) {$+$};
			\node at (0.9,0.2) {\small $+$};
		\end{tikzpicture}
	}
	=
	q^{\tfrac{1}{2}}
	\raisebox{-.16in}{
		\begin{tikzpicture}
			\filldraw[draw=white,fill=gray!20] (-0.2,-0) rectangle (1.2, 1); 
			\draw [line width =1.5pt](-0.2,-0)--(1.2,0);
			\draw [line width =1pt](0,1)--(0.5,0);
			\draw [line width =1pt](1,1)--(0.6,0.2);
			\filldraw[draw=black,fill=white] (0.4,-0.1) rectangle (0.6, 0.1);
			\node at (0.1,0.2) {$+$};
			\node at (0.9,0.2) {\small $+$};
		\end{tikzpicture}
	},
\end{align}
and we say that the two half-edges in 
$\raisebox{-.16in}{
	\begin{tikzpicture}
		\filldraw[draw=white,fill=gray!20] (-0.2,-0) rectangle (1.2, 1); 
		\draw [line width =1.5pt](-0.2,-0)--(1.2,0);
		\draw [line width =1pt](0,1)--(0.5,0);
		\draw [line width =1pt](1,1)--(0.6,0.2);
		\draw [line width =1pt](0.5,0)--(0.6,0.2);
		\filldraw[draw=black,fill=white] (0.4,-0.1) rectangle (0.6, 0.1);
		\node at (0.1,0.2) {$+$};
		\node at (0.9,0.2) {\small $+$};
	\end{tikzpicture}
}$
are placed at the same height.
Note that this notation and the phrase “placing at the same height’’ can be extended to any number of half-edges incident to the same boundary puncture.

Let $\alpha$ be an admissible ideal multi-arc in $\fS$. 
Lemma~\ref{lem-height-change} implies that 
$D_+(\alpha)$ is reflection normalizable (see \S\ref{sub-reflection}). Then denote 
\begin{align}\label{eq-D-X-q}
	A_\alpha:=\left(\prod_{v\in C(\alpha)} v^{\frac{1}{2}} \right)[D_+(\alpha)]_{\rm norm}\in\cS_q(\fS),\qquad
	\text{(see Definition~\ref{def-CCCC})}.
\end{align}

Using Lemma~\ref{lem-height-change},
the following comes from direct computation. It
was proved in \cite{bloomquist2025degenerations} in the case $\fS$ contains no interior punctures, in \cite{LY22} in the case $\fS$ contains no boundary circles, and in \cite{muller2016skein} in the case $\fS$ contains neither interior punctures nor boundary circles.

\begin{proposition}\label{prop:quasi-commutative}
	For any punctured surface  $\fS$, in ${\cS}_q(\fS)$, for any two compatible admissible ideal arcs $\alpha,\beta$ in $\fS$, 
	we have
	\begin{equation}
		A_{\alpha}A_\beta
		=q^{\Lambda_{\alpha,\beta}}A_\beta A_\alpha. 
	\end{equation}
\end{proposition}

The following is immediate.

\begin{lemma}
	Let $\alpha$ be an admissible ideal multi-arc consisting of admissible ideal arc components $\alpha_1,\dots,\alpha_r$. 
	We have
	\begin{align}\label{eq-X-multi}
		A_\alpha\;=\;\left[A_{\alpha_1}\cdots A_{\alpha_r} \right]_{\rm norm} \;=\; 
		q^{-\tfrac{1}{2}\sum_{1 \leq i \leq j \leq r} \Lambda_{\alpha_i,\alpha_j}}
		\, A_{\alpha_1}\cdots A_{\alpha_r}\in\SS.
	\end{align}
\end{lemma}

\def\qe{\overset{q}{=}}

For two elements $x$ and $y$, we write
$x\overset{q}{=}y$
if
$x=q^{\frac{m}{2}}f\,y$
for some $m\in\mathbb Z$ and some Laurent monomial $f$ in the variables
$\left\{v^{\frac12}\ \middle|\ v \text{ is a boundary circle of }\fS\right\}.$

\def\A{\mathcal A_q(\fS,\Delta)}
\def\Ap{\mathcal A_q^{+}(\fS,\Delta)}

\begin{definition}\label{def:orientationmatrix}
	For a quasi-ideal triangulation $\Delta$ of $\fS$, we call the matrix $\Lambda^\Delta=(\Lambda_{\alpha,\beta})_{\alpha,\beta\in \Delta}$ the {\bf orientation matrix } of $\Delta$.
\end{definition}

Denote
\begin{align}\label{eq-RP}
	\wR[P]:= \wR[\alpha_p\mid p \text{ is an interior puncture of $\fS$}],
\end{align}
where the elements $\alpha_p$ are algebraically independent over $\wR$.
Define 
\begin{align}
	\mathcal A_q(\fS,\Delta)=
	\wR[P]\langle A_\alpha^{\pm 1}\mid \alpha\in\Delta \rangle/
	(A_\alpha A_\beta = q^{\Lambda_{\alpha,\beta}} A_{\beta} A_\alpha
	\text{ for }\alpha,\beta\in\Delta).
\end{align}
For each ${\bf k}\in\mathbb Z^{\Delta}$, define
\begin{align}\label{def-monomial-A}
	A({\bf k}) \;:=\;
	\left[\;\prod_{\alpha\in\Delta} A_\alpha^{{\bf k}(\alpha)} \;\right]_{\rm norm},
\end{align}
where the bracket $[\;\cdot\;]_{\rm norm}$ is defined as in \eqref{eq-weyl}.  
Set
\begin{align}
	\mathcal A_q^{+}(\fS,\Delta)
	\;:=\;
	{\rm Span}_{\wR[P]}\{\,A({\bf k})\mid {\bf k}\in \mathbb N^{\Delta}\,\}.
\end{align}
It follows immediately that $\mathcal A_q^{+}(\fS,\Delta)$ is an 
$\wR[P]$-algebra, free over $\wR[P]$ with basis 
$\{\,A({\bf k})\mid {\bf k}\in \mathbb N^{\Delta}\,\}$.

For each interior puncture $p$, by a slight abuse of notation, we denote the simple loop in $\fS$ surrounding $p$ by $\alpha_p$; this causes no ambiguity since  all these loops are algebraically independent over $\wR$.
It is straightforward to see that $\alpha_p$ lies in the center of 
$\overline{\cS}_q(\fS)$. 
Thus, $\overline{\cS}_q(\fS)$ can be viewed as an $\wR[P]$-algebra.

Theorem~\ref{thm-basis-reduced} implies the following.
\begin{corollary}\label{cor-basis-reduced}
	$\overline{\cS}_q(\fS)$ is a free $\wR[P]$-module with basis $\widetilde{B}(\fS)$ (see \eqref{def-WB}).
\end{corollary}

Let $\cS_q^+(\fS)$ denote the subalgebra of $\cS_q(\fS)$ generated by positively stated $\fS$-tangles.
We call $\cS_q^+(\fS)$ the \textbf{Muller--Roger--Yang skein algebra} (or \textbf{MRY skein algebra}).
Let $\mathcal M$ be the multiplicative subset of $\cS_q^+(\fS)$ generated by the elements $A_e$, where $e$ ranges over all boundary edges of $\fS$.
Then we have the following.

\begin{lemma}{\cite[Proposition~3.1]{karuo2025cluster}}\label{lem-positive-reduced}
	The canonical projection
	$\cS_q(\fS)\longrightarrow \overline{\cS}_q(\fS)$
	restricts to an algebra embedding
	\[
	\cS_q^+(\fS)\hookrightarrow \overline{\cS}_q(\fS).
	\]
	Moreover,
	\[
	\overline{\cS}_q(\fS)
	=
	\cS_q^+(\fS)[\mathcal M^{-1}].
	\]
\end{lemma}

For each ${\bf k}\in\mathbb N^{\Delta}$, we define an admissible ideal multiarc $\beta_{\bf k}$ as follows. 
For every ideal arc $\alpha$, the admissible ideal multiarc $\beta_{\bf k}$ contains ${\bf k}(\alpha)$ copies of $\alpha$. 
For every pending arc $\alpha'$ (Definition~\ref{def-quasi}), let $\alpha''$ be the ideal arc bounding a monogon that contains $\alpha'$. Then:  
\begin{itemize}
	\item if ${\bf k}(\alpha')$ is even, $\beta_{\bf k}$ contains $\tfrac{1}{2}{\bf k}(\alpha')$ copies of $\alpha''$;  
	\item if ${\bf k}(\alpha')$ is odd, $\beta_{\bf k}$ contains $\tfrac{1}{2}\bigl({\bf k}(\alpha')-1\bigr)$ copies of $\alpha''$ together with one copy of $\alpha'$.  
\end{itemize}
It is straightforward to verify that $\beta_{\bf k}\neq \beta_{\bf t}$ if ${\bf k}\neq {\bf t}$.

The following Proposition establishes the $\mathcal A$-quantum trace map of the reduced LRY skein algebra, which
extends \cite[Theorem 6.5]{baseilhac2023unrestricted} to the case when $\fS$ contains interior punctures. 
In \S\ref{sec-skein=A=U}, we will use this quantum trace map to establish a generalized cluster structure of the reduced LRY skein algebra.

\begin{proposition}\label{prop:quantumtrace}
	Let $\fS$ be a punctured surface with a quasi-ideal triangulation $\Delta$.
	Then there exists a unique 
	$\wR[P]$-algebra embedding 
	\begin{align*}
		{\rm tr}_\Delta^A\colon \overline{\cS}_q(\fS)\hookrightarrow \A, \qquad
		A_\alpha \longmapsto A_\alpha \quad (\alpha\in\Delta),
	\end{align*}
	such that 
	\begin{align}\label{eq-sandwitsh}
		\Ap \subset \im {\rm tr}_\Delta^A \subset \A.
	\end{align}
\end{proposition}

\begin{proof}
	Proposition~\ref{prop:quasi-commutative} implies that the map
	\begin{align*}
		\iota\colon \Ap \longrightarrow \rdS,
		\qquad 
		A_\alpha \longmapsto A_\alpha \quad (\alpha\in\Delta),
	\end{align*}
	is a well-defined $\wR[P]$-algebra homomorphism. 
	
	Moreover, Equations~\eqref{RY-relation} and \eqref{eq-same-height} imply
	\begin{align}\label{eq-square-double-plus}
		v\;   \left( \raisebox{-.15in}{
			\begin{tikzpicture}
				\filldraw[draw=white,fill=gray!20] (-0,-0) rectangle (1, 1);
				\draw [line width =1.5pt](-0,-0)--(1,0);
				\draw [line width =1pt](0.5,-0)--(0.5,0.5);
				\filldraw[draw=black,fill=white] (0.4,-0.1) rectangle (0.6, 0.1);
				\filldraw[draw=black,fill=white]  (0.5,0.5) circle (0.1);
				\node at (0.3,0.2) {\small $+$};
			\end{tikzpicture}
		}\right)^2 \;=\;
		\raisebox{-.15in}{
			\begin{tikzpicture}
				\filldraw[draw=white,fill=gray!20] (-0.2,-0) rectangle (1.2, 1);
				\draw [line width =1.5pt](-0.2,-0)--(1.2,0);
				\draw [line width =1pt]  (0.5 ,0) arc (-90:250:0.3 and 0.4);
				\filldraw[draw=black,fill=white] (0.4,-0.1) rectangle (0.6, 0.1);
				\filldraw[draw=black,fill=white]  (0.5,0.5) circle (0.1);
				\node at (-0,0.2) {\small $+$};
				\node at (1,0.2) {\small $+$};
			\end{tikzpicture}
		}
		\;\in\;\rdS,
	\end{align}
	where $v$ is the boundary circle in the equality. Note that and the equality in fact already holds in $\SS$.
	
	Combining \eqref{eq-D-X-q}, \eqref{eq-X-multi},  \eqref{def-monomial-A}, 
	and \eqref{eq-square-double-plus}, we obtain
	\begin{align*}
		\iota(A({\bf k})) \;=\; A_{\beta_{\bf k}} \;\qe\; D_+(\beta_{\bf k}) \;\in\; \WB.
	\end{align*}
	Since $D_+(\beta_{\bf k})\neq D_+(\beta_{\bf t})$ whenever ${\bf k}\neq{\bf t}\in\mathbb N^{\Delta}$, Corollary~\ref{cor-basis-reduced} shows that $\iota$ is injective.
	
	Hence we may regard $\Ap$ as an $\wR[P]$-subalgebra of $\rdS$ via~$\iota$. 
	Let $X\in\rdS$.  
	By Lemma~\ref{lem-positive-reduced}, there exists 
	${\bf k}\in\mathbb Z^{\Delta}$, with ${\bf k}(\alpha)=0$ whenever 
	$\alpha$ is an interior edge, such that
	\[
	A({\bf k}) X \;=\; Y \;\in\; \cS_q^{+}(\fS).
	\]

	Using the technique of \cite[Corollary~4.13]{muller2016skein}, one shows that there exists ${\bf t}\in\mathbb N^{\Delta}$ such that $A({\bf t}) Y\in \Ap$.  
	Set ${\bf k}'\in \mathbb N^{\Delta}$ by ${\bf k}'(\alpha) = |{\bf k}(\alpha)|$ for all $\alpha\in\Delta$, and define 
	${\bf t}'={\bf k}'+{\bf k}+{\bf t}\in\mathbb N^{\Delta}$. 
	Then
	\begin{align}\label{eq-proof-A-trace}
		A({{\bf t}'}) X \;\qe\; A({{\bf k}'}) A({\bf t}) Y \;\in\; \Ap.
	\end{align}
	
	Finally, together with \eqref{eq-proof-A-trace}, \cite[Proposition~2.2]{LY22} implies that the algebra embedding 
	$\Ap\hookrightarrow \A$ extends uniquely to an algebra embedding
	\[
	{\rm tr}_\Delta^A\colon \rdS \;\hookrightarrow\; \A.
	\]
\end{proof}

\subsection{${\mathcal {X}}$-quantum trace maps}\label{sec:xqtm}

In this section, we review the $\mathcal{X}$-quantum trace map for the reduced LRY skein algebra constructed in \cite{bloomquist2025degenerations}. We then establish the compatibility between the $\mathcal{X}$- and $\mathcal{A}$-quantum trace maps; see \eqref{eq-com-traces-AX}. In \S\ref{Sec-Fro}, we use this compatibility to prove that the Frobenius map commutes with the cutting homomorphism.

Let $\overline{Q}_\lambda\colon
\lambda\times\lambda\longrightarrow\mathbb{Z}$
be the antisymmetric function defined by
\begin{equation}\label{eq:Q-Delta}
	\overline{Q}_{\lambda}(a,b)
	=
	\#\left(
	\raisebox{-0.4\height}{%
		\begin{tikzpicture}[scale=0.42]
			\fill[gray!30]
			(0,0) -- (-1.2,1.2)
			arc[start angle=135,end angle=45,radius=1.697]
			-- cycle;
			\draw (-1.4,1.4) -- (0,0) -- (1.4,1.4);
			\filldraw[fill=white]
			(-0.12,0.1) -- (-0.12,-0.15)
			-- (0.12,-0.15) -- (0.12,0.1) -- cycle;
			\node at (-0.75,0.35) {$\scriptstyle b$};
			\node at (0.75,0.35) {$\scriptstyle a$};
		\end{tikzpicture}%
	}
	\right)
	-
	\#\left(
	\raisebox{-0.4\height}{%
		\begin{tikzpicture}[scale=0.42]
			\fill[gray!30]
			(0,0) -- (-1.2,1.2)
			arc[start angle=135,end angle=45,radius=1.697]
			-- cycle;
			\draw (-1.4,1.4) -- (0,0) -- (1.4,1.4);
			\filldraw[fill=white]
			(-0.12,0.1) -- (-0.12,-0.15)
			-- (0.12,-0.15) -- (0.12,0.1) -- cycle;
			\node at (-0.75,0.35) {$\scriptstyle a$};
			\node at (0.75,0.35) {$\scriptstyle b$};
		\end{tikzpicture}%
	}
	\right).
\end{equation}
Here, each shaded region represents a corner of an ideal triangle.
Thus, the right-hand side of \eqref{eq:Q-Delta} is the number of
corners at which $b$ lies counterclockwise from $a$, minus the
number of corners at which $a$ lies clockwise from $b$, as viewed
from the puncture.

Let $\overline{\mathcal{X}}_q(\fS,\lambda)$ be the quantum torus
associated with 
$\overline{Q}_\lambda$. More precisely,
\begin{align}\label{eq-def-X-torus}
	{\overline{\mathcal{X}}}_q(\fS,\lambda)
	=\wR
	\left\langle
	{x}_e^{\pm1},\ e\in\lambda
	\ \middle|\
	{x}_e {x}_{e'}
	=
	q^{\overline{Q}_\lambda(e,e')}
	{x}_{e'}{x}_e
	\right\rangle.
\end{align}

\def\bS{\mathbb{S}}

Let $e\in {\lambda}$ be an interior edge. Cutting $\fS$ along $e$ produces a surface
${\rm Cut}_e\fS$ with two boundary edges $e'$ and $e''$, which are identified to recover $e$. The induced triangulation of ${\rm Cut}_e\fS$ is
\[
{\rm Cut}_e\lambda
=
\bigl(\lambda\setminus\{e\}\bigr)\cup\{e',e''\}.
\]

There is an embedding
\[
\bS_e\colon
\overline{\mathcal{X}}_q(\fS,\lambda)
\hookrightarrow
\overline{\mathcal{X}}_q\bigl({\rm Cut}_e\fS,
{\rm Cut}_e\lambda\bigr)
\]
defined on the generators by
\[
\bS_e({x}_a)
=
\begin{cases}
	{x}_a,
	& a\neq e,\\[2mm]
	\bigl[\overline{x}_{e'}\overline{x}_{e''}\bigr]_{\mathrm{norm}},
	& a=e.
\end{cases}
\]

\def\tr{{\rm tr}}
\def\Cut{{\rm Cut}}

The following establishes the $\mathcal X$-quantum trace map of $\rdS$.

\begin{theorem}[\cite{bloomquist2025degenerations}]
	Suppose that $\fS$ is a triangulable punctured surface with an ideal triangulation $\lambda$. Then there exists an
	algebra embedding
	\[
	{\tr}_{\lambda}\colon
	\overline{\cS}_q(\fS)\longrightarrow \overline{\mathcal{X}}_q(\fS,\lambda)
	\]
	Moreover, the homomorphism ${\tr}_{\lambda}$ is compatible
	with the cutting homomorphism in the following sense. For every
	non-boundary edge $e\in {\lambda}$, the following diagram commutes:
	\begin{equation}\label{thm-eq-cut-trace}
		\begin{tikzcd}[column sep=large,row sep=large]
			\overline{\cS}_q(\fS)
			\arrow[r,"{\tr}_{\lambda}"]
			\arrow[d,"\overline\Theta_e"']
			&
			\overline{\mathcal{X}}_q(\fS,\lambda)
			\arrow[d,"\bS_e"]
			\\
			\rd_q({\Cut}_e\fS)
			\arrow[r,"{\tr}_{{\Cut}_e\lambda}"']
			&
			\overline{\mathcal{X}}_q	\bigl({\Cut}_e\fS,{\Cut}_e\lambda\bigr).
		\end{tikzcd}
	\end{equation}
\end{theorem}

\begin{figure}
	\centering
	\includegraphics[width=0.25\linewidth]{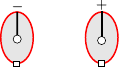}
	\caption{The left puncture is labeled $u_-$, and the right puncture is labeled $u_+$.}
	\label{fig:uplus-a}
\end{figure}

Let $\mathfrak{M}$ denote the punctured surface 
obtained from a closed annulus by removing a puncture from one of its boundary component
(see Figure~\ref{fig:uplus-a}), and let $u_-$ and $u_+$ be the two stated
$\mathfrak{M}$-tangle diagrams depicted therein. We denote the unique
boundary circle of $\mathfrak{M}$ by $v$.

\begin{theorem}{\cite[Theorem~5.6]{bloomquist2025degenerations}}\label{trace-monogon}
	The $\mathcal X$-quantum trace map of 
	$\rd_q(\mathfrak{M})$ is given by the following isomorphism
	\begin{align}\label{eq-trace-monogon}
		{\rm tr}\colon \rd_q(\mathfrak{M})
		\longrightarrow \wR[x^{\pm1}]; \qquad
		u_- &\longmapsto v^{-\frac12}x^{-1}, \quad
		u_+ \longmapsto v^{-\frac12}x.
	\end{align}
\end{theorem}

Next we describe a basis of the MRY skein algebra $\cS_q^+(\fS)$ and its parametrization with respect to an ideal triangulation of $\fS$. We show that the leading term of the $\mathcal X$-quantum trace of each basis element is determined by its coordinates; see Lemma~\ref{lem-leading-term}. In \S\ref{sec-center}, we use this result to compute the center of $\cS_q^+(\fS)$.

\vs

Let $B'(\fS)\subset B(\fS)$ consist of basis elements that are positively stated. 
Define 
$$B^+(\fS) := \{[\alpha]_{\rm norm}\mid \alpha\in B'(\fS)\}.$$

Suppose that $\fS$ is triangulable with an ideal triangulation $\lambda$ (Definition~\ref{def-ideal-arc}).
For each $\alpha\in B^+(\fS)$, define a function
\begin{align}\label{eq-coordinate}
	\mathbf{n}_{\alpha}\colon \lambda\longrightarrow\mathbb{N};
	\qquad e\mapsto I(\alpha,e).
\end{align}
Here, $I(\alpha,e)$ denotes the geometric intersection number of $\alpha$ and $e$.

The proof of \cite[Theorem~4.5]{bloomquist2025degenerations} can be
adapted directly to establish the following result.

\begin{lemma}
	\label{lem-basis-coordinate-monoid}
	Let $\lambda$ be an ideal triangulation of a triangulable punctured
	surface $\fS$.
	\begin{enumerate}
		\item[\textnormal{(a)}]
		The map
		\[
		\kappa \colon B^+(\fS)\longrightarrow \mathbb N^\lambda,
		\qquad
		\kappa (\alpha)={\bf n}_\alpha,
		\]
		is a bijection from $B^+(\fS)$ onto the submonoid
		$\Gamma_\lambda\subset\mathbb N^\lambda$ consisting of all
		${\bf n}\in\mathbb N^\lambda$ such that, for every triangle of $\lambda$ with edges $a,b,c$,
		\begin{equation}\label{eq-tri-1}
			{\bf n}(a)+{\bf n}(b)+{\bf n}(c)\in 2\mathbb N
		\end{equation}
		and
		\begin{equation}\label{eq-tri-2}
			{\bf n}(a)\leq {\bf n}(b)+{\bf n}(c).
		\end{equation}
		
		\item[\textnormal{(b)}]
		The submonoid $\Gamma_\lambda\subset\mathbb N^\lambda$ is finitely
		generated. Moreover, the subgroup of $\mathbb Z^\lambda$ generated by
		$\Gamma_\lambda$ is a free abelian group of rank
		\begin{equation}
			\label{eq-rank-lambda}
			\lvert\lambda\rvert
			=3a+2 b + 6g + 3 d -6
		\end{equation}
		where $a$ is the number of interior punctures,
		$g$ is the genus of $\fS$, $b$ is the number of connected components of
		$\partial\fS$, and $d$ is the number of components of
		$\partial\overline{\fS}$ that contain at least one puncture.
	\end{enumerate}
\end{lemma}

The following lemma illustrates the leading term of each basis element in $\alpha\in\cS_q^+(\fS).$

\begin{lemma}\label{lem-leading-term}
	For each $\alpha\in B^+(\fS)$, we have
	\[
	\tr_\lambda(\alpha)
	=
	\left(\prod_{v\in C(\alpha)} v^{-\frac{1}{2}} \right) x^{{\bf n}_\alpha}
	+
	\sum_{{\bf t}\in S} r_{\bf t}x^{\bf t},
	\]
	where $S\subset \mathbb Z^\lambda$ is a finite subset not containing
	${\bf n}_\alpha$, each $r_{\bf t}\in \wR\setminus\{0\}$, and
	\[
	{\bf n}_\alpha-{\bf t}\in (2\mathbb N)^\lambda
	\qquad\text{for every ${\bf t}\in S$.}
	\]
\end{lemma}

\begin{proof}
	Place $\alpha$ in minimal position with respect to every  edge
	of $\lambda$.
	Let $\mathring{\lambda}$ denote the subset of $\lambda$ consisting of all non-boundary edges.
	Let $\lambda(\pm)$ denote the set of all maps
	\[
	f\colon
	\alpha\cap\left(\bigcup_{e\in\mathring{\lambda}}e\right)
	\longrightarrow \{\pm1\}.
	\]
	For each $f\in\lambda(\pm)$, define
	${\bf n}_f\in\mathbb Z^\lambda$ by
	\[
	{\bf n}_f(e)
	=
	\begin{cases}
		\#(\alpha\cap e),
		& e\in\lambda\setminus\mathring{\lambda},\\[2mm]
		\displaystyle\sum_{p\in\alpha\cap e}f(p),
		& e\in\mathring{\lambda}.
	\end{cases}
	\]
	
	Let $g\in\lambda(\pm)$ be the constant map with value $1$. Then
	${\bf n}_g={\bf n}_\alpha.$
	Moreover, for every $f\neq g$, we have
	\[
	{\bf n}_f\neq{\bf n}_\alpha
	\qquad\text{and}\qquad
	{\bf n}_\alpha-{\bf n}_f
	=
	{\bf n}_g-{\bf n}_f
	\in(2\mathbb N)^\lambda.
	\]
	
	Cutting $\fS$ along all interior edges
	$e\in\mathring{\lambda}$, the commutative diagram
	\eqref{thm-eq-cut-trace}, together with
	\cite[Theorem~5.5]{bloomquist2025degenerations} and Theorem~\ref{trace-monogon}, implies that
	every monomial occurring in $\tr_\lambda(\alpha)$ with nonzero
	coefficient is of the form $x^{{\bf n}_f}$ for some
	$f\in\lambda(\pm)$. Furthermore, \eqref{eq-trace-monogon} implies that the coefficient of
	$x^{{\bf n}_g}=x^{{\bf n}_\alpha}$ is equal to $\left(\prod_{v\in C(\alpha)} v^{-\frac{1}{2}} \right)$. The desired
	expansion follows.
\end{proof}

Combining Lemmas~\ref{lem-basis-coordinate-monoid} and
\ref{lem-leading-term} with the techniques used in
\cite[Proposition~5.16 and Corollary~5.17]{kim2024unicity}, we obtain
the following result.

\begin{lemma}\label{lem-finite-order}
	Let $\fS$ be a punctured surface. Then there exist elements
	$\beta_1,\ldots,\beta_m\in B^+(\fS)$, each represented by either an
	arc or a loop, that orderly generate ${\cS}_q^+(\fS)$; that is,
	\begin{align}\label{eq-span-finite}
		{\cS}_q^+(\fS)
		=
		{\rm Span}_{\wR}
		\left\{
		\beta_1^{k_1}\cdots\beta_m^{k_m}
		\;\middle|\;
		k_i\in\mathbb N,\ 1\leq i\leq m
		\right\}.
	\end{align}
\end{lemma}

The following will be used in \S\ref{sec-skein=A=U}.

\begin{corollary}\label{cor:gen}
	Suppose that each component of $\fS$ contains at least one boundary puncture. Then
	the $\wR$-algebra $\overline{\cS}_q(\fS)$ is generated by finitely many stated simple essential arcs (Definition~\ref{def-arcs}).
\end{corollary}

\begin{proof}
	Lemmas~\ref{lem-positive-reduced} and \ref{lem-finite-order} imply that the $\wR$-algebra $\overline{\cS}_q(\fS)$ is generated by finitely many stated simple arcs and loops. 
	Using relation~\eqref{hight}, every simple loop or circle-arc can be generated by finitely many stated simple essential arcs.
	This proves the claim. 
\end{proof}

We then define the degree of each nonzero element of $\cS_q^+(\fS)$. In \S\ref{sec-center}, we use this degree to compute the center and the rank of $\cS_q^+(\fS)$ (Definition~\ref{def-rank}).

Fix a linear order on $\lambda$, and let $\leq_{\mathrm{lex}}$ denote
the induced lexicographic order on $\mathbb Z^\lambda$. This, in turn,
induces a linear order $\leq$ on $B^+(\fS)$ defined by
\begin{align}
	\alpha\leq\beta
	\quad\Longleftrightarrow\quad
	{\bf n}_\alpha\leq_{\mathrm{lex}}{\bf n}_\beta,
	\qquad
	\alpha,\beta\in B^+(\fS).
\end{align}

\def\lt{{\rm lt}}

\begin{definition}
	\label{def-leading-term-degree}
	Let $X\in\cS_q^+(\fS)$ be nonzero. Write
	\[
	X
	=
	c\alpha+\sum_{i=1}^{m}c_i\alpha_i,
	\]
	where $\alpha,\alpha_1,\ldots,\alpha_m\in B^+(\fS)$ are distinct,
	\[
	\alpha_i<\alpha
	\qquad\text{for }1\leq i\leq m,
	\]
	and $c,c_i\in \wR\setminus\{0\}$. The \textbf{leading term} and
	\textbf{degree} of $X$ are defined, respectively, by
	\[
	{\lt}(X):=c\alpha
	\qquad\text{and}\qquad
	\deg(X):={\bf n}_\alpha\in\Gamma_\lambda.
	\]
\end{definition}

Together with Lemma~\ref{lem-leading-term}, the argument used in the
proof of \cite[Lemma~5.14]{kim2024unicity} yields the following result.

\begin{lemma}\label{lem-leading-deg}
	For any nonzero elements $X,Y\in\cS_q^+(\fS)$, one has
	the following:
	\begin{enumerate}
		\item [\rm (a)] $\deg(XY) = \deg(X) + \deg(Y)$,
		
		\item [\rm (b)] ${\lt}(XY)
		=
		q^{\,\deg(X)\overline Q_\lambda\deg(Y)^{ T}}
		{\lt}(YX).$
	\end{enumerate}
\end{lemma}

\begin{figure}
	\centering
	\includegraphics[width=0.5\linewidth]{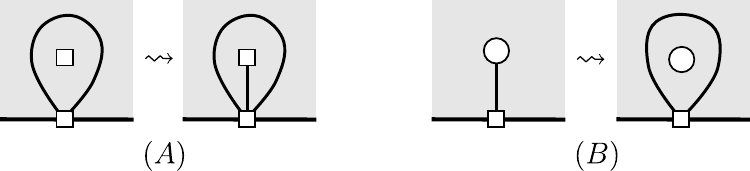}
	\caption{The steps to obtain the ideal triangulation $\lambda(\Delta)$ from a quasi-ideal triangulation $\Delta$.}
	\label{fig:triangulatio}
\end{figure}

At the end of this section, as promised, we establish the compatibility between the $\mathcal A$- and $\mathcal X$-quantum trace maps.

Let $\Delta$ be a quasi-ideal triangulation of $\fS$. We construct an ideal
triangulation $\lambda(\Delta)$ of $\fS$ from $\Delta$ as follows. For each interior puncture
$p\in{\mathcal P}$, add an ideal arc $e_p$ connecting $p$ to
a boundary puncture, as illustrated in Figure~\ref{fig:triangulatio}(A). For each
boundary circle, remove the pending arc $\alpha$ connecting the boundary
circle to a boundary puncture and replace it with an ideal arc
surrounding $\alpha$, as illustrated in Figure~\ref{fig:triangulatio}(B).

\begin{lemma}[\cite{LY22}]\label{lem-traceXA} 
	(a)   For any admissible ideal arc $\alpha\in \Delta$, there exists a vector
	${\bf k}_\alpha$ such that 
	$$\tr_\lambda(A_\alpha) = x^{{\bf k}_\alpha}\in \overline{\mathcal {X}}_q(\fS, \lambda(\Delta)) \qquad \text{(see \eqref{eq-D-X-q})}.$$
	
	(b) For each interior puncture $p$, we have 
	$$\tr_\lambda(\alpha_p) = x_{e_p} + x_{e_p}^{-1}\in \overline{\mathcal {X}}_q(\fS, \lambda(\Delta)),$$
	where $\alpha_p$ is the peripheral loop surrounding $p$.
\end{lemma}

Lemma~\ref{prop:quasi-commutative} and
Lemma~\ref{lem-traceXA}(a) imply that
\[
x^{{\bf k}_\alpha}x^{{\bf k}_\beta}
=
q^{\Lambda_{\alpha,\beta}}
x^{{\bf k}_\beta}x^{{\bf k}_\alpha}.
\]
Moreover, it is well known that $x_{e_p}$ is central in
$\overline{\mathcal {X}}_q\bigl(\fS,\lambda(\Delta)\bigr)$ \cite{LY22}. Therefore, there is a
well-defined $\wR$-algebra homomorphism
\[
\begin{aligned}
	\psi\colon \mathcal A_q(\fS,\Delta)
	&\longrightarrow \overline{\mathcal {X}}_q\bigl(\fS,\lambda(\Delta)\bigr),\\
	A_\alpha &\longmapsto x^{{\bf k}_\alpha},
	&&\alpha\in\Delta,\\
	\alpha_p &\longmapsto x_{e_p}+x_{e_p}^{-1},
	&&p \text{ is an interior puncture}.
\end{aligned}
\]

The following lemma will be used later to establish the compatibility
of the desired Frobenius map with the cutting map.

\begin{lemma}
	The algebra homomorphism $\psi$ is an embedding, and the following
	diagram commutes:
	\begin{equation}\label{eq-com-traces-AX}
		\begin{tikzcd}[column sep=1.5em, row sep=1.8em]
			\mathcal A_q(\fS,\Delta)
			\arrow[rr, "\psi"]
			&&
			\overline{\mathcal {X}}_q\bigl(\fS,\lambda(\Delta)\bigr)
			\\
			&
			\rdS
			\arrow[ul, "\tr_\Delta^A"]
			\arrow[ur, "\tr_\lambda"']
			&
		\end{tikzcd}.
	\end{equation}
\end{lemma}

\begin{proof}
	Via the embedding $\tr_\Delta^A$, we identify $\rdS$ with a subalgebra
	of $\mathcal A_q(\fS,\Delta)$.
	Lemma~\ref{lem-traceXA}  implies 
	$\left.\psi\right|_{\mathcal A_q^+(\fS,\Delta)}
	=
	\left.\tr_\lambda\right|_{\mathcal A_q^+(\fS,\Delta)}.$
	Then the result now follows from \eqref{eq-sandwitsh}.
\end{proof}

\section{Generalized quantum cluster algebras and Frobenius maps}
Throughout this section, let $I$ be a finite set, and let $\mathfrak R$ be a domain containing a distinguished invertible element $q^{\frac{1}{2}}$. We use the notation $\mathfrak R$ rather than $R$ because, in \S\ref{sec-skein=A=U}, we will take $\mathfrak R=\wR$ (see \eqref{eq-Rtilde}).
Define
\[
\mathcal R=\mathfrak R[p_i\mid i\in I]
\]
to be the polynomial ring over $\mathfrak R$ in the variables $p_i$, for $i\in I$.

We first recall the definition of generalized quantum cluster algebras, and then construct Frobenius maps for them (Theorem~\ref{thm-Fro-cluster}).

\subsection{Generalized quantum cluster algebras}\label{subsec-generalized}

\begin{definition}[\cite{BZ}]\label{def:compatible pair}
	Let $\widetilde{B}=(b_{ij})_{}$ be an $m\times n$ integer matrix with $m\geq n$, and let $\Lambda=(\lambda_{ij})$ be an $m\times m$ skew-symmetric integer matrix. The pair $(\Lambda,\widetilde{B})$ is said to be {\bf compatible} if
	\begin{equation}\label{eq:compatible pair}
		(\widetilde{B})^\top \Lambda=(D\; O)
	\end{equation}
	where $D=\text{diag}\{d_1,\ldots, d_n\}$ is an $n\times n$ diagonal matrix with positive integers diagonal entries $d_i$ and $O$ is the $n\times (m-n)$ zero matrix.
\end{definition}

It follows from \eqref{eq:compatible pair} that the top $n\times n$ submatrix $B$ of $\widetilde B$ is skew-symmetrizable. Precisely, $DB$ is skew-symmetric.

\begin{definition}\label{def:Lambda}
	The {\bf quantum torus} $\mathbb{T}(\Lambda)=\mathbb{T}_{\mathcal R}(\Lambda)$ is the $\mathcal R$-algebra with a distinguished $\mathcal R$-basis $\{A(c)\mid
	c\in \mathbb Z^{m}\}$ whose multiplication is given by
	\begin{equation}\label{eq-def-A-relation}
		A(c)A(d)=q^{\frac{1}{2}\Lambda(c,d)}A(c+d)
	\end{equation}
	for any $c,d\in\mathbb Z^{m}$, where $\Lambda(-,-): \mathbb Z^m\times \mathbb Z^m\to \mathbb Z$ is the skew-symmetric bilinear form induced by $\Lambda$.
\end{definition}

Let ${\rm Frac}(\mathbb T(\Lambda))$ denote the skew-field of fractions of $\mathbb T(\Lambda)$.

Fix an $m\times m$ skew-symmetric integral matrix $\Lambda_0$ and write $\textbf{F} = {\rm Frac}(\mathbb T(\Lambda_0))$.

\begin{definition}[\cite{BCDX0}]\label{def:generalizedQS}
	A \textbf{generalized quantum seed} in $\textbf{F}$ is a quadruple $(A, \mathbf{h}, \Lambda, \widetilde{B})$, where:
	\begin{itemize}
		\item $\Lambda$ is an $m \times m$ skew-symmetric integral matrix and $\widetilde{B} = (b_{ij})$ is an $m \times n$ integral matrix such that the pair $(\Lambda, \widetilde{B})$ is compatible and ${\rm Frac}(\mathbb{T}(\Lambda)) = \textbf{F}$. Then $\widetilde B$ is called the {\bf extended exchange matrix} and the top $n\times n$ submatrix $B$ is called the {\bf exchange matrix}.
		
		\item The map $A\colon \mathbb Z^m\rightarrow \textbf{F} $ satisfying \eqref{eq-def-A-relation}. Note that $A$ is uniquely determined by $A_i:= A({e_i})$ for $1\leq i\leq m$, where $e_i \in \mathbb{Z}^m$ is the $i$-th standard basis vector. So we may write $A = (A_1, \ldots, A_m)$, called the \textbf{extended cluster}.
		The elements $A_1, \ldots, A_n$ are called \textbf{exchangeable cluster variables}, and  $A_{n+1}, \ldots, A_m$ are called \textbf{frozen cluster variables}.
		
		\item $\mathbf{h} = (\mathbf{h}_1, \ldots, \mathbf{h}_n)$, where for each $i = 1, \ldots, n$, the $i$-th string is $\mathbf{h}_i = \{h_{i,0},h_{i,1}, \ldots, h_{i,\ell_i}\}$ with $h_{i,r} \in \mathcal {R}$. Here, $\ell_i$ is a fixed positive integer such that $\ell_i \mid b_{ji}$ for all $j$, and we require $h_{i,0} = h_{i,\ell_i} = 1$ and $h_{i,r} = h_{i,\ell_i - r}$ for all $r = 1, \ldots, \ell_i - 1$. We call $h_{i,1}, \ldots, h_{i,\ell_i-1}, i=1,\cdots,n$ {\bf coefficients}.
		
	\end{itemize}
	
\end{definition}

\begin{remark}
	In \cite{BCDX0,BCDX}, the authors require $p_i=1$ for all $i\in I$.
\end{remark}

Before introducing the mutation of a generalized quantum seed, we first recall the notion of matrix mutation.
Let $\widetilde B=(b_{ij})\in {\rm Mat}_{m\times n}(\mathbb Z)$ be an integral matrix.
For any $k\in \{1,2,\cdots, n\}$, we define the integral matrix $\widetilde B'=\mu_k(\widetilde B)$ by
\begin{align}\label{eq-mutation-Q}
	b'_{ij}= 
	\begin{cases}
		- b_{ij}, & \text{if } k\in\{j,j\},\\[0.4em]
		b_{ij} + \dfrac{1}{2}\!\left(b_{ij}\,|b_{kj}| + |b_{ij}|\,b_{kj}\right), & \text{if } k\notin\{i,j\}.
	\end{cases}
\end{align}

\begin{definition}\label{def:mutation}
	Let ${\bf s} = (X, \mathbf{h}, \Lambda, \widetilde{B})$ be a generalized quantum seed in $\textbf{F}$. For any $1 \leq k \leq n$, the \textbf{mutation} of ${\bf s}$ at direction $k$ is the new quadruple $\mu_k({\bf s}) := (A', \mathbf{h}, \Lambda', \widetilde{B}')$, defined as follows:
	\begin{enumerate}
		\item The new cluster variables $A' = (A'_1, \ldots, A'_m)$ are given by
		\begin{equation}\label{eq:mutation-X}
			A_i'=A'(e_i) = 
			\begin{cases}
				A(e_i) & \text{if } i \neq k, \\
				\displaystyle \sum_{r=0}^{\ell_k} h_{k,r} \cdot A\left( \sum_{j=1}^{m} \big(r[\beta_{jk}]_+ + (\ell_k - r)[-\beta_{jk}]_+\big) e_j - e_k \right) & \text{if } i = k,
			\end{cases}
		\end{equation}
		where $\beta_{jk} = b_{jk}/\ell_k \in \mathbb{Z}$, and denote $\mu_k(A_i)=A'_i$.
		
		\item The matrix $\widetilde \Lambda'$ is determined by the quasi-commutativity of $A'_1,\cdots, A'_m$. 
		
		\item The new matrix $\widetilde B'$ is obtained by matrix mutation:
		\[
		\widetilde{B}' = \mu_k(\widetilde{B}).
		\]
	\end{enumerate}
	
\end{definition}

Note that the coefficient strings ${\bf h}$ are invariant.

For a generalized quantum seed ${\bf s}=(A,{\bf h},\Lambda,\widetilde B)$, we denote 
\begin{equation}\label{eq:quantumtorus}
	\mathbb T({\bf s})=\mathbb T_{\mathfrak R[h_{i,r}\mid 1\leq i\leq n, 0\leq r\leq \ell_k]}(\Lambda),
\end{equation}
where $\mathfrak R[h_{i,r}\mid 1\leq i\leq n, 0\leq r\leq \ell_k]$ is the $\mathfrak R$-subalgebra of $\mathcal R$ generated by $h_{i,r}, 1\leq i\leq n, 0\leq r\leq \ell_k$.

The quadruple $(A', \mathbf{h}, \Lambda', \widetilde{B}')$ is a generalized quantum seed in $\textbf{F}$, and mutation is an involution, i.e., $\mu_k(\mu_k({\bf s})) = {\bf s}$. See Propositions 3.6 and 3.7 of \cite{BCDX0}.

Two generalized quantum seeds ${\bf s}$ and ${\bf s}'$ are said to be \textbf{mutation-equivalent} if they are related by a sequence of mutations; this relation is denoted by ${\bf s}\sim {\bf s'}$.

\begin{definition}\label{def:qca}
	The \textbf{generalized quantum cluster algebra} $\mathscr{A}_q(\mathbf{s})$ with ground ring $\mathfrak R$ associated with the initial seed $\mathbf{s}$ is the 
	$\mathfrak R$-subalgebra of $\textbf{F}$ generated by all exchangeable cluster variables from seeds that are mutation-equivalent to $\mathbf{s}$, coefficients $h_{i,r},\, 1\leq i\leq n, 0\leq r\leq \ell_k$, and all frozen variables with their inverses $A^{\pm 1}_i, n+1\leq i\leq m$. The \textbf{generalized upper quantum cluster algebra} $\mathscr{U}_q(\mathbf{s})$ associated with the initial seed $\mathbf{s}$ is defined as 
	$$\mathscr{U}_q(\mathbf{s})=\bigcap_{{\bf s'}\sim {\bf s}} \mathbb T({\bf s}'), \quad\text{(see \eqref{eq:quantumtorus})}$$
	where the intersection runs over all seeds ${\bf s'}$ mutation-equivalent to ${\bf s}$. 
\end{definition}

\begin{theorem}[Laurent phenomenon, {\cite[Theorem 3.1]{BCDX}}]\label{thm:laurent}
	For any generalized quantum seed ${\bf s}$, we have
	$$\mathscr{A}_q(\mathbf{s})\subset \mathscr{U}_q(\mathbf{s}).$$
\end{theorem}

\begin{remark}
	Although \cite[Theorem 3.1]{BCDX} is stated under the assumption that $p_i=1$ for all $i\in I$, their proof applies verbatim in the more general setting considered here.
\end{remark}

\subsection{Frobenius maps for generalized quantum cluster algebras}\label{sec-fro-cluster}

In this section, we construct Frobenius maps for generalized quantum cluster algebras. To this end, for any positive integers $N$ and $\ell$, we first define a sequence of generalized Chebyshev polynomials 
\[
T_{\ell,N}^{(k)} \in \mathbb Z[h_1,\dots,b_{N-1}],\qquad k=0,\dots,N,
\]
(see Definition~\ref{def:generalizedCP}) using a quantum analogue of the multinomial expansion given in Theorem \ref{thm:expan}. In particular, when $N=2$ and $k=1$, these polynomials become to the Chebyshev polynomials of the first kind $T_\ell$ (see \eqref{Chebyshev-polynomial}). We then employ these generalized Chebyshev polynomials to construct, for a given generalized quantum seed $\mathbf s$ of rank $n$, its $\ell$-th deformed generalized quantum seed 
\[
\hat {\mathbf{s}}=\bigl(\hat A,\hat{\mathbf{h}},\hat\Lambda,\hat{\widetilde{B}}\bigr), \quad \text{(see Definition \ref{def:deformedseed})}
\]
where
\begin{enumerate}[label=(\arabic*)]
	\item $\hat A_i=A_i^\ell$ for $i=1,\dots,m$;
	\item for $i=1,\dots,n$,
	\[
	\hat{\mathbf{h}}_i=\{ \hat h_{i,0},\hat h_{i,1},\dots,\hat h_{i,\ell_i}\},
	\quad
	\hat h_{i,k}:=T_{\ell,\ell_i}^{(k)}\bigl(h_{i,1},\dots,h_{i,\ell_i-1}\bigr)
	\ \ (0\le k\le \ell_i);
	\]
	\item $\hat \Lambda=\Lambda$;
	\item $\hat{\widetilde{B}}=\widetilde{B}$.
\end{enumerate}

Let $\overline{\mathcal X}$ denote the set of cluster variables in $\mathscr A_q(\mathbf s)$. Under the assumption that $q^{d_i/\ell_i}$ has order $\ell$ for every $i\in\{1,\dots,n\}$, Proposition~\ref{pro:compatible} implies that the quantum cluster variables of $\mathscr{A}_\eta(\hat{\mathbf s})$ with $\eta^{\frac 1 2}=q^{\frac{\ell^2}{2}}$ are precisely
\[
\widehat{\overline{\mathcal X}}:=\{X^\ell \mid X\in \overline{\mathcal X}\}.
\]
Consequently, the injective Frobenius map from the generalized quantum cluster algebra $\mathscr A_\eta(\hat{\mathbf s})$ to $\mathscr A_q(\mathbf s)$ follows immediately (Theorem \ref{thm-Fro-cluster}).

A particular case where $\ell_i \le 2$ for all $i \in \{1,\dots,n\}$ is addressed in Corollary~\ref{cor-Fro-cluster}. This result, together with the generalized quantum cluster algebra structures of the reduced LRY skein algebras established in Section~\ref{sec-skein=A=U}, will be employed in Section~\ref{Sec-Fro} to construct Frobenius maps for (reduced) LRY skein algebras.

The {\bf  Chebyshev polynomials of the  first kind} $T_\ell\in \mathbb Z[x]$ is defined as follows: 
\begin{align}\label{Chebyshev-polynomial}
	T_0=2, T_1=x,T_{n+1}=xT_n-T_{n-1}, n\geq 1.
\end{align}

{Lemmas~\ref{lem-root-Tl} and~\ref{lem:expan1} record well-known identities for root of unity commuting variables. They will be used later.}

\begin{lemma}[{\cite[Corollary~3.2]{le2019kauffman}; \cite[Equation~(2)]{bonahon2019miraculous}}]\label{lem-root-Tl}
	Let $x$ and $y$ be two indeterminates such that
	$xy=\omega yx,$
	where $\omega$ is a primitive $\ell$-th root of unity. Then
	\[
	T_\ell\!\left(x+x^{-1}+y\right)
	=
	x^\ell+x^{-\ell}+y^\ell,
	\qquad n\ge 1.
	\]
\end{lemma}

\begin{lemma}\label{lem:expan1}
	Let $x$ and $y$ be two indeterminates such that
	$xy=\omega yx$. Assume that the order of $\omega$ is $\ell$. Then
	\[
	(x+y)^\ell=x^\ell +y^\ell.
	\]
\end{lemma}

The following theorem provides the basis for defining the polynomial $T^{(k)}_{\ell,N}$.

\begin{theorem}\label{thm:expan}
	Let $A$ be an associative $\mathfrak R$-algebra,
	and let $\omega \in \mathfrak R$ be a primitive $\ell$-th root of unity such that
	$\omega$ has a square root $\omega^{1/2}\in\mathfrak R$.
	Assume that $x_0,\dots,x_N \in A$ satisfy
	\[
	x_i x_j = \omega^{\,j-i} x_j x_i
	\qquad (0 \le i < j \le N).
	\]
	Let $S=x_0+x_1+\cdots+x_N$.
	We have
	\[
	S^\ell
	=
	\sum_{\substack{m_0,\dots,m_N \ge 0\\ m_0+\cdots+m_N=\ell\\
			\sum_{j=0}^N j\,m_j \equiv 0 \pmod{\ell}}}
	C_{m_0,\dots,m_N}(\omega)\,
	[x_0^{m_0}\cdots x_N^{m_N}]_{\rm norm}
	\in \bigoplus \mathbb Z[\omega^{\pm \frac{1}{2}}][x_0^{m_0}\cdots x_N^{m_N}]_{\rm norm},
	\]
	where
	\begin{align}\label{def-Cmw}
		C_{m_0,\dots,m_N}(\omega)
		=
		\sum_{(i_1,\dots,i_\ell)\in \mathcal W(m_0,\dots,m_N)}
		\omega^{\frac12\sum_{1\le a<b\le \ell}(i_b-i_a)},
	\end{align}
	and $\mathcal W(m_0,\dots,m_N)$ denotes the set of all words
	$(i_1,\dots,i_\ell)$ in $\{0,\dots,N\}$ such that the letter
	$j$ occurs exactly $m_j$ times. In particular,
	$C_{m_0,\dots,m_N}(\omega)\in \mathbb Z.$
\end{theorem}

\begin{proof}
	
	Expand
	\[
	S^\ell=\sum_{(i_1,\dots,i_\ell)\in\{0,\dots,N\}^\ell} x_{i_1}x_{i_2}\cdots x_{i_\ell}.
	\]
	For each $(i_1,\dots,i_\ell)$, we have $x_{i_1}x_{i_2}\cdots x_{i_\ell}=\omega^{\frac12\sum_{1\le a<b\le \ell}(i_b-i_a)}[x_{i_1}x_{i_2}\cdots x_{i_\ell}]_{\rm norm}$.
	
	Thus, we have
	\[
	S^\ell
	=
	\sum_{\substack{m_0,\dots,m_N\ge 0\\ m_0+\cdots+m_N=\ell}}
	C_{m_0,\dots,m_N}(\omega)
	[x_0^{m_0}\cdots x_N^{m_N}]_{\rm norm},
	\]
	where
	\begin{equation}\label{proof-C-mi}
		\begin{aligned}
			C_{m_0,\dots,m_N}(\omega)
			&= \sum_{(i_1,\dots,i_\ell)\in \mathcal W(m_0,\dots,m_N)}
			\omega^{\frac12\sum_{1\le a<b\le \ell}(i_b-i_a)} \\
			&= \sum_{(i_1,\dots,i_\ell)\in \mathcal W(m_0,\dots,m_N)}
			\omega^{\frac12\sum_{r=1}^{\ell}(2r-\ell-1)i_r}.
		\end{aligned}
	\end{equation}
	
	Denote $\xi:=\omega^{\frac{1}{2}}$. Then $\xi^\ell=(-1)^{\ell+1}$. Let
	\begin{equation}\label{eq:xi}
		\xi_r:=\xi^{2r-\ell-1}, (r=1,\dots,\ell).   
	\end{equation}
	Then
	\[
	\omega^{\frac12(2r-\ell-1)i}=\xi_r^{\,i}
	\qquad (i\in\mathbb Z),
	\]
	and $\xi_1,\dots,\xi_\ell$ are precisely the roots of
	\[
	t^\ell-\xi^\ell=0.
	\]
	We now introduce commuting variables $u_0,\dots,u_N$ and consider
	\begin{equation}\label{eq:F}
		F(u_0,\dots,u_N)
		:=
		\prod_{r=1}^{\ell}(u_0+u_1\xi_r+u_2\xi_r^2+\cdots+u_N\xi_r^N).
	\end{equation}
	
	Expanding this product, and choosing from the $r$-th factor the term
	$u_{i_r}\xi_r^{i_r}$, we get
	\begin{equation}\label{eq:F1}
		F(u_0,\dots,u_N)
		=
		\sum_{\substack{m_0,\dots,m_N\ge 0\\ m_0+\cdots+m_N=\ell}}
		C_{m_0,\dots,m_N}(\omega)
		u_0^{m_0}\cdots u_N^{m_N}.
	\end{equation}
	
	Denote
	\begin{equation}\label{eq:P}
		P_{\mathbf u}(t):=u_0+u_1 t+u_2 t^2+\cdots+u_N t^N,
	\end{equation}
	then
	\begin{equation}\label{eq:P1}
		F(u_0,\dots,u_N)=\prod_{r=1}^{\ell}P_{\mathbf u}(\xi_r)
		={\rm Res}_t(t^\ell-\xi^\ell,P_{\mathbf u}(t)),
	\end{equation}
	the resultant of $t^\ell-\xi^\ell$ and $P_{\mathbf u}(t)$ (see \cite[Chapter~IV, \S8, Proposition 8.3]{LangAlgebra}).
	Since $t^\ell-\xi^\ell\in\mathbb Z[t]$ and
	$P_{\mathbf u}(t)\in\mathbb Z[u_0,\dots,u_N][t]$, the resultant $F(u_0,\dots,u_N)$ belongs to
	$\mathbb Z[u_0,\dots,u_N]$ (see \cite[Chapter~IV, \S8]{LangAlgebra}). Hence every coefficient
	$C_{m_0,\dots,m_N}(\omega)$ is an integer.
	
	Finally, because the set $\{\xi_1,\dots,\xi_\ell\}$ is stable under
	multiplication by $\omega$, one has
	\[
	F(u_0,\dots,u_N)=F(u_0,\omega u_1,\omega^2 u_2,\dots,\omega^N u_N).
	\]
	Comparing the coefficient of $u_0^{m_0}\cdots u_N^{m_N}$ gives
	\[
	C_{m_0,\dots,m_N}(\omega)
	=
	\omega^{\sigma({\bf m})} C_{m_0,\dots,m_N}(\omega),
	\]
	where $\sigma({\bf m}):=\sum_{j=0}^N j\,m_j$
	
	Since $\omega$ is a primitive $\ell$-th root of unity, this implies that
	\[
	C_{m_0,\dots,m_N}(\omega)=0
	\qquad\text{unless}\qquad
	\sigma({\bf m})=\sum_{j=0}^N j\,m_j\equiv 0\pmod{\ell}.
	\]
	
	The proof is complete.
\end{proof}

\begin{remark}
	It is immediate that ${C}_{(\ell,0,0,\cdots,0)}(\omega)={C}_{(0,0,\cdots,0,\ell)}(\omega)=1$.
\end{remark}

\begin{definition}\label{def:generalizedCP}
	For any positive integers $N$ and $\ell$, we define the \emph{generalized Chebyshev polynomials} (Theorem~\ref{thm:expan})
	\[
	T_{\ell,N}^{(k)}(h_1,\dots,h_{N-1})\in\mathbb Z[h_1,\dots,h_{N-1}],\qquad k=0,1,\dots,N,
	\]
	by
	\[
	T_{\ell,N}^{(k)}(h_1,\dots,h_{N-1})
	:=
	\sum_{\substack{m_0,\dots,m_N\ge 0\\[2pt] m_0+\cdots+m_N=\ell\\[2pt] \sum_{j=0}^N j\,m_j=k\ell}}
	C_{m_0,\dots,m_N}(\omega)\,
	h_1^{m_1}h_2^{m_2}\cdots h_{N-1}^{m_{N-1}},
	\]
	where $\omega$ is a root of unity of order $\ell$ and $C_{m_0,\dots,m_N}$ is defined in \eqref{def-Cmw}.
\end{definition}

Note that $T^{(0)}_{\ell,N} = T^{(N)}_{\ell,N} = 1$.

Assume that $\omega$ is a primitive $\ell$-th root of unity, we note the following identity, which follows immediately by considering the parity of $\ell$ and will be used later:
\begin{equation}
	\label{eq-root-of-unity-sign}
	\omega^{\frac{k\ell(\ell-1)}{2}}
	=
	(-1)^{k(\ell-1)}.
\end{equation}

\def\BC{\mathbb{Z}}
\def\W{X}

\vs

Define
$\BC[\W]
=
\BC[x_1,\dots,x_N]/(x_1\cdots x_N-1).$
For each $1\leq k\leq N-1$, let
\[
e_k
=
\sum_{1\leq i_1<\cdots<i_k\leq N}
x_{i_1}\cdots x_{i_k}
\in \BC[\W]
\]
be the $k$-th elementary symmetric polynomial. It was shown in
\cite{bonahon2024central} that there exists a unique polynomial
$\overline{P}_{\ell,N}^{(k)}
\in
\mathbb{Z}[y_1,\dots,y_{N-1}]$
such that
\begin{align}
	\overline{P}_{\ell,N}^{(k)}(e_1,\dots,e_{N-1})
	=
	e_k(x_1^\ell,\dots,x_N^\ell)
	\in \BC[\W].
\end{align}
Moreover, it is well known that
$\overline{P}_{\ell,2}^{(1)}=T_\ell$.

The following proposition establishes that
\[
T_{\ell,N}^{(k)}=\overline{P}_{\ell,N}^{(k)}
\qquad\text{for all }1\leq k\leq N-1.
\]

\def\bi{\textbf{i}}

\begin{proposition}
	\label{thm-generalized-Chebyshev-Adams}
	Let $N,\ell\in\mathbb{N}_{>0}$.
	Then, for every $k=1,\ldots,N-1$,
	\[
	T_{\ell,N}^{(k)}
	(h_1,\ldots,h_{N-1})
	=
	\overline P_{\ell,N}^{(k)}
	(h_1,\ldots,h_{N-1}).
	\]
	
\end{proposition}

\begin{proof}
	Set
	\[
	e_0=e_N=1
	\]
	and consider the polynomial
	\[
	F(z)
	=
	\sum_{j=0}^{N}(-1)^j e_jz^j
	=
	\prod_{s=1}^{N}(1-x_sz).
	\]
	Since $\omega$ is a primitive $\ell$-th root of unity, we have
	\[
	\prod_{r=0}^{\ell-1}(1-u\omega^r)
	=
	1+ (-1)^{\ell}\omega^{\frac{\ell (\ell -1)}{2}} u^\ell
	\overset{\eqref{eq-root-of-unity-sign}}{=} 1 - u^\ell.
	\]
	It follows that
	\begin{align}
		\prod_{r=0}^{\ell-1}F(\omega^rz)
		&=
		\prod_{r=0}^{\ell-1}
		\prod_{s=1}^{N}(1-x_s\omega^rz)
		\nonumber
		=
		\prod_{s=1}^{N}
		\left(
		\prod_{r=0}^{\ell-1}(1-x_s\omega^rz)
		\right)
		\nonumber\\
		&=
		\prod_{s=1}^{N}(1-x_s^\ell z^\ell)
		=
		\sum_{k=0}^{N}
		(-1)^k
		e_k(x_1^\ell,\ldots,x_N^\ell)z^{k\ell}.
		\label{eq-root-of-unity-factorization}
	\end{align}
	
	We next calculate the coefficient of $z^{k\ell}$ on the
	left-hand side of
	\eqref{eq-root-of-unity-factorization}.
	For a word
	\[
	(i_1,\ldots,i_\ell)\in\{0,\ldots,N\}^{\ell},
	\]
	choosing the term indexed by $i_a$ from
	$F(\omega^{a-1}z)$ gives
	\[
	(-1)^{i_1+\cdots+i_\ell}
	\omega^{\sum_{a=1}^{\ell}(a-1)i_a}
	e_{i_1}\cdots e_{i_\ell}
	z^{i_1+\cdots+i_\ell}.
	\]
	Therefore, the coefficient of $z^{k\ell}$ is
	\begin{equation}
		\label{eq-coefficient-left-hand-side}
		(-1)^{k\ell}
		\sum_{\substack{m_0,\ldots,m_N\geq 0\\
				m_0+\cdots+m_N=\ell\\
				\sum_{j=0}^{N}jm_j=k\ell}}
		A_{m_0,\ldots,m_N}(\omega)
		e_1^{m_1}\cdots e_{N-1}^{m_{N-1}},
	\end{equation}
	where
	\[
	A_{m_0,\ldots,m_N}(\omega)
	:=
	\sum_{(i_1,\ldots,i_\ell)\in\mathcal W(m_0,\ldots,m_N)}
	\omega^{\sum_{a=1}^{\ell}(a-1)i_a}.
	\]

	We now compare $A_{m_0,\ldots,m_N}(\omega)$ with
	$ C_{m_0,\ldots,m_N}(\omega)$.
	For
	$(i_1,\ldots,i_\ell)\in
	\mathcal W(m_0,\ldots,m_N)$,
	we have
	\begin{align}\label{eq-inversion-position-statistics}
		\frac{1}{2}\sum_{a=1}^{\ell}
		\left(2a-\ell-1\right)i_a
		=
		\sum_{a=1}^{\ell}(a-1)i_a
		-
		\frac{\ell-1}{2}
		\sum_{a=1}^{\ell}i_a=\sum_{a=1}^{\ell}(a-1)i_a
		-
		\frac{k\ell(\ell-1)}{2}.
	\end{align}
	Together with \eqref{proof-C-mi},
	it follows that
	\begin{align*}
		C_{m_0,\ldots,m_N}(\omega)
		=
		\omega^{-\frac{k\ell(\ell-1)}{2}}
		A_{m_0,\ldots,m_N}(\omega)
		\overset{\eqref{eq-root-of-unity-sign}}{=}
		(-1)^{k(\ell-1)}
		A_{m_0,\ldots,m_N}(\omega).
	\end{align*}

	Comparing the coefficients of $z^{k\ell}$ in
	\eqref{eq-root-of-unity-factorization} and
	\eqref{eq-coefficient-left-hand-side}, we obtain
	\begin{align*}
		e_k(x_1^\ell,\ldots,x_N^\ell)
		&=
		(-1)^{k(\ell+1)}
		\sum_{\substack{m_0,\ldots,m_N\geq 0\\
				m_0+\cdots+m_N=\ell\\
				\sum_{j=0}^{N}jm_j=k\ell}}
		A_{m_0,\ldots,m_N}(\omega)
		e_1^{m_1}\cdots e_{N-1}^{m_{N-1}}\\
		&=
		\sum_{\substack{m_0,\ldots,m_N\geq 0\\
				m_0+\cdots+m_N=\ell\\
				\sum_{j=0}^{N}jm_j=k\ell}}
		C_{m_0,\ldots,m_N}(\omega)
		e_1^{m_1}\cdots e_{N-1}^{m_{N-1}}.
	\end{align*}
	By definition, the last expression is
	\[
	T_{\ell,N}^{(k)}
	(e_1,\ldots,e_{N-1}).
	\]
	Therefore,
	\[
	T_{\ell,N}^{(k)}
	(e_1,\ldots,e_{N-1})
	=
	e_k(x_1^\ell,\ldots,x_N^\ell)
	=
	\Psi_\ell(e_k).
	\]
	The defining property and uniqueness of
	$\overline P_{\ell,N}^{(k)}$ now imply that
	\[
	T_{\ell,N}^{(k)}
	(h_1,\ldots,h_{N-1})
	=
	\overline P_{\ell,N}^{(k)}
	(h_1,\ldots,h_{N-1}).
	\]
	The cases $k=0,N$ follow from
	\[
	e_0(x_1^\ell,\ldots,x_N^\ell)=e_N(x_1^\ell,\ldots,x_N^\ell)=1.
	\]
	
	The proof is complete.
\end{proof}

\begin{remark}
	The definition of $T_{\ell,N}^{(k)}$ is motivated by
	Theorem~\ref{thm:expan} and is tailored to the proof of
	Theorem~\ref{thm-Fro-cluster}. This is why we do not simply define
	$T_{\ell,N}^{(k)}$ to be $\overline{P}_{\ell,N}^{(k)}$.
	Moreover, $\overline{P}_{\ell,N}^{(k)}$ is defined implicitly through
	an existence result, and it is generally difficult to write down an
	explicit formula for this polynomial. Definition~\ref{def:generalizedCP}
	and Proposition~\ref{thm-generalized-Chebyshev-Adams} therefore provide
	a concrete expression for $\overline{P}_{\ell,N}^{(k)}$.
\end{remark}

Note that $e_k = e_{N-k}(x_1^{-1},\cdots,x_N^{-1})$. The following lemma follows directly from Proposition~\ref{thm-generalized-Chebyshev-Adams}. 

\begin{lemma}\label{lem:antisym}
	$T_{\ell,N}^{(k)}(h_1,\dots,h_{N-1})=T_{\ell,N}^{N-k}(h_{N-1},\dots,h_2,h_{1})$ for any $k=0,1,\cdots, N$.
\end{lemma}

Lemma~\ref{lem:antisym} allows us to give the following definition.

\begin{definition}\label{def:deformedseed}
	For any generalized quantum seed $\mathbf{s}=(A,\mathbf{h},\Lambda,\widetilde{B})$ as in Definition~\ref{def:generalizedQS} and a given positive integer $\ell$, we define its \textbf{$\ell$-th deformed} generalized quantum seed by
	\[
	\hat {\mathbf{s}}:=\bigl(\hat A,\hat{\mathbf{h}},\hat\Lambda,\hat{\widetilde{B}}\bigr),
	\]
	where
	\begin{enumerate}[label=(\arabic*)]
		\item $\hat A_i=A_i^\ell$ for $i=1,\dots,m$;
		\item for $i=1,\dots,n$,
		\[
		\hat{\mathbf{h}}_i=\{ \hat h_{i,0},\hat h_{i,1},\dots,\hat h_{i,\ell_i}\},
		\quad
		\hat h_{i,k}:=T_{\ell,\ell_i}^{(k)}\bigl(h_{i,1},\dots,h_{i,\ell_i-1}\bigr)
		\ \ (0\le k\le \ell_i);
		\]
		\item $\hat \Lambda=\Lambda$;
		\item $\hat{\widetilde{B}}=\widetilde{B}$.
	\end{enumerate}
	We call $\mathscr{A}_{\eta}(\hat{\mathbf{s}})$ the \textbf{$\ell$-th deformed} generalized quantum cluster algebra of $\mathscr{A}_{q}(\mathbf{s})$, where $\eta^{\frac{1}{2}}=q^{\frac{\ell^2}{2}}$.
\end{definition}

The following result follows from Theorem~\ref{thm:expan} by letting $\omega=q^{d_i/\ell_i}$ and Definition \ref{def:deformedseed}.
It demonstrates that the deformation is compatible with mutation; specifically,  $\mu_i(\hat A_i)=\widehat{\mu_i(A)}$ under certain condition.

\begin{proposition}\label{pro:compatible}
	Let $\mathscr{A}_q(\mathbf{s})$ be a generalized quantum cluster algebra. For a given positive integer $\ell$, assume that $q^{d_i/\ell_i}$ has order $\ell$ for all $i\in\{1,2,\dots,n\}$. Then for any $i\in\{1,2,\dots,n\}$, we have
	\[
	\begin{aligned}
		(\mu_i A_i)^\ell
		&=\sum_{r=0}^{\ell_i} \hat h_{i,r}\,A\Bigl(\sum_{j=1}^{m}\bigl(r[\ell\beta_{ji}]_++(\ell_i-r)[-\ell\beta_{ji}]_+\bigr)e_j- \ell e_i\Bigr)\\
		&=\sum_{r=0}^{\ell_i} \hat h_{i,r}\,A^\ell\Bigl(\sum_{j=1}^{m}\bigl(r[\beta_{ji}]_++(\ell_i-r)[-\beta_{ji}]_+\bigr)e_j- e_i\Bigr).
	\end{aligned}
	\]
\end{proposition}

Denote by $\overline{\mathcal{X}}$ the set of cluster variables of $\mathscr{A}_q(\mathbf{s})$. Then from Proposition~\ref{pro:compatible} we see that
\[
\widehat{\overline{\mathcal{X}}}:=\{\hat X:=X^\ell\mid X\in\overline{\mathcal{X}}\}
\]
is the set of cluster variables of $\mathscr{A}_\eta(\hat{\mathbf{s}})$.

As an immediate consequence of Proposition~\ref{pro:compatible}, we obtain the following theorem, which establishes the Frobenius homomorphism for generalized quantum cluster algebras. This is the main result of this section.

\begin{theorem}\label{thm-Fro-cluster}
	Let $\mathscr{A}_q(\mathbf{s})$ be a generalized quantum cluster algebra. For a given positive integer $\ell$, assume that $q^{d_i/\ell_i}$ has order $\ell$ for all $i\in\{1,2,\dots,n\}$. Let $\eta^{1/2} = q^{\ell^2/2}$.
	There exists an injective $\mathfrak R$-algebra homomorphism 
	\[\mathscr F:\mathscr{A}_{\eta}(\hat{\bf s})\to\mathscr{A}_q(\mathbf{s})\] 
	sending $\hat X=X^\ell$ to $X^\ell$ for any cluster variable $X\in\overline{\mathcal{X}}$ and $\hat h_{i,k}$ to $\hat h_{i,k}=T_{\ell,\ell_i}^k(h_{i,1},\cdots, h_{i,\ell_i-1})$ for all $i\in\{1,\dots,n\}$ and $k\in\{0,\dots,\ell_i\}$.
\end{theorem}

We now consider the particular case where  $\ell_i\leq 2$ for all $i$, with $I=\{i\mid \ell_i=2\}$ and $\mathbf{h}_i=\{1,p_i,1\}$. From Definition \ref{def:deformedseed}, we have
\begin{equation}
	{\hat{\mathbf{h}}}_i =
	\begin{cases}
		\{1,1\} & \text{if } i\notin I,\\[4pt]
		\{1,\hat p_i,1\} & \text{if } i\in I,
	\end{cases}
\end{equation}
where $\hat p_i=T_{\ell,2}^{(1)}(p_i) = T_\ell (p_i)$. Consequently, Theorem \ref{thm-Fro-cluster} becomes to the following corollary.

\begin{corollary}\label{cor-Fro-cluster}
	In the special case where  $\ell_i\leq 2$ for all $i$, with $I=\{i\mid \ell_i=2\}$ and $\mathbf{h}_i=\{1,p_i,1\}$, assume further that $q^{d_i/\ell_i}$ has order $\ell$ for every $i\in\{1,\dots,n\}$. Then there exists an injective $\mathfrak R$-algebra homomorphism
	\[
	\mathscr F:\mathscr{A}_{\eta}(\hat{\bf s})\longrightarrow \mathscr{A}_q(\mathbf{s})
	\]
	which sends $\hat X=X^\ell$ to $X^\ell$ for each cluster variable $X\in\overline{\mathcal{X}}$, and sends $\hat p_i$ to $T_\ell(p_i)$ for all $i\in I$.
\end{corollary}

\section{Generalized quantum cluster algebra structures on reduced L{\^e}-Roger-Yang skein algebras}\label{sec-skein=A=U}

Let $\fS$ be a quasi-triangulable punctured surface (Definition~\ref{def-quasitri}). In this section, we show that the reduced Lê--Roger--Yang skein algebra  $\overline {\cS}_q(\fS)$ admits a generalized quantum cluster algebra structure. To this end, we first construct, for any quasi-triangulation $\Delta$, an extended exchange matrix $B^\Delta$ in Section~\ref{sec:extendedmatrix}, which we show in Proposition~\ref{prop:compatible} to be compatible with the orientation matrix $\Lambda^\Delta$ (Definition~\ref{def:orientationmatrix}).
For each admissible ideal arc $\gamma \in \mathring\Delta$, define its {\bf weight} $w(\gamma)$ to be:
\begin{equation}
	w(\gamma)=\begin{dcases}
		\frac{1}{2} & \mbox{ if }\gamma \mbox{ is a pending arc},\\
		2 & \mbox{ if }\gamma \mbox{ is a loop enclosing an interior puncture},\\
		1 & \mbox{ otherwise.}
	\end{dcases}
\end{equation}
Set $d_\gamma = \max\{1, w(\gamma)\}$, and define the coefficient string
\begin{equation*}
	\mathbf{h}_\gamma =
	\begin{cases}
		(1, 1) & \text{if } d_\gamma = 1, \\
		(1,\alpha_p, 1) & \text{if } d_\gamma = 2 \text{ and } \gamma \text{ is the loop enclosing } p \in \mathring{\mathcal P},
	\end{cases}
\end{equation*}
where $\mathring{\mathcal P}$ is the set of interior punctures and $\alpha_p$ is the peripheral loop around $p$.
It follows that 
\[
\mathbf s(\Delta)=\bigl(A^\Delta,\mathbf h^\Delta,\Lambda^\Delta,B^\Delta\bigr),
\]
with $A^\Delta=(A_\gamma\mid \gamma\in \Delta)$ and $\mathbf h^\Delta=(\mathbf{h}_\gamma\mid \gamma\in\mathring\Delta)$, is a generalized quantum seed inside the skew-field of fractions ${\rm Frac}(\overline {\cS}_q(\fS))$ of $\overline {\cS}_q(\fS)$. We then show that seed mutations are compatible with flips of quasi-triangulations in Proposition~\ref{prop:mutation}. Finally, we prove that the reduced Lê--Roger--Yang skein algebra $\overline {\cS}_q(\fS)$ coincides with the generalized quantum cluster algebra $\mathscr A_q(\mathbf s(\Delta))$ in Theorem~\ref{thm:skein=A}.

\subsection{Extended exchange matrix and compatibility}\label{sec:extendedmatrix}

Denote the two half-edges (Definition \ref{def-quasi}) of an admissible ideal arc $\gamma$ by $\partial_1(\gamma)$ and $\partial_2(\gamma)$.

Let $\Delta$ be a quasi-ideal triangulation (Definition \ref{def-quasitri}) of $\fS$. For any $\alpha,\beta\in \Delta$, define the matrix entry $b^\Delta_{\alpha,\beta}$ as follows:

If $w(\beta)\neq \frac{1}{2}$, set
\begin{equation}\label{def1-BB}
	b^\Delta_{\alpha,\beta} =
	\sum_{i,j\in \{1,2\}} \left\{
	\begin{array}{c|c}
		0 & \text{if $\partial_i(\alpha)$ and $\partial_j(\beta)$ are not incident} \\
		\text{max} \{w(\beta),\frac{w(\beta)}{w(\alpha)}\} & \text{if $\partial_i(\alpha)$ is immediately clockwise to $\partial_j(\beta)$ in $\Delta$} \\
		-\text{max} \{w(\beta),\frac{w(\beta)}{w(\alpha)}\} & \text{if $\partial_j(\beta)$ is immediately clockwise to $\partial_i(\alpha)$ in $\Delta$} \\
	\end{array}\right\}
\end{equation}

If $w(\beta)=\frac{1}{2}$, set 
\begin{equation}\label{def2-BB}
	b^\Delta_{\alpha,\beta} = \sum_{i,j\in \{1,2\}} \left\{
	\begin{array}{c|c}
		0 & \text{if $\partial_i(\alpha)$ and $\partial_j(\beta)$ have different endpoints} \\
		\text{max}\{1, \frac{1}{w(\alpha)}\} & \text{if $\partial_i(\alpha)$ is immediately clockwise to $\partial_j(\beta)$ in $\Delta$} \\
		-\text{max}\{1, \frac{1}{w(\alpha)}\} & \text{if $\partial_j(\beta)$ is immediately clockwise to $\partial_i(\alpha)$ in $\Delta$} \\
	\end{array}\right\}
\end{equation}

We call the matrix $B^\Delta=(b^\Delta_{\alpha,\beta})_{\alpha\in \Delta, \beta\in \mathring{\Delta}}$ the {\bf skew-adjacency matrix} of $\Delta$.

\def\MDe{\mathring{\Delta}}

Let $\alpha\in\mathring\Delta$. There exists a unique admissible ideal arc $\alpha'\neq \alpha$ such that $\Delta'=(\Delta\setminus\{\alpha\})\sqcup \{\alpha'\}$ is also a quasi-triangulation (Definition \ref{def-quasitri}) of $\fS$. We say that $\Delta$ and $\Delta'$ are related by a {\bf flip} at $\alpha$.

\begin{proposition}\label{prop:compatible}
	Let $\fS$ be a triangulable marked orbifold with a triangulation $\Delta$.
	The pair of matrices $(\Lambda^\Delta, B^\Delta)$ is compatible, i.e.,  
	\begin{equation*}
		(B^\Delta)^\top\Lambda^\Delta=
		\begin{pmatrix}
			D & O\\
			* & *
		\end{pmatrix},
	\end{equation*}
	where $D=\text{diag}((d_{\alpha})_{\alpha\in \mathring\Delta})$ with $d_\alpha=4w(\alpha)$ and $O$ is the zero matrix.
	Moreover, for any $\Delta, \Delta'$ are related by a flip at $\gamma\in \mathring\Delta$, we have  
	\[B^{\Delta'}=\mu_\gamma( B^\Delta),\] 
	where $\mu_\gamma$ is defined in \eqref{eq-mutation-Q}.
\end{proposition}

\begin{proof}
	Let $\alpha$ be an interior arc in $\Delta$. For any $\beta\in \Delta$, consider the $(\alpha,\beta)$-entry of the matrix $(B^\Delta)^\top\Lambda^\Delta$, we have 
	\begin{equation*}
		((B^\Delta)^\top\Lambda^\Delta)_{\alpha,\beta}=\sum_{\gamma\in \Delta} b^\Delta_{\gamma,\alpha} \Lambda_{\gamma,\beta}.
	\end{equation*}
	
	{\bf Case 1}. $w(\alpha)=\frac{1}{2}$. Suppose $\alpha_1,\alpha_2$ are the arcs in $\Delta$ such that $(\alpha,\alpha_1,\alpha_2)$ forms the unique clockwise triangle with $\alpha$ as one of its sides. Then we have $b_{\alpha_i,\alpha}^\Delta=(-1)^i\text{max} \{1, \frac{1}{w(\alpha_i)}\}$ for $i\in\{1,2\}$ and $Q_{\gamma,\alpha}^\Delta=0$ for any $\gamma\neq \alpha_1,\alpha_2$. 
	
	If $\beta\neq \alpha,\alpha_1,\alpha_2$, then we have 
	$$\Lambda_{\alpha_1,\beta}=
	\begin{cases}
		\Lambda_{\alpha_2,\beta} & \mbox{ if $w(\alpha_1)=w(\alpha_2)=1$ or $w(\alpha_1)w(\alpha_2)=2$},\\
		2\Lambda_{\alpha_2,\beta} & \mbox{ if $w(\alpha_1)=1, w(\alpha_2)=\frac{1}{2}$},\\
		\frac{1}{2}\Lambda_{\alpha_2,\beta} & \mbox{ if $w(\alpha_1)=\frac{1}{2}, w(\alpha_2)=1$}.
	\end{cases}$$ 
	Thus, in all the cases, we have
	\begin{equation*}      ((B^\Delta)^\top\Lambda^\Delta)_{\alpha,\beta}=\sum_{\gamma\in \Delta} b^\Delta_{\gamma,\alpha} \Lambda_{\gamma,\beta}=b^\Delta_{\alpha_1,\alpha} \Lambda_{\alpha_1,\beta}+b^\Delta_{\alpha_2,\alpha} \Lambda_{\alpha_2,\beta}=0.
	\end{equation*}
	
	If $\beta=\alpha_1$, then we have $\Lambda_{\alpha_1,\beta}=\Lambda_{\alpha_2,\beta}=0$.
	Thus, we have
	\begin{equation*}       ((B^\Delta)^\top\Lambda^\Delta)_{\alpha,\beta}=\sum_{\gamma\in \Delta} b^\Delta_{\gamma,\alpha} \Lambda_{\gamma,\beta}=b^\Delta_{\alpha_1,\alpha} \Lambda_{\alpha_1,\beta}+b^\Delta_{\alpha_2,\alpha} \Lambda_{\alpha_2,\beta}=0.
	\end{equation*}
	
	Similarly, $((B^\Delta)^\top\Lambda^\Delta)_{\alpha,\beta}=0$ in case $\beta=\alpha_2$.
	
	If $\beta=\alpha$, then we have
	
	$$\Lambda_{\alpha_1,\beta}=
	\begin{cases}
		\Lambda_{\alpha_2,\beta}-2 & \mbox{ if $w(\alpha_1)=w(\alpha_2)=1$ or $w(\alpha_1)w(\alpha_2)=2$},\\
		0=\Lambda_{\alpha_2,\beta}-1 & \mbox{ if $w(\alpha_1)=1, w(\alpha_2)=\frac{1}{2}$},\\
		-1=\Lambda_{\alpha_2,\beta}-1 & \mbox{ if $w(\alpha_1)=\frac{1}{2}, w(\alpha_2)=1$}.
	\end{cases}$$
	Thus, in all the cases, we have
	\begin{equation*}
		((B^\Delta)^\top\Lambda^\Delta)_{\alpha,\beta}=\sum_{\gamma\in \Delta} b^\Delta_{\gamma,\alpha} \Lambda_{\gamma,\alpha}=b^\Delta_{\alpha_1,\alpha} \Lambda_{\alpha_1,\alpha}+b^\Delta_{\alpha_2,\alpha} \Lambda_{\alpha_2,\alpha}=2.
	\end{equation*}
	
	{\bf Case 2}. $w(\alpha)=2$. Suppose $\alpha_1,\alpha_2$ are the arcs in $\Delta$ such that $(\alpha,\alpha_1,\alpha_2)$  forms the unique clockwise triangle with $\alpha$ as one of its sides. Then we have $b_{\alpha_i,\alpha}^\Delta=(-1)^i\text{max} \{2, \frac{2}{w(\alpha_i)}\}$ for $i\in\{1,2\}$ and $b_{\gamma,\alpha}^\Delta=0$ for any $\gamma\neq \alpha_1,\alpha_2$. 
	
	If $\beta\neq \alpha,\alpha_1,\alpha_2$, then we have 
	$$\Lambda_{\alpha_1,\beta}=
	\begin{cases}
		\Lambda_{\alpha_2,\beta} & \mbox{ if $w(\alpha_1)=w(\alpha_2)=1$ or $w(\alpha_1)w(\alpha_2)=2$},\\
		2\Lambda_{\alpha_2,\beta} & \mbox{ if $w(\alpha_1)=1, w(\alpha_2)=\frac{1}{2}$},\\
		\frac{1}{2}\Lambda_{\alpha_2,\beta} & \mbox{ if $w(\alpha_1)=\frac{1}{2}, w(\alpha_2)=1$}.
	\end{cases}$$ 
	Thus, in all the cases, we have
	\begin{equation*}
		((B^\Delta)^\top\Lambda^\Delta)_{\alpha,\beta}=\sum_{\gamma\in \Delta} b^\Delta_{\gamma,\alpha} \Lambda_{\gamma,\beta}=b^\Delta_{\alpha_1,\alpha} \Lambda_{\alpha_1,\beta}+b^\Delta_{\alpha_2,\alpha} \Lambda_{\alpha_2,\beta}=0.
	\end{equation*}
	
	If $\beta=\alpha_1$, then we have $\Lambda_{\alpha_1,\beta}=\Lambda_{\alpha_2,\beta}=0$.
	Thus, we have
	\begin{equation*}
		((B^\Delta)^\top\Lambda^\Delta)_{\alpha,\beta}=\sum_{\gamma\in \Delta} b^\Delta_{\gamma,\alpha} \Lambda_{\gamma,\beta}=b^\Delta_{\alpha_1,\alpha} \Lambda_{\alpha_1,\beta}+b^\Delta_{\alpha_2,\alpha} \Lambda_{\alpha_2,\beta}=0.
	\end{equation*}
	
	Similarly, $((B^\Delta)^\top\Lambda^\Delta)_{\alpha,\beta}=0$ in case $\beta=\alpha_2$.
	
	If $\beta=\alpha$, then we have
	
	$$\Lambda_{\alpha_1,\beta}=
	\begin{cases}
		\Lambda_{\alpha_2,\beta}-4 & \mbox{ if $w(\alpha_1)=w(\alpha_2)=1$ or $w(\alpha_1)w(\alpha_2)=2$},\\
		0=\Lambda_{\alpha_2,\beta}-2 & \mbox{ if $w(\alpha_1)=1, w(\alpha_2)=\frac{1}{2}$},\\
		-2=\Lambda_{\alpha_2,\beta}-2 & \mbox{ if $w(\alpha_1)=\frac{1}{2}, w(\alpha_2)=1$}.
	\end{cases}$$
	Thus, in all the cases, we have
	\begin{equation*}
		((B^\Delta)^\top\Lambda^\Delta)_{\alpha,\beta}=\sum_{\gamma\in \Delta} b^\Delta_{\gamma,\alpha} \Lambda_{\gamma,\alpha}=b^\Delta_{\alpha_1,\alpha} \Lambda_{\alpha_1,\alpha}+b^\Delta_{\alpha_2,\alpha} \Lambda_{\alpha_2,\alpha}=8.
	\end{equation*}
	
	{\bf Case 3}. $w(\alpha)=1$. Suppose $\alpha_1,\alpha_2,\alpha_3,\alpha_4$ are arcs in $\Delta$ such that $(\alpha,\alpha_1,\alpha_2)$ and $(\alpha,\alpha_3,\alpha_4)$ form the clockwise triangles with $\alpha$ as one of their sides. We may assume that $\alpha_1\neq \alpha_3$ and $\alpha_2\neq \alpha_4$, as the other cases follow by a similar argument. Then we have $b_{\alpha_i,\alpha}^\Delta=(-1)^i\text{max} \{1, \frac{1}{w(\alpha_i)}\}$ for $i\in\{1,2,3,4\}$ and $b_{\gamma,\alpha}^\Delta=0$ for any $\gamma\neq \alpha_1,\alpha_2,\alpha_3,\alpha_4$. 
	
	{\bf Case 3.1}. $(\alpha,\alpha_1,\alpha_2)$ and $(\alpha,\alpha_3,\alpha_4)$ are weight $1$ normal triangles, i.e., $w(\alpha_1)=w(\alpha_2)=w(\alpha_3)=w(\alpha_4)=1$. We only consider the case where the shared endpoints of the pairs $(\alpha_1,\alpha_2)$, $(\alpha_2,\alpha_3)$, $(\alpha_3,\alpha_4)$, and $(\alpha_4,\alpha_1)$ are different, as the other cases can be proved similarly.
	
	If $\beta\neq \alpha,\alpha_1,\alpha_2,\alpha_3,\alpha_4$, then we have $-\Lambda_{\alpha_1,\beta}+\Lambda_{\alpha_2,\beta}-\Lambda_{\alpha_3,\beta}+\Lambda_{\alpha_4,\beta}=0$. Thus,
	\begin{equation*}
		((B^\Delta)^\top\Lambda^\Delta)_{\alpha,\beta}=\sum_{\gamma\in \Delta} b^\Delta_{\gamma,\alpha} \Lambda_{\gamma,\beta}=-\Lambda_{\alpha_1,\beta}+\Lambda_{\alpha_2,\beta}-\Lambda_{\alpha_3,\beta}+\Lambda_{\alpha_4,\beta}=0.
	\end{equation*}
	
	If $\beta=\alpha_1$, then we have $\Lambda_{\alpha_1,\beta}=\Lambda_{\alpha_3,\beta}=0$ and $\Lambda_{\alpha_4,\beta}=-\Lambda_{\alpha_2,\beta}=1$.
	Thus, we have
	\begin{equation*}
		((B^\Delta)^\top\Lambda^\Delta)_{\alpha,\beta}=\sum_{\gamma\in \Delta} b^\Delta_{\gamma,\alpha} \Lambda_{\gamma,\beta}=\sum_{i=1}^4 b^\Delta_{\alpha_i,\alpha}\Lambda_{\alpha_i,\beta}=0
	\end{equation*}
	
	Similarly, $((B^\Delta)^\top\Lambda^\Delta)_{\alpha,\beta}=0$ in case $\beta=\alpha_2, \alpha_3$ or $\alpha_4$.
	
	If $\beta=\alpha$, then we have
	$\Lambda_{\alpha_2,\beta}=\Lambda_{\alpha_4,\beta}=-\Lambda_{\alpha_1,\beta}=-\Lambda_{\alpha_3,\beta}=1$.
	Thus, 
	\begin{equation*}
		((B^\Delta)^\top\Lambda^\Delta)_{\alpha,\beta}=\sum_{\gamma\in \Delta} b^\Delta_{\gamma,\alpha} \Lambda_{\gamma,\alpha}=\sum_{i=1}^4 b^\Delta_{\alpha_i,\alpha}\Lambda_{\alpha_i,\beta}=4.
	\end{equation*}
	
	{\bf Case 3.2}. Exactly one of $(\alpha,\alpha_1,\alpha_2)$ and $(\alpha,\alpha_3,\alpha_4)$ is a weight $1$ normal triangle; we may assume it is $(\alpha,\alpha_1,\alpha_2)$. 
	
	{\bf Case 3.2.1}. Exactly one of $\alpha_3$ and $\alpha_4$ has weight $1$; we may assume that $w(\alpha_3)=1$. We only consider the case where the shared endpoints of the pairs $(\alpha_1,\alpha_2)$, $(\alpha_2,\alpha)$, and $(\alpha,\alpha_1)$ are different, as the other cases can be proved similarly.
	
	If $\beta\neq \alpha,\alpha_1,\alpha_2,\alpha_3,\alpha_4$, then we have 
	$$\begin{cases}
		-\Lambda_{\alpha_1,\beta}+\Lambda_{\alpha_2,\beta}-\Lambda_{\alpha_3,\beta}+\Lambda_{\alpha_4,\beta}=0 & \mbox{ if $w(\alpha_4)=2$,}\\
		-\Lambda_{\alpha_1,\beta}+\Lambda_{\alpha_2,\beta}-\Lambda_{\alpha_3,\beta}+2\Lambda_{\alpha_4,\beta}=0 & \mbox{ if $w(\alpha_4)=\frac{1}{2}$.}  
	\end{cases}$$
	Thus,
	\begin{equation*}
		((B^\Delta)^\top\Lambda^\Delta)_{\alpha,\beta}=\sum_{\gamma\in \Delta} b^\Delta_{\gamma,\alpha} \Lambda_{\gamma,\beta}=-\Lambda_{\alpha_1,\beta}+\Lambda_{\alpha_2,\beta}-\Lambda_{\alpha_3,\beta}+{\rm max}\{1, \frac{1}{w(\alpha_4)}\}\Lambda_{\alpha_4,\beta}=0.
	\end{equation*}
	
	If $\beta=\alpha_1$, then we have $\Lambda_{\alpha_1,\beta}=0$, $\Lambda_{\alpha_3,\beta}=-\Lambda_{\alpha_2,\beta}=1$ and $b^\Delta_{\alpha_4,\alpha}\Lambda_{\alpha_4,\beta}=2$.
	Thus, we have
	\begin{equation*}
		((B^\Delta)^\top\Lambda^\Delta)_{\alpha,\beta}=\sum_{\gamma\in \Delta} b^\Delta_{\gamma,\alpha} \Lambda_{\gamma,\beta}=-\Lambda_{\alpha_1,\beta}+\Lambda_{\alpha_2,\beta}-\Lambda_{\alpha_3,\beta}+b^\Delta_{\alpha_4,\alpha}\Lambda_{\alpha_4,\beta}=0.
	\end{equation*}
	
	Similarly, $((B^\Delta)^\top\Lambda^\Delta)_{\alpha,\beta}=0$ in case $\beta=\alpha_2, \alpha_3$ or $\alpha_4$.
	
	If $\beta=\alpha$, then we have
	$\Lambda_{\alpha_2,\beta}=-\Lambda_{\alpha_1,\beta}=1$, $\Lambda_{\alpha_3,\beta}=0$, and $b^\Delta_{\alpha_4,\alpha}\Lambda_{\alpha_4,\beta}=2$.
	Thus, 
	\begin{equation*}
		((B^\Delta)^\top\Lambda^\Delta)_{\alpha,\beta}=\sum_{\gamma\in \Delta} b^\Delta_{\gamma,\alpha} \Lambda_{\gamma,\alpha}=-\Lambda_{\alpha_1,\beta}+\Lambda_{\alpha_2,\beta}-\Lambda_{\alpha_3,\beta}+b^\Delta_{\alpha_4,\alpha}\Lambda_{\alpha_4,\beta}=4.
	\end{equation*}
	
	{\bf Case 3.2.2}. Neither $\alpha_3$ nor $\alpha_4$ has weight $1$. Then $\alpha$ is a loop. We only consider the case where $\alpha_1$ is not a loop, as the other case can be proved similarly.
	
	If $\beta\neq \alpha,\alpha_1,\alpha_2,\alpha_3,\alpha_4$, then we have $\Lambda_{\alpha_1,\beta}=\Lambda_{\alpha_2,\beta}$ and ${\rm max}\{1, \frac{1}{w(\alpha_3)}\}\Lambda_{\alpha_3,\beta}={\rm max}\{1, \frac{1}{w(\alpha_4)}\}\Lambda_{\alpha_4,\beta}$. Thus,
	\begin{equation*}
		((B^\Delta)^\top\Lambda^\Delta)_{\alpha,\beta}=\sum_{\gamma\in \Delta} b^\Delta_{\gamma,\alpha} \Lambda_{\gamma,\beta}=-\Lambda_{\alpha_1,\beta}+\Lambda_{\alpha_2,\beta}-{\rm max}\{1, \frac{1}{w(\alpha_3)}\}\Lambda_{\alpha_3,\beta}+{\rm max}\{1, \frac{1}{w(\alpha_4)}\}\Lambda_{\alpha_4,\beta}=0.
	\end{equation*}
	
	If $\beta=\alpha_1$, then we have $\Lambda_{\alpha_1,\beta}=0$, $\Lambda_{\alpha_2,\beta}=0$, $\Lambda_{\alpha_3,\beta}={\rm max} \{1, w(\alpha_3)\}$, $\Lambda_{\alpha_4,\beta}={\rm max} \{1, w(\alpha_4)\}$.
	Thus, we have
	\begin{equation*}
		((B^\Delta)^\top\Lambda^\Delta)_{\alpha,\beta}=\sum_{\gamma\in \Delta} b^\Delta_{\gamma,\alpha} \Lambda_{\gamma,\beta}=-\Lambda_{\alpha_1,\beta}+\Lambda_{\alpha_2,\beta}-{\rm max}\{1, \frac{1}{w(\alpha_3)}\}\Lambda_{\alpha_3,\beta}+{\rm max}\{1, \frac{1}{w(\alpha_4)}\}\Lambda_{\alpha_4,\beta}=0.
	\end{equation*}
	
	Similarly, $((B^\Delta)^\top\Lambda^\Delta)_{\alpha,\beta}=0$ in case $\beta=\alpha_2, \alpha_3$ or $\alpha_4$.
	
	If $\beta=\alpha$, then we have
	$\Lambda_{\alpha_2,\beta}=-\Lambda_{\alpha_1,\beta}=2$ and $\Lambda_{\alpha_3,\beta}=\Lambda_{\alpha_4,\beta}=0$. Thus, 
	\begin{equation*}
		((B^\Delta)^\top\Lambda^\Delta)_{\alpha,\beta}=\sum_{\gamma\in \Delta} b^\Delta_{\gamma,\alpha} \Lambda_{\gamma,\alpha}=-\Lambda_{\alpha_1,\beta}+\Lambda_{\alpha_2,\beta}-{\rm max}\{1, \frac{1}{w(\alpha_3)}\}\Lambda_{\alpha_3,\beta}+{\rm max}\{1, \frac{1}{w(\alpha_4)}\}\Lambda_{\alpha_4,\beta}=4.
	\end{equation*}
	
	{\bf Case 3.3}. Neither $(\alpha,\alpha_1,\alpha_2)$ nor $(\alpha,\alpha_3,\alpha_4)$ is a weight $1$ normal triangle. In this case, we have at least one of $\alpha_1,\alpha_2,\alpha_3,\alpha_4$ has at weight $1$. We may assume that $w(\alpha_1)=1$.
	
	{\bf Case 3.3.1}. Exactly one of $\alpha_3$ and $\alpha_4$ has weight $1$; we may assume that $w(\alpha_3)=1$. We only consider the case that $\alpha$ is not a loop, as the other case can be proved similarly.
	
	If $\beta\neq \alpha,\alpha_1,\alpha_2,\alpha_3,\alpha_4$, then we have 
	$$  -\Lambda_{\alpha_1,\beta}+{\rm max} \{1, \frac{1}{w(\alpha_2)}\}\Lambda_{\alpha_2,\beta}-\Lambda_{\alpha_3,\beta}+{\rm max} \{1, \frac{1}{w(\alpha_4)}\}\Lambda_{\alpha_4,\beta}=0.$$
	
	Thus,
	\begin{equation*}
		((B^\Delta)^\top\Lambda^\Delta)_{\alpha,\beta}=\sum_{\gamma\in \Delta} b^\Delta_{\gamma,\alpha} \Lambda_{\gamma,\beta}=-\Lambda_{\alpha_1,\beta}+{\rm max}\{1, \frac{1}{w(\alpha_2)}\}\Lambda_{\alpha_2,\beta}-\Lambda_{\alpha_3,\beta}+{\rm max}\{1, \frac{1}{w(\alpha_4)}\}\Lambda_{\alpha_4,\beta}=0.
	\end{equation*}
	
	If $\beta=\alpha_1$, then we have $\Lambda_{\alpha_1,\beta}=\Lambda_{\alpha_3,\beta}=0$, $\Lambda_{\alpha_2,\beta}=-{\rm max} \{1, w(\alpha_2)\}$ and $\Lambda_{\alpha_4,\beta}={\rm max} \{1, w(\alpha_4)\}$.
	Thus, we have
	\begin{equation*}
		((B^\Delta)^\top\Lambda^\Delta)_{\alpha,\beta}=-\Lambda_{\alpha_1,\beta}+{\rm max}\{1, \frac{1}{w(\alpha_2)}\}\Lambda_{\alpha_2,\beta}-\Lambda_{\alpha_3,\beta}+{\rm max}\{1, \frac{1}{w(\alpha_4)}\}\Lambda_{\alpha_4,\beta}=0.
	\end{equation*}
	
	Similarly, $((B^\Delta)^\top\Lambda)_{\alpha,\beta}=0$ in case $\beta=\alpha_2, \alpha_3$ or $\alpha_4$.
	
	If $\beta=\alpha$, then we have
	$\Lambda_{\alpha_2,\beta}={\rm max}\{1, w(\alpha_2)\}$, $\Lambda_{\alpha_1,\beta}=\Lambda_{\alpha_3,\beta}=0$, and $\Lambda_{\alpha_4,\beta}={\rm max}\{1, w(\alpha_4)\}$.
	Thus, 
	\begin{equation*}
		((B^\Delta)^\top\Lambda^\Delta_{\alpha,\beta}=\sum_{\gamma\in \Delta} b^\Delta_{\gamma,\alpha} \Lambda_{\gamma,\alpha}=-\Lambda_{\alpha_1,\beta}+\Lambda_{\alpha_2,\beta}-\Lambda_{\alpha_3,\beta}+{\rm max}\{1, \frac{1}{w(\alpha_4)}\}\Lambda_{\alpha_4,\beta}=4.
	\end{equation*}
	
	{\bf Case 3.3.2}. Neither $\alpha_3$ nor $\alpha_4$ has weight $1$. Then $\alpha,\alpha_1$ are loops. 
	
	If $\beta\neq \alpha,\alpha_1,\alpha_2,\alpha_3,\alpha_4$, then we have 
	$$ -\Lambda_{\alpha_1,\beta}+{\rm max} \{1, \frac{1}{w(\alpha_2)}\}\Lambda_{\alpha_2,\beta}-{\rm max} \{1, \frac{1}{w(\alpha_3)}\}\Lambda_{\alpha_3,\beta}+{\rm max} \{1, \frac{1}{w(\alpha_4)}\}\Lambda_{\alpha_4,\beta}=0.$$

	Thus,
	\begin{equation*}
		((B^\Delta)^\top\Lambda^\Delta)_{\alpha,\beta}=-\Lambda_{\alpha_1,\beta}+{\rm max}\{1, \frac{1}{w(\alpha_2)}\}\Lambda_{\alpha_2,\beta}-{\rm max} \{1, \frac{1}{w(\alpha_3)}\}\Lambda_{\alpha_3,\beta}+{\rm max}\{1, \frac{1}{w(\alpha_4)}\}\Lambda_{\alpha_4,\beta}=0.
	\end{equation*}
	
	If $\beta=\alpha_1$, then we have $\Lambda_{\alpha_1,\beta}=\Lambda_{\alpha_2,\beta}=\Lambda_{\alpha_3,\beta}=\Lambda_{\alpha_4,\beta}=0$.
	Thus, we have
	\begin{equation*}
		((B^\Delta)^\top\Lambda^\Delta)_{\alpha,\beta}=-\Lambda_{\alpha_1,\beta}+{\rm max}\{1, \frac{1}{w(\alpha_2)}\}\Lambda_{\alpha_2,\beta}-{\rm max} \{1, \frac{1}{w(\alpha_3)}\}\Lambda_{\alpha_3,\beta}+{\rm max}\{1, \frac{1}{w(\alpha_4)}\}\Lambda_{\alpha_4,\beta}=0.
	\end{equation*}
	
	Similarly, $((B^\Delta)^\top\Lambda)_{\alpha,\beta}=0$ in case $\beta=\alpha_2, \alpha_3$ or $\alpha_4$.
	
	If $\beta=\alpha$, then we have
	$\Lambda_{\alpha_2,\beta}=2\; {\rm max}\{1, w(\alpha_2)\}$ and $\Lambda_{\alpha_1,\beta}=\Lambda_{\alpha_3,\beta}=\Lambda_{\alpha_4,\beta}=0$.
	Thus, 
	\begin{equation*}
		((B^\Delta)^\top\Lambda^\Delta)_{\alpha,\beta}=-\Lambda_{\alpha_1,\beta}+{\rm max}\{1, \frac{1}{w(\alpha_2)}\}\Lambda_{\alpha_2,\beta}-{\rm max} \{1, \frac{1}{w(\alpha_3)}\}\Lambda_{\alpha_3,\beta}+{\rm max}\{1, \frac{1}{w(\alpha_4)}\}\Lambda_{\alpha_4,\beta}=4.
	\end{equation*}
	
	By \cite[Theorem 4.19.]{FST1} and Remark \ref{rem:thesame}, we have $B^{\Delta'}=\mu_\gamma( B^\Delta)$.
	
	The proof is complete.
\end{proof}

\subsection{Generalized quantum seed and mutation}
For quasi-triangulable punctured surface $\fS$, recall that $V$ denotes the set of boundary circles and $\mathring{\mathcal P}$ denotes the set of interior punctures. For each $p\in \mathring{\mathcal P}$, denote by $\alpha_p$ the peripheral loops around $p$.
Let
\begin{equation}\label{eq:groundring}
	\mathfrak R = \widetilde R=R[v^{\pm \frac{1}{2}}\mid v\in V],\qquad   \mathcal R=\widetilde R\,[\alpha_p\mid p\in \mathring{\mathcal P}].
\end{equation}

For any quasi-triangulation $\Delta$ of $\fS$ and any $\gamma \in \mathring\Delta$, let $d_\gamma = \max\{1, w(\gamma)\}$, and define the coefficient string
\begin{equation}\label{eq:h}
	\mathbf{h}_\gamma =
	\begin{cases}
		(1, 1) & \text{if } d_\gamma = 1, \\
		(1,\alpha_p, 1) & \text{if } d_\gamma = 2 \text{ and } \gamma \text{ is the loop enclosing } p \in \mathring{\mathcal P}.
	\end{cases}
\end{equation}

It is immediate that $d_\gamma\mid b^{\Delta}_{\alpha,\gamma}$ for all $\alpha\in \Delta$. Denote 
\[A^\Delta=(A_\gamma\mid \gamma\in \Delta),\quad
{\bf h}^\Delta=(\mathbf{h}_\gamma\mid \gamma\in\mathring\Delta),\quad
\text{and } {\bf s}(\Delta)=(A^\Delta,{\bf h}^\Delta,\Lambda^\Delta,B^\Delta).\]

\begin{proposition}\label{prop:mutation}
	Let $\fS$ be a quasi-triangulable punctured surface with a quasi-triangulation $\Delta$.
	Then ${\bf s}(\Delta)$ is a generalized quantum seed in ${\rm Frac}(\rdS)$. Moreover, if two triangulations $\Delta$ and $\Delta'$ are related by a flip at an admissible ideal arc $\alpha \in \Delta$, then
	\begin{equation}\label{eq:seedmutation}
		{\bf s}(\Delta') = \mu_\alpha({\bf s}(\Delta)).
	\end{equation}
\end{proposition}

\begin{proof}
	That ${\bf s}(\Delta)$ is a generalized quantum seed follows from Propositions~\ref{prop:quasi-commutative}, ~\ref{prop:quantumtrace}, and~\ref{prop:compatible}, the divisibility condition $d_\gamma \mid b^{\Delta}_{\alpha,\gamma}$ for all $\alpha, \gamma \in \Delta$, and equation~\eqref{eq:h}.
	
	We then prove \eqref{eq:seedmutation}. Assume $\alpha'\in \Delta'\setminus \Delta$.
	
	In this proof, all $\fS$-tangle diagrams carry only the $+$ state. Hence we omit the state symbol ``$+$" throughout.

	{\bf Case 1}. $w(\alpha)=\frac{1}{2}$. Suppose that $(\alpha,\alpha_1,\alpha_2)$ forms the unique clockwise special triangle in $\Delta$ with $\alpha$ as one of its sides. Then one of $\alpha_1$ and $\alpha_2$ has weight $1$, we may assume that $w(\alpha_1)=1$.
	
	We first consider the case that $w(\alpha_2)=1$.
	
	If $\alpha_1$ is not a loop, then for any $i\in \{1,2\}$, we have
	\begin{equation}\label{eq:lambda1}
		\Lambda_{\alpha_i,\alpha}=(-1)^{i}. 
	\end{equation}
	By relation \eqref{RY-relation}, we have 
	\begin{equation}\label{eqqqq-1}
		\raisebox{-.4in}{
			\begin{tikzpicture}
				\filldraw[draw=white,fill=gray!20] (0,0) rectangle (2, 2);
				\node (A) at (1,1.5) [left] {$\alpha'$};
				\node (B) at (1,0.5) [left] {$\alpha$};
				\draw [line width =1pt] (1,2) -- (1, 1);
				\draw [line width =1pt] (1,0.8) -- (1, 0);
				\filldraw[draw=black,fill=white]  (1,1) circle (0.1);
				\filldraw[draw=black,fill=white] (0.9,-0.1) rectangle (1.1, 0.1);
				\filldraw[draw=black,fill=white] (0.9,1.9) rectangle (1.1, 2.1);
			\end{tikzpicture}
		}=v^{-1}\left(q^{-\frac{1}{2}}
		\raisebox{-0.4in}{
			\begin{tikzpicture}
				\filldraw[draw=white,fill=gray!20] (0,0) rectangle (2, 2);
				\draw[line width =1pt] plot[smooth] coordinates {(1,0)  (0.7, 0.4) (0.5, 1) (0.7, 1.6) (1,2)};
				\node (A) at (0.5,1) [left] {$\alpha_1$};
				\filldraw[draw=black,fill=white]  (1,1) circle (0.1);
				\filldraw[draw=black,fill=white] (0.9,-0.1) rectangle (1.1, 0.1);
				\filldraw[draw=black,fill=white] (0.9,1.9) rectangle (1.1, 2.1);
			\end{tikzpicture}
		} + q^{\frac{1}{2}}
		\raisebox{-.4in}{
			\begin{tikzpicture}
				\filldraw[draw=white,fill=gray!20] (0,0) rectangle (2, 2);
				\draw[line width =1pt] plot[smooth] coordinates {(1,0)  (1.3, 0.4) (1.5, 1) (1.3, 1.6) (1,2)};
				\node (B) at (1.5,1) [right] {$\alpha_2$};
				\filldraw[draw=black,fill=white]  (1,1) circle (0.1);
				\filldraw[draw=black,fill=white] (0.9,-0.1) rectangle (1.1, 0.1);
				\filldraw[draw=black,fill=white] (0.9,1.9) rectangle (1.1, 2.1);
			\end{tikzpicture}
		}\right),
	\end{equation}
	where $v$ is the boundary circle in the picture.

	By \eqref{eq-D-X-q}, we have 
	\begin{equation}\label{eq:mu1}
		A_{\alpha'}A_\alpha=q^{-\frac{1}{2}} A_{\alpha_1}+q^{\frac{1}{2}} A_{\alpha_2}.
	\end{equation}
	
	Then by \eqref{eq:lambda1} \eqref{eq:mu1}, we obtain
	
	\begin{equation*}
		A_{\alpha'}=[A_{\alpha_1}A_\alpha^{-1}]+[A_{\alpha_2}A_\alpha^{-1}].
	\end{equation*}
	
	If $\alpha_1$ is a loop, then one of the following two cases holds:
	\begin{equation}\label{eq:lambda2}
		\Lambda_{\alpha_1,\alpha} = 0, \quad \Lambda_{\alpha_2,\alpha} =2
	\end{equation}
	or
	\begin{equation*}
		\Lambda_{\alpha_1,\alpha} =-2, \quad \Lambda_{\alpha_2,\alpha} = 0.
	\end{equation*}
	We may assume without loss of generality that the first case holds, as the second case follows by a similar argument.
	
	By relation \eqref{RY-relation}, we have 
	\begin{equation*}
		\begin{split}
			\begin{array}{rcl}
				\raisebox{-.4in}{
					\begin{tikzpicture}
						\filldraw[draw=white,fill=gray!20] (0,0) rectangle (2, 2);
						\node (A) at (1.4,0.65) [right] {$\alpha'$};
						\node (B) at (0.6,0.65) [left] {$\alpha$};
						\draw[line width =1pt] plot[smooth] coordinates {(1,0)  (1.25, 0.2) (1.4, 0.65) (1.25, 1.1) (1,1.3)};
						\draw[line width =1pt] plot[smooth] coordinates {(0.8,0.15)  (0.75, 0.2) (0.6, 0.65) (0.75, 1.1) (0.8,1.15)};
						\filldraw[draw=black,fill=white]  (1,1.3) circle (0.1);
						\filldraw[draw=black,fill=white] (0.9,-0.1) rectangle (1.1, 0.1);
						\filldraw[draw=white,fill=white!200] (1,0.5) circle (0.2);
					\end{tikzpicture}
				}
				&=&u^{-1}\left(q^{-\frac{1}{2}}
				\raisebox{-0.4in}{
					\begin{tikzpicture}
						\filldraw[draw=white,fill=gray!20] (0,0) rectangle (2, 2);
						\draw [line width =1pt] (1,0) arc[start angle=-90, end angle=250, radius=0.8];
						\filldraw[draw=black,fill=white]  (1,1.3) circle (0.1);
						\filldraw[draw=black,fill=white] (0.9,-0.1) rectangle (1.1, 0.1);
						\filldraw[draw=white,fill=white!200] (1,0.5) circle (0.2);
					\end{tikzpicture}
				} + q^{\frac{1}{2}}
				\raisebox{-.4in}{
					\begin{tikzpicture}
						\filldraw[draw=white,fill=gray!20] (0,0) rectangle (2, 2);
						\filldraw[draw=white,fill=white!200] (1,0.5) circle (0.2);
						\draw [line width =1pt] (1,0) arc[start angle=-90, end angle=250, radius=0.5];
						\filldraw[draw=black,fill=white]  (1,1.3) circle (0.1);
						\filldraw[draw=black,fill=white] (0.9,-0.1) rectangle (1.1, 0.1);
					\end{tikzpicture}
				}\right)\vspace{1mm}\\
				&\overset{\text{by \eqref{eq-same-height}}}{=}& 
				u^{-1}\left(\raisebox{-0.4in}{
					\begin{tikzpicture}
						\filldraw[draw=white,fill=gray!20] (0,0) rectangle (2, 2);
						\draw [line width =1pt] (1,0) arc[start angle=-90, end angle=270, radius=0.8];
						\filldraw[draw=black,fill=white]  (1,1.3) circle (0.1);
						\filldraw[draw=white,fill=white!200] (1,0.5) circle (0.2);
						\filldraw[draw=black,fill=white] (0.9,-0.1) rectangle (1.1, 0.1);
						\node (A) at (0.2,1) [right] {$\alpha_1$};
					\end{tikzpicture}
				} +q
				\raisebox{-.4in}{
					\begin{tikzpicture}
						\filldraw[draw=white,fill=gray!20] (0,0) rectangle (2, 2);
						\draw [line width =1pt] (1,0) arc[start angle=-90, end angle=270, radius=0.5];
						\filldraw[draw=black,fill=white]  (1,1.3) circle (0.1);
						\node (A) at (1.9,1) [left] {$\alpha_2$};
						\filldraw[draw=white,fill=white!200] (1,0.5) circle (0.2);
						\filldraw[draw=black,fill=white] (0.9,-0.1) rectangle (1.1, 0.1);
					\end{tikzpicture}
				}\right),
			\end{array}
		\end{split}
	\end{equation*}
	where $u$ is the boundary circle in the picture.
	
	By \eqref{eq-D-X-q}, we have
	\begin{equation}\label{eq:mu2}
		A_{\alpha'}A_\alpha= A_{\alpha_1}+ qA_{\alpha_2}.
	\end{equation}
	
	Then by \eqref{eq:lambda2} \eqref{eq:mu2}, we obtain
	
	\begin{equation*}
		A_{\alpha'}=[A_{\alpha_1}A_\alpha^{-1}]+[A_{\alpha_2}A_\alpha^{-1}].
	\end{equation*}
	
	The proof for the case where $w(\alpha_2)\neq 1$ follows a similar approach to the scenario where $w(\alpha_2)=1$ and $\alpha_1$ is a loop, so we omit it.
	
	This completes the proof of \eqref{eq:seedmutation} in the case $w(\alpha)=\frac{1}{2}$.

	{\bf Case 2}. $w(\alpha)=2$. Suppose that $(\alpha,\alpha_1,\alpha_2)$ forms the unique clockwise  triangle in $\Delta$ with $\alpha$ as one of its sides. 
	Then one of $\alpha_1$ and $\alpha_2$ has weight $1$; without loss of generality, assume $w(\alpha_1)=1$. Suppose further that $\alpha$ encircles the orbifold point $v\in U_1$.
	
	We first consider the case that $w(\alpha_2)=1$.
	
	If $\alpha_1$ is not a loop, then for any $i\in \{1,2\}$, we have
	\begin{equation}\label{eq:lambda3}
		\Lambda_{\alpha_i,\alpha}=2(-1)^i. 
	\end{equation}
	By relation \eqref{cross} twice, we have 
	
	\begin{equation*}
		\raisebox{-.4in}{
			\begin{tikzpicture}
				\filldraw[draw=white,fill=gray!20] (0,0) rectangle (2, 2);
				\node (A) at (0.8,1.5) [left] {$\alpha'$};
				\node (B) at (0.8,0.5) [left] {$\alpha$};
				\draw [line width =1pt] (1,2) arc[start angle=90, end angle=-270, radius=0.7];
				\draw[draw=black,fill=white]  (1,1) circle (0.1);
				\def\circleCenter{(1,0.7)}
				\def\circleRadius{0.7}
				\draw [line width =1pt] \circleCenter ++(-170:\circleRadius) arc (-170:10:\circleRadius);
				\draw[line width =1pt] \circleCenter ++(190:\circleRadius) arc (190:170:\circleRadius);
				\draw[line width =1pt] \circleCenter ++(140:\circleRadius) arc (140:40:\circleRadius);
				\draw[dashed] \circleCenter ++(170:\circleRadius) arc (170:180:\circleRadius);
				\draw[dashed] \circleCenter ++(-10:\circleRadius) arc (-10:0:\circleRadius);
				\filldraw[draw=black,fill=white] (0.9,-0.1) rectangle (1.1, 0.1);
				\filldraw[draw=black,fill=white] (0.9,1.9) rectangle (1.1, 2.1);
			\end{tikzpicture}
		}=q^{-1}
		\raisebox{-0.4in}{
			\begin{tikzpicture}
				\filldraw[draw=white,fill=gray!20] (0,0) rectangle (2, 2);
				\draw[line width =1pt] plot[smooth] coordinates {(1,0)  (0.8, 0.4) (0.6, 1) (0.8, 1.6) (1,2)};
				\draw[line width =1pt] plot[smooth] coordinates {(1,0)  (0.6, 0.4) (0.4, 1) (0.6, 1.6) (1,2)};
				\node (A) at (0.5,1) [left] {$\alpha_1^2$};
				\draw[draw=black,fill=white]  (1,1) circle (0.1);
				\filldraw[draw=black,fill=white] (0.9,-0.1) rectangle (1.1, 0.1);
				\filldraw[draw=black,fill=white] (0.9,1.9) rectangle (1.1, 2.1);
			\end{tikzpicture}
		} +
		\raisebox{-0.4in}{
			\begin{tikzpicture}
				\filldraw[draw=white,fill=gray!20] (0,0) rectangle (2, 2);
				\node (A) at (0.5,1) [left] {$\alpha_1$};
				\draw[line width =1pt] plot[smooth] coordinates {(1,0)  (0.6, 0.4) (0.4, 1) (0.6, 1.6) (1,2)};
				\draw[line width =1pt] plot[smooth] coordinates {(1,0)  (1.4, 0.4) (1.6, 1) (1.4, 1.6) (1,2)};
				\node (B) at (1.5,1) [right] {$\alpha_2$};
				\node (C) at (1,1) [right] {$\alpha_p$};
				\draw[draw=black,fill=white]  (1,1) circle (0.1);
				\draw[line width =1pt]  (1,1) circle (0.5);
				\filldraw[draw=black,fill=white] (0.9,-0.1) rectangle (1.1, 0.1);
				\filldraw[draw=black,fill=white] (0.9,1.9) rectangle (1.1, 2.1);
			\end{tikzpicture}
		}
		+
		q
		\raisebox{-.4in}{
			\begin{tikzpicture}
				\filldraw[draw=white,fill=gray!20] (0,0) rectangle (2, 2);
				\draw[line width =1pt] plot[smooth] coordinates {(1,0)  (1.2, 0.4) (1.4, 1) (1.2, 1.6) (1,2)};
				\draw[line width =1pt] plot[smooth] coordinates {(1,0)  (1.4, 0.4) (1.6, 1) (1.4, 1.6) (1,2)};
				\node (B) at (1.5,1) [right] {$\alpha_2^2$};
				\draw[draw=black,fill=white]  (1,1) circle (0.1);
				\filldraw[draw=black,fill=white] (0.9,-0.1) rectangle (1.1, 0.1);
				\filldraw[draw=black,fill=white] (0.9,1.9) rectangle (1.1, 2.1);
			\end{tikzpicture}
		}.
	\end{equation*}
	
	That is 
	\begin{equation}\label{eq:mu3}
		A_{\alpha'}A_\alpha=q^{-1} A^2_{\alpha_1}+\alpha_p[A_{\alpha_1}A_{\alpha_2}]+q A_{\alpha_2}^2.
	\end{equation}
	
	Then by \eqref{eq:lambda3} \eqref{eq:mu3}, we obtain
	
	\begin{equation*}
		A_{\alpha'}=[A^2_{\alpha_1}A_\alpha^{-1}]+\alpha_p[A_{\alpha_1}A_{\alpha_2}A_\alpha^{-1}]+[A^2_{\alpha_2}A_\alpha^{-1}].
	\end{equation*}
	
	If $\alpha_1$ is a loop, then one of the following two cases holds:
	\begin{equation}\label{eq:lambda4}
		\Lambda_{\alpha_1,\alpha} = 0, \quad \Lambda_{\alpha_2,\alpha} = 4
	\end{equation}
	or
	\begin{equation*}
		\Lambda_{\alpha_1,\alpha} =-4, \quad \Lambda_{\alpha_2,\alpha} = 0.
	\end{equation*}
	We may assume without loss of generality that the first case holds, as the second case follows by a similar argument.
	
	By relation \eqref{cross} twice, we have 
	
	\begin{equation*}
		\begin{split}
			\begin{array}{rcl}
				\raisebox{-.4in}{
					\begin{tikzpicture}
						\filldraw[draw=white,fill=gray!20] (0,0) rectangle (2, 2);
						\node (A) at (1.4,1.8) [right] {$\alpha'$};
						\node (B) at (0.6,1.8) [left] {$\alpha$};
						\draw[line width =1pt] plot[smooth] coordinates {(1,0) (1.75, 0.5)  (1.8, 1.3) (1, 1.8) (0.6, 1.3) (1,1) (1.3, 0.5) (1,0)};
						\draw[line width =1pt] plot[smooth] coordinates {(1,0) (0.25, 0.5)  (0.2, 1.3) (0.7, 1.7)}; 
					\draw[line width =1pt] plot[smooth] coordinates {(1,1.7)  (1.4, 1.3) (1.1,1.1)};
					\draw[line width =1pt] plot[smooth] coordinates { (0.9,0.9) (0.7, 0.5) (1,0)};
					\draw[draw=black,fill=white]  (1,1.3) circle (0.1);
					\filldraw[draw=black,fill=white] (0.9,-0.1) rectangle (1.1, 0.1);
					\filldraw[draw=white,fill=white!200] (1,0.59) circle (0.2);
					\node  at (1.2,0.3) {$2$};
					\node  at (1.5,0.3) {$2$};
					\node  at (0.8,0.3) {$1$};
					\node  at (0.5,0.3) {$1$};
				\end{tikzpicture}
			}
			&=&q^{-1}
			\raisebox{-0.4in}{
				\begin{tikzpicture}
					\filldraw[draw=white,fill=gray!20] (0,0) rectangle (2, 2);
					\draw [line width =1pt] (1,0) arc[start angle=-90, end angle=270, radius=0.9];
					\draw[line width =1pt] (1,0.75) ellipse (0.5cm and 0.75cm);
					\node  at (1.35,0.3) {$2$};
					\node  at (1.65,0.3) {$2$};
					\node  at (0.65,0.3) {$1$};
					\node  at (0.35,0.3) {$1$};
					\draw[draw=black,fill=white]  (1,1.3) circle (0.1);
					\filldraw[draw=black,fill=white] (0.9,-0.1) rectangle (1.1, 0.1);
					\filldraw[draw=white,fill=white!200] (1,0.5) circle (0.2);
				\end{tikzpicture}
			} + 
			\raisebox{-.4in}{
				\begin{tikzpicture}
					\filldraw[draw=white,fill=gray!20] (0,0) rectangle (2, 2);
					\filldraw[draw=white,fill=white!200] (1,0.5) circle (0.2);
					\draw [line width =1pt] (1,0) arc[start angle=-90, end angle=270, radius=0.9];
					\draw [line width =1pt] (1,0) arc[start angle=-90, end angle=270, radius=0.45];
					\draw[draw=black,fill=white]  (1,1.3) circle (0.1);
					\draw[line width =1pt][draw=black]  (1,1.3) circle (0.3);
					\filldraw[draw=black,fill=white] (0.9,-0.1) rectangle (1.1, 0.1);
					\node  at (1.25,0.07) {$2$};
					\node  at (0.75,0.07) {$1$};
				\end{tikzpicture}
			}+q
			\raisebox{-.4in}{
				\begin{tikzpicture}
					\filldraw[draw=white,fill=gray!20] (0,0) rectangle (2, 2);
					\filldraw[draw=white,fill=white!200] (1,0.5) circle (0.2);
					\draw [line width =1pt] (1,0) arc[start angle=-90, end angle=270, radius=0.55];
					\draw [line width =1pt] (1,0) arc[start angle=-90, end angle=270, radius=0.45];
					\filldraw[draw=black,fill=white]  (1,1.3) circle (0.1);
					\filldraw[draw=black,fill=white] (0.9,-0.1) rectangle (1.1, 0.1);
					\node  at (1.25,0.07) {$2$};
					\node  at (0.75,0.07) {$1$};
				\end{tikzpicture}
			}\vspace{1mm}\\
			&\overset{\text{by \eqref{eq-same-height}}}{=}& 
			\raisebox{-0.4in}{
				\begin{tikzpicture}
					\filldraw[draw=white,fill=gray!20] (0,0) rectangle (2, 2);
					\draw [line width =1pt] (1,0) arc[start angle=-90, end angle=270, radius=0.9];
					\draw [line width =1pt] (1,0) arc[start angle=-90, end angle=270, radius=0.75];
					\draw[draw=black,fill=white]  (1,1.3) circle (0.1);
					\filldraw[draw=black,fill=white] (0.9,-0.1) rectangle (1.1, 0.1);
					\filldraw[draw=white,fill=white!200] (1,0.5) circle (0.2);
					\node (A) at (0.2,0.6) [right] {$\alpha^2_1$};
				\end{tikzpicture}
			} + q
			\raisebox{-.4in}{
				\begin{tikzpicture}
					\filldraw[draw=white,fill=gray!20] (0,0) rectangle (2, 2);
					\filldraw[draw=white,fill=white!200] (1,0.5) circle (0.2);
					\draw [line width =1pt] (1,0) arc[start angle=-90, end angle=270, radius=0.9];
					\draw [line width =1pt] (1,0) arc[start angle=-90, end angle=270, radius=0.45];
					\draw[draw=black,fill=white]  (1,1.3) circle (0.1);
					\draw[line width =1pt][draw=black]  (1,1.3) circle (0.3);
					\filldraw[draw=black,fill=white] (0.9,-0.1) rectangle (1.1, 0.1);
					\node (A) at (0.07,1) [right] {$\alpha_1$};
					\node (B) at (1.2,1.3) [right] {$\alpha_p$};
					\node (C) at (1,0.5) [right] {$\alpha_2$};
				\end{tikzpicture}
			}+q^{2}
			\raisebox{-.4in}{
				\begin{tikzpicture}
					\filldraw[draw=white,fill=gray!20] (0,0) rectangle (2, 2);
					\filldraw[draw=white,fill=white!200] (1,0.5) circle (0.2);
					\draw [line width =1pt] (1,0) arc[start angle=-90, end angle=270, radius=0.55];
					\draw [line width =1pt] (1,0) arc[start angle=-90, end angle=270, radius=0.45];
					\filldraw[draw=black,fill=white]  (1,1.3) circle (0.1);
					\filldraw[draw=black,fill=white] (0.9,-0.1) rectangle (1.1, 0.1);
					\node (C) at (1.5,0.7) [right] {$\alpha^2_2$};
				\end{tikzpicture}
			}.
		\end{array}
	\end{split}
\end{equation*}
Here the numbers on the half-edges indicate their height order: a larger number means greater height, half-edges with the same number have the same height, and all half-edges are taken to have the same height when no numbers are assigned.

That is 
\begin{equation}\label{eq:mu4}
	A_{\alpha'}A_\alpha=A^2_{\alpha_1}+q \,\alpha_p\,[A_{\alpha_1}A_{\alpha_2}]+q^{2} A_{\alpha_2}^2.
\end{equation}

Then by \eqref{eq:lambda4} \eqref{eq:mu4}, we obtain

\begin{equation*}
	A_{\alpha'}=[A^2_{\alpha_1}A_\alpha^{-1}]+\alpha_p[A_{\alpha_1}A_{\alpha_2}A_\alpha^{-1}]+[A^2_{\alpha_2}A_\alpha^{-1}].
\end{equation*}

The proof for the case where $w(\alpha_2)\neq 1$ follows a similar approach to the scenario where $w(\alpha_2)=1$ and $\alpha_1$ is a loop, so we omit it.

This completes the proof of \eqref{eq:seedmutation} in the case $w(\alpha)=2$.

{\bf Case 3}. $w(\alpha)=1$. Suppose that $\alpha_1,\alpha_2,\alpha_3,\alpha_4$ are arcs in $\Delta$ such that $(\alpha,\alpha_1,\alpha_2)$ and $(\alpha,\alpha_3,\alpha_4)$ form the clockwise triangles with $\alpha$ as one of their sides. We may assume that $\alpha_1\neq \alpha_3$ and $\alpha_2\neq \alpha_4$, as the other cases follow by a similar argument. 

{\bf Case 3.1}. $(\alpha,\alpha_1,\alpha_2)$ and $(\alpha,\alpha_3,\alpha_4)$ are weight 1 normal triangles, i.e., $w(\alpha_1)=w(\alpha_2)=w(\alpha_3)=w(\alpha_4)=1$. We only consider the case where the shared endpoints of the pairs $(\alpha_1,\alpha_2)$, $(\alpha_2,\alpha_3)$, $(\alpha_3,\alpha_4)$, and $(\alpha_4,\alpha_1)$ are different, as the other cases can be proved similarly. Then for any $i\in \{1,2,3,4\}$, we have
\begin{equation}\label{eq:lambda5}
	\Lambda_{\alpha_i,\alpha}=(-1)^i.
\end{equation}

By relation \eqref{cross}, we have

\begin{equation*}
	\raisebox{-.4in}{
		\begin{tikzpicture}
			\filldraw[draw=white,fill=gray!20] (0.2,0) rectangle (2.2, 2);
			\node (A) at (1.5,1.3) [right] {$\alpha'$};
			\node (B) at (0.8,1.3) [left] {$\alpha$};
			\draw [line width =1pt] (0.2,0) -- (2.2, 2);
			\draw [line width =1pt] (0.2,2) -- (1, 1.2);
			\draw [line width =1pt] (1.4,0.8) -- (2.2, 0);
			\filldraw[draw=black,fill=white] (0.1,-0.1) rectangle (0.3, 0.1);
			\filldraw[draw=black,fill=white] (0.1,1.9) rectangle (0.3, 2.1);
			\filldraw[draw=black,fill=white] (2.1,-0.1) rectangle (2.3, 0.1);
			\filldraw[draw=black,fill=white] (2.1,1.9) rectangle (2.3, 2.1);
		\end{tikzpicture}
	}=q
	\raisebox{-0.4in}{
		\begin{tikzpicture}
			\filldraw[draw=white,fill=gray!20] (0.2,0) rectangle (2.2, 2);
			\node (A) at (0.2,1) [right] {$\alpha_2$};
			\node (B) at (2.2,1) [left] {$\alpha_4$};
			\draw [line width =1pt] (0.2,0) -- (0.2, 2);
			\draw [line width =1pt] (2.2,0) -- (2.2, 2);
			\filldraw[draw=black,fill=white] (0.1,-0.1) rectangle (0.3, 0.1);
			\filldraw[draw=black,fill=white] (0.1,1.9) rectangle (0.3, 2.1);
			\filldraw[draw=black,fill=white] (2.1,-0.1) rectangle (2.3, 0.1);
			\filldraw[draw=black,fill=white] (2.1,1.9) rectangle (2.3, 2.1);
		\end{tikzpicture}
	} + q^{-1}
	\raisebox{-.4in}{
		\begin{tikzpicture}
			\filldraw[draw=white,fill=gray!20] (0.2,0) rectangle (2.2, 2);
			\node (A) at (1.2,0) [above] {$\alpha_3$};
			\node (B) at (1.2,2) [below] {$\alpha_1$};
			\draw [line width =1pt] (0.2,0) -- (2.2, 0);
			\draw [line width =1pt] (0.2,2) -- (2.2, 2);
			\filldraw[draw=black,fill=white] (0.1,-0.1) rectangle (0.3, 0.1);
			\filldraw[draw=black,fill=white] (0.1,1.9) rectangle (0.3, 2.1);
			\filldraw[draw=black,fill=white] (2.1,-0.1) rectangle (2.3, 0.1);
			\filldraw[draw=black,fill=white] (2.1,1.9) rectangle (2.3, 2.1);
		\end{tikzpicture}
	}.
\end{equation*}

That is 
\begin{equation}\label{eq:mu5}
	A_{\alpha'}A_\alpha=q [A_{\alpha_2} A_{\alpha_4}]+q^{-1} [A_{\alpha_1} A_{\alpha_3}].
\end{equation}

Then by \eqref{eq:lambda5} \eqref{eq:mu5}, we obtain

\begin{equation*}
	A_{\alpha'}=[A_{\alpha_1} A_{\alpha_3}A_\alpha^{-1}]+[A_{\alpha_2} A_{\alpha_4}A_\alpha^{-1}].
\end{equation*}

{\bf Case 3.2}. Exactly one of $(\alpha,\alpha_1,\alpha_2)$ and $(\alpha,\alpha_3,\alpha_4)$ is a weight 1 normal triangle; we may assume it is $(\alpha,\alpha_1,\alpha_2)$. 

{\bf Case 3.2.1}. Exactly one of $\alpha_3$ and $\alpha_4$ has weight $1$; we may assume that $w(\alpha_3)=1$. We only consider the case where the shared endpoints of the pairs $(\alpha_1,\alpha_2)$, $(\alpha_2,\alpha)$, and $(\alpha,\alpha_1)$ are different, as the other cases can be proved similarly. Then we have
\begin{equation}\label{eq:lambda6}
	\Lambda_{\alpha_1,\alpha}=-1,\quad  \Lambda_{\alpha_2,\alpha}=1, \quad \Lambda_{\alpha_3,\alpha}=0,\quad \Lambda_{\alpha_4,\alpha}={\rm max}\{1,w(\alpha_4)\}.
\end{equation}

Then by relation \ref{cross}, we have

\begin{equation*}
	\begin{split}
		\begin{array}{rcl}
			\raisebox{-.4in}{
				\begin{tikzpicture}
					\filldraw[draw=white,fill=gray!20] (0,0) rectangle (2, 2);
					\draw[line width =1pt] plot[smooth] coordinates {(0.7,0)  (1.3, 0.5) (1.5, 1) (1.2, 1.45) (1, 1.58) (0.7,1.6) (0.5,1.53) (0,1)};
					\node (B) at (0.7,0.7) [left] {$\alpha$};
					\node (A) at (1.4,0.7) [right] {$\alpha'$};
					\draw [line width =1pt] (0.7,1.7) -- (0.7, 2);
					\draw [line width =1pt] (0.7,0.3) -- (0.7, 1.5);
					\draw [line width =1pt] (0.9,0.9) -- (1.1, 1.1);
					\draw [line width =1pt] (0.9,1.1) -- (1.1, 0.9);
					\filldraw[draw=black,fill=white] (-0.1,0.9) rectangle (0.1, 1.1);
					\filldraw[draw=black,fill=white] (0.6,-0.1) rectangle (0.8, 0.1);
					\filldraw[draw=black,fill=white] (0.6,1.9) rectangle (0.8, 2.1);
				\end{tikzpicture}
			}
			&=&q
			\raisebox{-0.4in}{
				\begin{tikzpicture}
					\filldraw[draw=white,fill=gray!20] (0,0) rectangle (2, 2);
					\node (B) at (-0.1,1.7) [right] {$\alpha_2$};
					\draw [line width =1pt] (0.7,2) -- (0, 1);
					\draw[line width =1pt] plot[smooth] coordinates {(0.7,0)  (1.3, 0.5) (1.5, 1) (1.2, 1.45) (1, 1.5) (0.7,1.4) (0.5,1) (0.7,0.3)};
					\draw [line width =1pt] (0.9,0.9) -- (1.1, 1.1);
					\draw [line width =1pt] (0.9,1.1) -- (1.1, 0.9);
					\filldraw[draw=black,fill=white] (-0.1,0.9) rectangle (0.1, 1.1);
					\filldraw[draw=black,fill=white] (0.6,-0.1) rectangle (0.8, 0.1);
					\filldraw[draw=black,fill=white] (0.6,1.9) rectangle (0.8, 2.1);
				\end{tikzpicture}
			} + q^{-1}
			\raisebox{-.4in}{
				\begin{tikzpicture}
					\filldraw[draw=white,fill=gray!20] (0,0) rectangle (2, 2);
					\node (C) at (1.8,1) [left] {$\alpha_3$};
					\draw [line width =1pt] (0.9,0.9) -- (1.1, 1.1);
					\draw [line width =1pt] (0.9,1.1) -- (1.1, 0.9);
					\draw [line width =1pt] (0,1) -- (0.6, 0.2);
					\draw [line width =1pt] (0.7,0) arc[start angle=-90, end angle=90, radius=1];
					\filldraw[draw=black,fill=white] (-0.1,0.9) rectangle (0.1, 1.1);
					\filldraw[draw=black,fill=white] (0.6,-0.1) rectangle (0.8, 0.1);
					\filldraw[draw=black,fill=white] (0.6,1.9) rectangle (0.8, 2.1);
				\end{tikzpicture}
			}\vspace{1mm}\\
			&\overset{\text{by \eqref{eq-same-height}}}{=}&  q^{\frac{3}{2}}
			\raisebox{-0.4in}{
				\begin{tikzpicture}
					\filldraw[draw=white,fill=gray!20] (0,0) rectangle (2, 2);
					\node (B) at (-0.1,1.7) [right] {$\alpha_2$};
					\draw [line width =1pt] (0.7,2) -- (0, 1);
					\draw[line width =1pt] plot[smooth] coordinates {(0.7,0)  (1.3, 0.5) (1.5, 1) (1.2, 1.45) (1, 1.5) (0.7,1.4) (0.5,1) (0.7,0)};
					\draw [line width =1pt] (0.9,0.9) -- (1.1, 1.1);
					\draw [line width =1pt] (0.9,1.1) -- (1.1, 0.9);
					\filldraw[draw=black,fill=white] (-0.1,0.9) rectangle (0.1, 1.1);
					\filldraw[draw=black,fill=white] (0.6,-0.1) rectangle (0.8, 0.1);
					\filldraw[draw=black,fill=white] (0.6,1.9) rectangle (0.8, 2.1);
				\end{tikzpicture}
			} + q^{-\frac{1}{2}}
			\raisebox{-.4in}{
				\begin{tikzpicture}
					\filldraw[draw=white,fill=gray!20] (0,0) rectangle (2, 2);
					\node (B) at (-0.1,0.3) [right] {$\alpha_1$};
					\node (C) at (1.8,1) [left] {$\alpha_3$};
					\draw [line width =1pt] (0.9,0.9) -- (1.1, 1.1);
					\draw [line width =1pt] (0.9,1.1) -- (1.1, 0.9);
					\draw [line width =1pt] (0,1) -- (0.7, 0);
					\draw [line width =1pt] (0.7,0) arc[start angle=-90, end angle=90, radius=1];
					\filldraw[draw=black,fill=white] (-0.1,0.9) rectangle (0.1, 1.1);
					\filldraw[draw=black,fill=white] (0.6,-0.1) rectangle (0.8, 0.1);
					\filldraw[draw=black,fill=white] (0.6,1.9) rectangle (0.8, 2.1);
				\end{tikzpicture}
			},
		\end{array}
	\end{split}
\end{equation*}
where $\begin{tikzpicture}
	\draw [line width =1pt] (0.9,0.9) -- (1.1, 1.1);
	\draw [line width =1pt] (0.9,1.1) -- (1.1, 0.9);
\end{tikzpicture}\in \{\begin{tikzpicture}
	\filldraw[draw=black,fill=white]  (0.5,0) circle (0.1);
\end{tikzpicture}, \begin{tikzpicture}
	\draw[draw=black,fill=white] (0.4,-0.1) rectangle (0.6,0.1);
\end{tikzpicture}\}.$

That is 
\begin{equation}\label{eq:mu6}
	A_{\alpha'}A_\alpha=q^{-\frac{1}{2}}[A_{\alpha_1} A_{\alpha_3}]+q^{\frac{3}{2}} [A_{\alpha_2}A_{\alpha_4}^{b^\Delta_{\alpha_4,\alpha}}].
\end{equation}

Then by \eqref{eq:lambda6} \eqref{eq:mu6}, we obtain

\begin{equation*}
	A_{\alpha'}=[A_{\alpha_1} A_{\alpha_3}A_\alpha^{-1}]+ [A_{\alpha_2}A_{\alpha_4}^{b^\Delta_{\alpha_4,\alpha}}A_\alpha^{-1}].
\end{equation*}

{\bf Case 3.2.2}. Neither $\alpha_3$ nor $\alpha_4$ has weight $1$. Then $\alpha$ is a loop. We only consider the case where $\alpha_1$ is not a loop, as the other case can be proved similarly. Then we have
\begin{equation}\label{eq:lambda7}
	-\Lambda_{\alpha_1,\alpha}=
	\Lambda_{\alpha_2,\alpha}=2,\quad \Lambda_{\alpha_3,\alpha}= \Lambda_{\alpha_4,\alpha}=0. 
\end{equation}

By relation \eqref{cross} , we have 

\begin{equation*}
	\begin{split}
		\begin{array}{rcl}
			\raisebox{-.4in}{
				\begin{tikzpicture}
					\filldraw[draw=white,fill=gray!20] (0,0) rectangle (2, 2);
					\node (A) at (1,1.7) [left] {$\alpha'$};
					\node (B) at (0.4,0.5) [left] {$\alpha$};
					\draw[line width =1pt] (1,0) -- (1,2);
					\draw [line width =1pt] (0.3,0.7)  arc (-180:0:0.7);
					\def\circleCenter{(1,0.7)}
					\def\circleRadius{0.7}
					\draw[line width =1pt] \circleCenter ++(250:\circleRadius) arc (250:100:\circleRadius);
					\draw[line width =1pt] \circleCenter ++(80:\circleRadius) arc (80:-70:\circleRadius);
					\draw[draw=black]  (0.7,0.5) circle (0.1);
					\draw[line width =1pt] (0.6,0.4) -- (0.8,0.6);
					\draw[line width =1pt] (0.6,0.6) -- (0.8,0.4);
					\draw[line width =1pt] (1.2,0.4) -- (1.4,0.6);
					\draw[line width =1pt] (1.2,0.6) -- (1.4,0.4);
					\node (A) at (1,0.3)  {$2$};
					\node (B) at (0.75,0.07)  {$1$};
					\node (B) at (1.25,0.07)  {$1$};
					\filldraw[draw=black,fill=white] (0.9,-0.1) rectangle (1.1, 0.1);
					\filldraw[draw=black,fill=white] (0.9,1.9) rectangle (1.1, 2.1);
				\end{tikzpicture}
			}
			&=&q^{-1}
			\raisebox{-0.4in}{
				\begin{tikzpicture}
					\filldraw[draw=white,fill=gray!20] (0,0) rectangle (2, 2);
					\def\circleCenter{(1,1)}
					\def\circleRadius{1}
					\draw[line width =1pt] \circleCenter ++(250:\circleRadius) arc (250:90:\circleRadius);
					\draw[draw=black]  (0.7,0.5) circle (0.1);
					\draw[line width =1pt] (0.6,0.4) -- (0.8,0.6);
					\draw[line width =1pt] (0.6,0.6) -- (0.8,0.4);
					\draw[line width =1pt] (1.2,0.4) -- (1.4,0.6);
					\draw[line width =1pt] (1.2,0.6) -- (1.4,0.4);
					\filldraw[draw=black,fill=white] (0.9,-0.1) rectangle (1.1, 0.1);
					\filldraw[draw=black,fill=white] (0.9,1.9) rectangle (1.1, 2.1);
					\draw[line width =1pt] plot[smooth] coordinates {(1,0)   (1.6, 0.3) (1.8, 1) (1.5, 1.45) (1.3, 1.5) (1,1.4) (0.8,1) (1,0)};
					\draw [line width =1pt] (0,1)  arc (-180:-90:1);
					\node at (0.93,0.3)  {$2$};
					\node  at (0.75,0.07)  {$1$};
					\node  at (1.25,0.07)  {$1$};
					\filldraw[draw=black,fill=white] (0.9,-0.1) rectangle (1.1, 0.1);
					\filldraw[draw=black,fill=white] (0.9,1.9) rectangle (1.1, 2.1);
				\end{tikzpicture}
			} 
			+
			q
			\raisebox{-.4in}{
				\begin{tikzpicture}
					\filldraw[draw=white,fill=gray!20] (0,0) rectangle (2, 2);
					\draw[draw=black]  (0.7,0.5) circle (0.1);
					\draw[line width =1pt] (0.6,0.4) -- (0.8,0.6);
					\draw[line width =1pt] (0.6,0.6) -- (0.8,0.4);
					\draw[line width =1pt] (1.2,0.4) -- (1.4,0.6);
					\draw[line width =1pt] (1.2,0.6) -- (1.4,0.4);
					\filldraw[draw=black,fill=white] (0.9,-0.1) rectangle (1.1, 0.1);
					\filldraw[draw=black,fill=white] (0.9,1.9) rectangle (1.1, 2.1);
					\def\circleCenter{(1,1)}
					\def\circleRadius{1}
					\draw[line width =1pt] \circleCenter ++(90:\circleRadius) arc (90:-75:\circleRadius);
					\draw[line width =1pt] plot[smooth] coordinates {(1,0)  (0.4, 0.23) (0.2, 1) (0.5, 1.45) (0.7, 1.5) (1,1.4) (1.2,1) (1,0)};
					\draw [line width =1pt] (1,0)  arc (-90:0:1);
					\node at (1.07,0.3)  {$2$};
					\node  at (0.75,0.07)  {$1$};
					\node  at (1.25,0.07)  {$1$};
					\filldraw[draw=black,fill=white] (0.9,-0.1) rectangle (1.1, 0.1);
					\filldraw[draw=black,fill=white] (0.9,1.9) rectangle (1.1, 2.1);
				\end{tikzpicture}
			}\vspace{1mm}\\
			&\overset{\text{by \eqref{eq-same-height}}}{=}& 
			q^{-1}
			\raisebox{-0.4in}{
				\begin{tikzpicture}
					\filldraw[draw=white,fill=gray!20] (0,0) rectangle (2, 2);
					\def\circleCenter{(1,1)}
					\def\circleRadius{1}
					\draw[line width =1pt] \circleCenter ++(270:\circleRadius) arc (270:90:\circleRadius);
					\node (A) at (0.6,1) [left] {$\alpha_1$};
					\draw[draw=black]  (0.7,0.5) circle (0.1);
					\draw[line width =1pt] (0.6,0.4) -- (0.8,0.6);
					\draw[line width =1pt] (0.6,0.6) -- (0.8,0.4);
					\draw[line width =1pt] (1.2,0.4) -- (1.4,0.6);
					\draw[line width =1pt] (1.2,0.6) -- (1.4,0.4);
					\filldraw[draw=black,fill=white] (0.9,-0.1) rectangle (1.1, 0.1);
					\filldraw[draw=black,fill=white] (0.9,1.9) rectangle (1.1, 2.1);
					\draw[line width =1pt] plot[smooth] coordinates {(1,0)  (1.6, 0.5) (1.8, 1) (1.5, 1.45) (1.3, 1.5) (1,1.4) (0.8,1) (1,0)};
					\filldraw[draw=black,fill=white] (0.9,-0.1) rectangle (1.1, 0.1);
					\filldraw[draw=black,fill=white] (0.9,1.9) rectangle (1.1, 2.1);
				\end{tikzpicture}
			} 
			+
			q
			\raisebox{-.4in}{
				\begin{tikzpicture}
					\filldraw[draw=white,fill=gray!20] (0,0) rectangle (2, 2);
					\draw[draw=black]  (0.7,0.5) circle (0.1);
					\draw[line width =1pt] (0.6,0.4) -- (0.8,0.6);
					\draw[line width =1pt] (0.6,0.6) -- (0.8,0.4);
					\draw[line width =1pt] (1.2,0.4) -- (1.4,0.6);
					\draw[line width =1pt] (1.2,0.6) -- (1.4,0.4);
					\node (B) at (1.4,1) [right] {$\alpha_2$};
					\filldraw[draw=black,fill=white] (0.9,-0.1) rectangle (1.1, 0.1);
					\filldraw[draw=black,fill=white] (0.9,1.9) rectangle (1.1, 2.1);
					\def\circleCenter{(1,1)}
					\def\circleRadius{1}
					\draw[line width =1pt] \circleCenter ++(90:\circleRadius) arc (90:-90:\circleRadius);
					\draw[line width =1pt] plot[smooth] coordinates {(1,0)  (0.4, 0.5) (0.2, 1) (0.5, 1.45) (0.7, 1.5) (1,1.4) (1.2,1) (1,0)};
					\filldraw[draw=black,fill=white] (0.9,-0.1) rectangle (1.1, 0.1);
					\filldraw[draw=black,fill=white] (0.9,1.9) rectangle (1.1, 2.1);
				\end{tikzpicture}
			},
		\end{array}
	\end{split}
\end{equation*}
where $\begin{tikzpicture}
	\draw[draw=black]  (1,1) circle (0.1);
	\draw [line width =1pt] (0.9,0.9) -- (1.1, 1.1);
	\draw [line width =1pt] (0.9,1.1) -- (1.1, 0.9);
\end{tikzpicture}, \begin{tikzpicture}
	\draw [line width =1pt] (0.9,0.9) -- (1.1, 1.1);
	\draw [line width =1pt] (0.9,1.1) -- (1.1, 0.9);
\end{tikzpicture}\in \{\begin{tikzpicture}
	\filldraw[draw=black,fill=white]  (0.5,0) circle (0.1);
\end{tikzpicture}, \begin{tikzpicture}
	\draw[draw=black,fill=white] (0.4,-0.1) rectangle (0.6,0.1);
\end{tikzpicture}\}.$

That is 
\begin{equation}\label{eq:mu7}
	A_{\alpha'}A_\alpha=q^{-1}[A_{\alpha_1} A_{\alpha_3}^{-b^\Delta_{\alpha_3,\alpha}}]+q [A_{\alpha_2}A_{\alpha_4}^{b^\Delta_{\alpha_4,\alpha}}].
\end{equation}
Then by \eqref{eq:lambda7} and \eqref{eq:mu7}, we obtain

\begin{equation*}
	A_{\alpha'}=[A_{\alpha_1} A_{\alpha_3}^{-b^\Delta_{\alpha_3,\alpha}}A_\alpha^{-1}]+ [A_{\alpha_2}A_{\alpha_4}^{b^\Delta_{\alpha_4,\alpha}}A_\alpha^{-1}].
\end{equation*}

{\bf Case 3.3}. Neither $(\alpha,\alpha_1,\alpha_2)$ nor $(\alpha,\alpha_3,\alpha_4)$ is a weight 1 normal triangle. In this case, at least one of $\alpha_1,\alpha_2,\alpha_3,\alpha_4$ has at weight $1$. We may assume that $w(\alpha_1)=1$.

{\bf Case 3.3.1}. Exactly one of $\alpha_3$ and $\alpha_4$ has weight $1$; we may assume that $w(\alpha_3)=1$. We only consider the case that $\alpha$ is not a loop, as the other case can be proved similarly. Then we have
\begin{equation}\label{eq:lambda8}
	\Lambda_{\alpha_1,\alpha}=\Lambda_{\alpha_3,\alpha}=0,\quad \Lambda_{\alpha_2,\alpha}={\rm max}\{1,w(\alpha_2)\},\quad\Lambda_{\alpha_4,\alpha}={\rm max}\{1,w(\alpha_4)\}. 
\end{equation}

It follows that 
\begin{equation}\label{eq:lambda81}
	b^\Delta_{\alpha_2,\alpha}\Lambda_{\alpha_2,\alpha}=b^\Delta_{\alpha_4,\alpha}\Lambda_{\alpha_4,\alpha}=2. 
\end{equation}

By relation \eqref{cross} , we have

\begin{equation*}
	\begin{split}
		\begin{array}{rcl}
			\raisebox{-.4in}{
				\begin{tikzpicture}
					\filldraw[draw=white,fill=gray!20] (0,0) rectangle (2, 2);
					\node (A) at (0.7,1.7) [left] {$\alpha'$};
					\node (B) at (1,0.4) [left] {$\alpha$};
					\draw[line width =1pt] (1,0.2) -- (1,0.9);
					\draw[line width =1pt] (1,1.1) -- (1,1.8);
					\filldraw[draw=black,fill=white] (0.9,-0.1) rectangle (1.1, 0.1);
					\filldraw[draw=black,fill=white] (0.9,1.9) rectangle (1.1, 2.1);
					\def\circleCenter{(1,0.7)}
					\def\circleRadius{0.7}
					\draw[line width =1pt] (0.4,0.9) -- (0.6,1.1);
					\draw[line width =1pt] (0.6,0.9) -- (0.4,1.1);
					\draw[line width =1pt] (1.4,0.9) -- (1.6,1.1);
					\draw[line width =1pt] (1.6,0.9) -- (1.4,1.1);
					\draw[draw=black]  (0.5,1) circle (0.1);
					\draw[line width =1pt] plot[smooth] coordinates {(1,0)  (1.8, 1) (1.5, 1.4) (1,1) (0.5,0.6) (0.2,1) (1,2)};
					\filldraw[draw=black,fill=white] (0.9,-0.1) rectangle (1.1, 0.1);
					\filldraw[draw=black,fill=white] (0.9,1.9) rectangle (1.1, 2.1);
				\end{tikzpicture}
			}
			&=&q^{-1}
			\raisebox{-0.4in}{
				\begin{tikzpicture}
					\filldraw[draw=white,fill=gray!20] (0,0) rectangle (2, 2);
					\def\circleCenter{(1,1)}
					\def\circleRadius{1}
					\draw[line width =1pt] \circleCenter ++(250:\circleRadius) arc (250:90:\circleRadius);
					\draw[draw=black]  (0.5,1) circle (0.1);
					\draw[line width =1pt] (0.4,0.9) -- (0.6,1.1);
					\draw[line width =1pt] (0.6,0.9) -- (0.4,1.1);
					\draw[line width =1pt] (1.4,0.9) -- (1.6,1.1);
					\draw[line width =1pt] (1.6,0.9) -- (1.4,1.1);
					\filldraw[draw=black,fill=white] (0.9,-0.1) rectangle (1.1, 0.1);
					\filldraw[draw=black,fill=white] (0.9,1.9) rectangle (1.1, 2.1);
					\draw[line width =1pt] \circleCenter ++(75:\circleRadius) arc (75:-90:\circleRadius);
					\filldraw[draw=black,fill=white] (0.9,-0.1) rectangle (1.1, 0.1);
					\filldraw[draw=black,fill=white] (0.9,1.9) rectangle (1.1, 2.1);
				\end{tikzpicture}
			} 
			+
			q
			\raisebox{-.4in}{
				\begin{tikzpicture}
					\filldraw[draw=white,fill=gray!20] (0,0) rectangle (2, 2);
					\draw[draw=black]  (0.5,1) circle (0.1);
					\draw[line width =1pt] (0.4,0.9) -- (0.6,1.1);
					\draw[line width =1pt] (0.6,0.9) -- (0.4,1.1);
					\draw[line width =1pt] (1.4,0.9) -- (1.6,1.1);
					\draw[line width =1pt] (1.6,0.9) -- (1.4,1.1);
					\filldraw[draw=black,fill=white] (0.9,-0.1) rectangle (1.1, 0.1);
					\filldraw[draw=black,fill=white] (0.9,1.9) rectangle (1.1, 2.1);
					\def\circleCenter{(1,1)}
					\def\circleRadius{1}
					\draw[line width =1pt] plot[smooth] coordinates {(1,0)  (1.8, 1) (1.5, 1.4) (1.2, 1) (1,0.2)};
					\draw[line width =1pt] plot[smooth] coordinates {(1,2)  (0.2, 1) (0.5, 0.6) (0.8, 1) (1,1.8)};
					\filldraw[draw=black,fill=white] (0.9,-0.1) rectangle (1.1, 0.1);
					\filldraw[draw=black,fill=white] (0.9,1.9) rectangle (1.1, 2.1);
				\end{tikzpicture}
			}\vspace{1mm}\\
			&\overset{\text{by \eqref{eq-same-height}}}{=}& 
			\raisebox{-0.4in}{
				\begin{tikzpicture}
					\filldraw[draw=white,fill=gray!20] (0,0) rectangle (2, 2);
					\def\circleCenter{(1,1)}
					\def\circleRadius{1}
					\draw[line width =1pt] \circleCenter ++(270:\circleRadius) arc (270:90:\circleRadius);
					\node (A) at (0.7,1.3) [left] {$\alpha_1$};
					\node (B) at (1.3,0.7) [right] {$\alpha_3$};
					\draw[draw=black]  (0.5,1) circle (0.1);
					\draw[line width =1pt] (0.4,0.9) -- (0.6,1.1);
					\draw[line width =1pt] (0.6,0.9) -- (0.4,1.1);
					\draw[line width =1pt] (1.4,0.9) -- (1.6,1.1);
					\draw[line width =1pt] (1.6,0.9) -- (1.4,1.1);
					\filldraw[draw=black,fill=white] (0.9,-0.1) rectangle (1.1, 0.1);
					\filldraw[draw=black,fill=white] (0.9,1.9) rectangle (1.1, 2.1);
					\draw[line width =1pt] \circleCenter ++(90:\circleRadius) arc (90:-90:\circleRadius);
					\filldraw[draw=black,fill=white] (0.9,-0.1) rectangle (1.1, 0.1);
					\filldraw[draw=black,fill=white] (0.9,1.9) rectangle (1.1, 2.1);
				\end{tikzpicture}
			} 
			+
			q^2
			\raisebox{-.4in}{
				\begin{tikzpicture}
					\filldraw[draw=white,fill=gray!20] (0,0) rectangle (2, 2);
					\draw[draw=black]  (0.5,1) circle (0.1);
					\draw[line width =1pt] (0.4,0.9) -- (0.6,1.1);
					\draw[line width =1pt] (0.6,0.9) -- (0.4,1.1);
					\draw[line width =1pt] (1.4,0.9) -- (1.6,1.1);
					\draw[line width =1pt] (1.6,0.9) -- (1.4,1.1);
					\filldraw[draw=black,fill=white] (0.9,-0.1) rectangle (1.1, 0.1);
					\filldraw[draw=black,fill=white] (0.9,1.9) rectangle (1.1, 2.1);
					\def\circleCenter{(1,1)}
					\def\circleRadius{1}
					\draw[line width =1pt] plot[smooth] coordinates {(1,0)  (1.8, 1) (1.5, 1.4) (1.2, 1) (1,0)};
					\draw[line width =1pt] plot[smooth] coordinates {(1,2)  (0.2, 1) (0.5, 0.6) (0.8, 1) (1,2)};
					\filldraw[draw=black,fill=white] (0.9,-0.1) rectangle (1.1, 0.1);
					\filldraw[draw=black,fill=white] (0.9,1.9) rectangle (1.1, 2.1);
				\end{tikzpicture}
			},
		\end{array}
	\end{split}
\end{equation*}
where $\begin{tikzpicture}
	\draw[draw=black]  (1,1) circle (0.1);
	\draw [line width =1pt] (0.9,0.9) -- (1.1, 1.1);
	\draw [line width =1pt] (0.9,1.1) -- (1.1, 0.9);
\end{tikzpicture}, \begin{tikzpicture}
	\draw [line width =1pt] (0.9,0.9) -- (1.1, 1.1);
	\draw [line width =1pt] (0.9,1.1) -- (1.1, 0.9);
\end{tikzpicture}\in \{\begin{tikzpicture}
	\filldraw[draw=black,fill=white]  (0.5,0) circle (0.1);
\end{tikzpicture}, \begin{tikzpicture}
	\draw[draw=black,fill=white] (0.4,-0.1) rectangle (0.6,0.1);
\end{tikzpicture}\}.$

That is 
\begin{equation}\label{eq:mu8}
	A_{\alpha'}A_\alpha=[A_{\alpha_1} A_{\alpha_3}]+q^2 [A_{\alpha_2}^{b^\Delta_{\alpha_2,\alpha}}A_{\alpha_4}^{b^\Delta_{\alpha_4,\alpha}}].
\end{equation}
Then by \eqref{eq:lambda8} \eqref{eq:lambda81} and \eqref{eq:mu8}, we obtain

\begin{equation}
	A_{\alpha'}=[A_{\alpha_1} A_{\alpha_3}A_\alpha^{-1}]+ [A_{\alpha_2}^{b^\Delta_{\alpha_2,\alpha}}A_{\alpha_4}^{b^\Delta_{\alpha_4,\alpha}}A_\alpha^{-1}].
\end{equation}

{\bf Case 3.3.2}. Neither $\alpha_3$ nor $\alpha_4$ has weight $1$. Then $\alpha,\alpha_1$ are loops. Then we have
\begin{equation}\label{eq:lambda9}
	\Lambda_{\alpha_1,\alpha}=\Lambda_{\alpha_3,\alpha}=\Lambda_{\alpha_4,\alpha}=0,\quad \Lambda_{\alpha_2,\alpha}=2{\rm max}\{1,w(\alpha_2)\}. 
\end{equation}

It follows that 
\begin{equation}\label{eq:lambda91}
	b^\Delta_{\alpha_2,\alpha}\Lambda_{\alpha_2,\alpha}=4. 
\end{equation}
By relation \eqref{cross}, we have 

\begin{equation}\label{eqqqq-9}
	\begin{split}
		\begin{array}{rcl}
			\raisebox{-.4in}{
				\begin{tikzpicture}
					\filldraw[draw=white,fill=gray!20] (0,0) rectangle (2, 2);
					\node (A) at (1.9,0.6) [left] {$\alpha'$};
					\node (B) at (0.2,0.6) [right] {$\alpha$};
					\filldraw[draw=black,fill=white] (0.9,-0.1) rectangle (1.1, 0.1);
					\def\circleCenter{(1,0.7)}
					\def\circleRadius{0.7}
					\draw[line width =1pt] (0.2,0.9) -- (0.4,1.1);
					\draw[line width =1pt] (0.4,0.9) -- (0.2,1.1);
					\draw[line width =1pt] (0.9,0.9) -- (0.9,1.1);
					\draw[line width =1pt] (1.1,0.9) -- (0.9,0.9);
					\draw[line width =1pt] (1.1,0.9) -- (1.1,1.1);
					\draw[line width =1pt] (1.1,1.1) -- (0.9,1.1);
					\draw[line width =1pt] (1.6,0.9) -- (1.8,1.1);
					\draw[line width =1pt] (1.8,0.9) -- (1.6,1.1);
					\draw[draw=black]  (0.3,1) circle (0.1);
					\draw[line width =1pt] plot[smooth] coordinates {(1,0) (0.5,0.2) (0.3,0.4)  (0.1, 1) (0.3, 1.5) (0.65,1.5) (0.8,1.45) (0.9,1.4)};
					\draw[line width =1pt] plot[smooth] coordinates {(1.1,1.3) (1.2,1) (1.25, 0.5) (1,0)};
					\draw[line width =1pt] plot[smooth] coordinates {(1,0) (1.5,0.2) (1.7,0.4)  (1.9, 1) (1.7, 1.5) (1.35,1.5) (1.2,1.45) (1,1.35) (0.8,1) (0.75, 0.5) (1,0)};
					\filldraw[draw=black,fill=white] (0.9,-0.1) rectangle (1.1, 0.1);
				\end{tikzpicture}
			}
			&=&q^{-1}
			\raisebox{-0.4in}{
				\begin{tikzpicture}
					\filldraw[draw=white,fill=gray!20] (0,0) rectangle (2, 2);
					\draw[line width =1pt] (0.2,0.9) -- (0.4,1.1);
					\draw[line width =1pt] (0.4,0.9) -- (0.2,1.1);
					\draw[line width =1pt] (0.9,0.9) -- (0.9,1.1);
					\draw[line width =1pt] (1.1,0.9) -- (0.9,0.9);
					\draw[line width =1pt] (1.1,0.9) -- (1.1,1.1);
					\draw[line width =1pt] (1.1,1.1) -- (0.9,1.1);
					\draw[line width =1pt] (1.6,0.9) -- (1.8,1.1);
					\draw[line width =1pt] (1.8,0.9) -- (1.6,1.1);
					\draw[draw=black]  (0.3,1) circle (0.1);
					\def\circleCenter{(1,1)}
					\def\circleRadius{1}
					\draw[line width =1pt] \circleCenter ++(270:\circleRadius) arc (270:-90:\circleRadius);
					\draw[line width =1pt] plot[smooth] coordinates {(1,0) (0.75, 0.5) (0.8,1) (1,1.35)  (1.2,1) (1.25, 0.5) (1,0)};
					\filldraw[draw=black,fill=white] (0.9,-0.1) rectangle (1.1, 0.1);
					\node  at (1.2,0.3) {$1$};
					\node  at (1.5,0.15) {$2$};
					\node  at (0.8,0.3) {$2$};
					\node  at (0.5,0.15) {$1$};
					\filldraw[draw=black,fill=white] (0.9,-0.1) rectangle (1.1, 0.1);
				\end{tikzpicture}
			} 
			+
			q
			\raisebox{-.4in}{
				\begin{tikzpicture}
					\filldraw[draw=white,fill=gray!20] (0,0) rectangle (2, 2);
					\draw[line width =1pt] (0.2,0.9) -- (0.4,1.1);
					\draw[line width =1pt] (0.4,0.9) -- (0.2,1.1);
					\draw[line width =1pt] (0.9,0.9) -- (0.9,1.1);
					\draw[line width =1pt] (1.1,0.9) -- (0.9,0.9);
					\draw[line width =1pt] (1.1,0.9) -- (1.1,1.1);
					\draw[line width =1pt] (1.1,1.1) -- (0.9,1.1);
					\draw[line width =1pt] (1.6,0.9) -- (1.8,1.1);
					\draw[line width =1pt] (1.8,0.9) -- (1.6,1.1);
					\draw[draw=black]  (0.3,1) circle (0.1);
					\filldraw[draw=black,fill=white] (0.9,-0.1) rectangle (1.1, 0.1);
					\def\circleCenter{(1,1)}
					\def\circleRadius{1}
					\draw[line width =1pt] plot[smooth] coordinates {(1,0) (1.5,0.25) (1.7,0.45)  (1.9, 1) (1.7,1.3) (1.5,1.2)  (1,0)};
					\draw[line width =1pt] plot[smooth] coordinates {(1,0) (0.5,0.25) (0.3,0.45)  (0.1, 1) (0.3,1.3) (0.5,1.2)  (1,0)};
					\node  at (1.15,0.4) {$1$};
					\node  at (1.35,0.15) {$2$};
					\node  at (0.85,0.4) {$2$};
					\node  at (0.65,0.15) {$1$};
					\filldraw[draw=black,fill=white] (0.9,-0.1) rectangle (1.1, 0.1);
				\end{tikzpicture}
			}\vspace{1mm}\\
			&\overset{\text{by \eqref{eq-same-height}}}{=}& 
			\raisebox{-0.4in}{
				\begin{tikzpicture}
					\filldraw[draw=white,fill=gray!20] (0,0) rectangle (2, 2);
					\draw[line width =1pt] (0.2,0.9) -- (0.4,1.1);
					\draw[line width =1pt] (0.4,0.9) -- (0.2,1.1);
					\draw[line width =1pt] (0.9,0.9) -- (0.9,1.1);
					\draw[line width =1pt] (1.1,0.9) -- (0.9,0.9);
					\draw[line width =1pt] (1.1,0.9) -- (1.1,1.1);
					\draw[line width =1pt] (1.1,1.1) -- (0.9,1.1);
					\draw[line width =1pt] (1.6,0.9) -- (1.8,1.1);
					\draw[line width =1pt] (1.8,0.9) -- (1.6,1.1);
					\draw[draw=black]  (0.3,1) circle (0.1);
					\def\circleCenter{(1,1)}
					\def\circleRadius{1}
					\draw[line width =1pt] \circleCenter ++(270:\circleRadius) arc (270:-90:\circleRadius);
					\draw[line width =1pt] plot[smooth] coordinates {(1,0) (0.75, 0.5) (0.8,1) (1,1.35)  (1.2,1) (1.25, 0.5) (1,0)};
					\filldraw[draw=black,fill=white] (0.9,-0.1) rectangle (1.1, 0.1);
				\end{tikzpicture}
			} 
			+
			q^2
			\raisebox{-.4in}{
				\begin{tikzpicture}
					\filldraw[draw=white,fill=gray!20] (0,0) rectangle (2, 2);
					\draw[line width =1pt] (0.2,0.9) -- (0.4,1.1);
					\draw[line width =1pt] (0.4,0.9) -- (0.2,1.1);
					\draw[line width =1pt] (0.9,0.9) -- (0.9,1.1);
					\draw[line width =1pt] (1.1,0.9) -- (0.9,0.9);
					\draw[line width =1pt] (1.1,0.9) -- (1.1,1.1);
					\draw[line width =1pt] (1.1,1.1) -- (0.9,1.1);
					\draw[line width =1pt] (1.6,0.9) -- (1.8,1.1);
					\draw[line width =1pt] (1.8,0.9) -- (1.6,1.1);
					\draw[draw=black]  (0.3,1) circle (0.1);
					\filldraw[draw=black,fill=white] (0.9,-0.1) rectangle (1.1, 0.1);
					\def\circleCenter{(1,1)}
					\def\circleRadius{1}
					\draw[line width =1pt] plot[smooth] coordinates {(1,0) (1.5,0.25) (1.7,0.45)  (1.9, 1) (1.7,1.3) (1.5,1.2)  (1,0)};
					\draw[line width =1pt] plot[smooth] coordinates {(1,0) (0.5,0.25) (0.3,0.45)  (0.1, 1) (0.3,1.3) (0.5,1.2)  (1,0)};
					\filldraw[draw=black,fill=white] (0.9,-0.1) rectangle (1.1, 0.1);
				\end{tikzpicture}
			},
		\end{array}
	\end{split}
\end{equation}
where $\begin{tikzpicture}
	\draw[line width =1pt] (0.9,0.9) -- (0.9,1.1);
	\draw[line width =1pt] (1.1,0.9) -- (0.9,0.9);
	\draw[line width =1pt] (1.1,0.9) -- (1.1,1.1);
	\draw[line width =1pt] (1.1,1.1) -- (0.9,1.1);
\end{tikzpicture},\begin{tikzpicture}
	\draw[draw=black]  (1,1) circle (0.1);
	\draw [line width =1pt] (0.9,0.9) -- (1.1, 1.1);
	\draw [line width =1pt] (0.9,1.1) -- (1.1, 0.9);
\end{tikzpicture}, \begin{tikzpicture}
	\draw [line width =1pt] (0.9,0.9) -- (1.1, 1.1);
	\draw [line width =1pt] (0.9,1.1) -- (1.1, 0.9);
\end{tikzpicture}\in \{\begin{tikzpicture}
	\filldraw[draw=black,fill=white]  (0.5,0) circle (0.1);
\end{tikzpicture}, \begin{tikzpicture}
	\draw[draw=black,fill=white] (0.4,-0.1) rectangle (0.6,0.1);
\end{tikzpicture}\}.$

That is 
\begin{equation}\label{eq:mu9}
	A_{\alpha'}A_\alpha=[A_{\alpha_1} A_{\alpha_3}^{-b^\Delta_{\alpha_3,\alpha}}]+q^2 [A_{\alpha_2}^{b^\Delta_{\alpha_2,\alpha}}A_{\alpha_4}^{b^\Delta_{\alpha_4,\alpha}}].
\end{equation}
Then by \eqref{eq:lambda9} \eqref{eq:lambda91} and \eqref{eq:mu9}, we obtain

\begin{equation*}
	A_{\alpha'}=[A_{\alpha_1} A_{\alpha_3}^{-b^\Delta_{\alpha_3,\alpha}}A_\alpha^{-1}]+ [A_{\alpha_2}^{b^\Delta_{\alpha_2,\alpha}}A_{\alpha_4}^{b^\Delta_{\alpha_4,\alpha}}A_\alpha^{-1}].
\end{equation*}

This completes the proof of \eqref{eq:seedmutation} in the case $w(\alpha)=1$.

The proof is complete.
\end{proof}

\subsection{$\overline {\cS}_q(\fS)=\mathscr A_q(\fS)$}

From~\cite{FST1}, any two quasi-triangulations of $\fS$ are related by a sequence of flips. In light of Proposition~\ref{prop:mutation}, the generalized quantum cluster algebra associated with the seed ${\bf s}(\Delta)$ is independent of the choice of quasi-triangulation $\Delta$. We denote this algebra with ground ring $\mathfrak R=\widetilde R$ (see Definition \ref{def:qca} and \eqref{eq:groundring}) by $\mathscr{A}_q(\fS) = \mathscr{A}_q({\bf s}(\Delta))$ and call it the \textbf{generalized quantum cluster algebra} associated with $\fS$.

The following result is an immediate corollary.

\begin{corollary}\label{cor:bi}
Let $\fS$ be a quasi-triangulable punctured surface. There is a natural bijection between:
\begin{enumerate}[label={\rm (\alph*)}]
	\item \{\text{Exchangeable quantum cluster variables of }$\mathscr{A}_q(\fS)$\} and \{\text{Interior admissible ideal arcs in }$\fS$\};
	\item \{\text{Frozen cluster variables of }$\mathscr{A}_q(\fS)$\} and \{\text{Boundary ideal arcs in }$\fS$\};
	\item \{\text{Coefficients of }$\mathscr{A}_q(\fS)$\} and \{\text{Peripheral loops in }$\fS$\};
	\item \{\text{Quantum clusters of }$\mathscr{A}_q(\fS)$\} and \{\text{Quasi-triangulations of }$\fS$\}.
\end{enumerate}
In particular, these bijections are compatible with mutation and flip.
\end{corollary}

\begin{lemma}\label{lem-arc-in-A}
Let $\fS$ be a quasi-triangulable punctured surface, and let
$\alpha\in\rdS$ be a stated simple essential arc (Definition~\ref{def-arcs}). Then
$\alpha\in\mathscr A_q(\fS).$
\end{lemma}

\begin{proof}
If $\alpha$ is positively stated, then it is a cluster variable of
$\mathscr A_q(\fS)$ and hence belongs to $\mathscr A_q(\fS)$.

If at least one endpoint of $\alpha$ is negatively stated, the result
follows from Lemma~\ref{lem-m-arc} together with the positively stated
case.
\end{proof}

\begin{theorem}\label{thm:skein=A}
Let $\fS$ be a quasi-triangulable punctured surface. Then we have 
\begin{equation*}
	\overline {\cS}_q(\fS)=\mathscr A_q(\fS). 
\end{equation*}

\end{theorem}

\begin{proof}
By Corollary \ref{cor:bi}, we have $\mathscr A_q(\fS)\subset \overline {\cS}_q(\fS)$. Conversely, by Corollary \ref{cor:gen} and Lemma~\ref{lem-arc-in-A}, we see that $\overline {\cS}_q(\fS)\subset \mathscr A_q(\fS)$. Thus we have  
\begin{equation*}
	\overline {\cS}_q(\fS)= \mathscr A_q(\fS),
\end{equation*}
which establishes the proof.

\end{proof}

\begin{remark}
In view of Lemma \ref{lem-positive-reduced}, for any quasi-triangulable punctured surface $\fS$, the Muller--Roger--Yang skein algebra $\cS_q^+(\fS)$ admits a generalized quantum cluster algebra structure whose frozen variables are not invertible.
\end{remark}

Determining whether a given (generalized) cluster algebra is finitely generated is generally a nontrivial task.

The following result, however, provides a clear answer in the case of quasi-triangulable punctured surfaces. It is an immediate consequence of Theorem~\ref{thm:skein=A} and Corollary~\ref{cor:gen}.

\begin{corollary}\label{cor:fg}
The generalized quantum cluster algebra associated with a quasi-triangulable punctured surface is always finitely generated.
\end{corollary}

\begin{remark}
In light of Remark \ref{rem:thesame}, the generalized quantum cluster algebra associated with an unpunctured orbifold is always finitely generated.
\end{remark}

Let $S\subset \mathring{\mathcal P}$, and let $w_p\in R$ for each $p\in S$.
Following Frohman, Kania-Bartoszynska and L{\^e} \cite{FKL}, for a punctured surface $\fS$, we define the \textbf{partially sliced reduced L{\^e}-Roger-Yang skein algebra} $\overline {\cS}^{\bf w}_q(\fS)$ to be 
\begin{equation}\label{eq-sliced-skein}
\overline {\cS}^{{\bf w}}_q(\fS):=\overline {\cS}_q(\fS)/(\alpha_p-w_p\mid p\in S),  
\end{equation} 
where ${\bf w}=(w_p)_{p\in S}$.

The following result is an immediate corollary of Theorem~\ref{thm:skein=A}.

\begin{corollary}\label{thm-sliced-cluster}
The partially sliced reduced L{\^e}-Roger-Yang skein algebra $\overline {\cS}^{\bf w}_q(\fS)$ associated with a quasi-triangulable punctured surface $\fS$ is a generalized quantum cluster algebra.
\end{corollary}

\section{Frobenius maps for (reduced) L{\^e}--Roger--Yang skein algebras}\label{Sec-Fro}

As shown in \S\ref{sec-fro-cluster}, if $q^2$ is a root of unity of odd order $\ell$, then the generalized quantum cluster algebra $\mathscr A_q(\fS)$ admits a Frobenius map. Combining this construction with Theorem~\ref{thm:skein=A}, we obtain a Frobenius map
\[
\cF\colon \rd_\eta(\fS)\longrightarrow \rd_q(\fS),
\qquad \eta^{\frac 1 2}=q^{\frac{\ell^2}{2}},
\]
for every quasi-triangulable punctured surface $\fS$; see Corollary~\ref{Cor-Fro-first}.
We then extend this construction to arbitrary punctured surfaces in Theorem~\ref{thm-Fro-main}. We also give an explicit description of the action of $\cF$ on stated $\fS$-tangles in terms of threading: each stated essential arc (Definition~\ref{def-arcs}) is threaded by the monomial $x^\ell$, while each loop or circle arc is threaded by the Chebyshev polynomial $T_\ell$ defined in \eqref{Chebyshev-polynomial} (see \eqref{eq-image-all-diagrams}). Finally, we prove that the Frobenius homomorphism is compatible with the cutting homomorphisms and the quantum trace maps in Theorem~\ref{thm-Fro-main}.

\vs

We state the following root of unity assumption, which will be used repeatedly.

\begin{itemize}
\item[\Rlabel{A}]:  
$q^{\frac{1}{2}}\in R$ is a root of unity such that the order of $q^2$ has odd order $\ell$.  
Set $ \eta^{\frac{1}{2}} = (q^{\frac{1}{2}})^{\ell^2}$. 
\label{RA}
\end{itemize}

\subsection{Threading of stated $\partial \fS$-tangles}

Let $\alpha$ be a connected (stated) $\fS$-tangle. For each
positive integer $m$, let $\alpha^{[m]}$ denote the disjoint union of
$m$ parallel copies of $\alpha$, taken in the direction of its framing.
We require that $\alpha^{[m]}$ be contained in a sufficiently small open
tubular neighborhood of $\alpha$. Define
\begin{align}
\alpha^{(m)}
:= \left(\prod_{v\in C(\alpha)} v^{\frac{m-1}{2}} \right) \alpha^{[m]},\qquad
\text{(see Definition~\ref{def-arcs} for $C(\alpha)$).}
\end{align}
For a polynomial
$P(x)=\sum_{i=1}^{n}c_i x^i\in \mathbb Z[x],$
define
\begin{align}
\alpha^{(P)}
:=
\sum_{i=1}^{n}c_i\alpha^{(i)}.
\end{align}

We use $\begin{array}{c}
\includegraphics[scale=0.7]{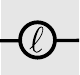}
\end{array}$ (resp. $\begin{array}{c}
\includegraphics[scale=0.7]{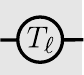}
\end{array}$) to denote 
$\left(\begin{array}{c}
\includegraphics[scale=0.7]{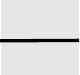}
\end{array} \right)^{(\ell)}$ (resp. $\left(\begin{array}{c}
\includegraphics[scale=0.7]{line.pdf}
\end{array} \right)^{(T_\ell)}$).
In lemmas~\ref{lem-relation-51}-\ref{lem-relation-53}, we state relations regarding these threadings.
These relations will be later to construct Frobenius map and to compute the center.

\begin{lemma}{\cite[Equations~(10) and (41)]{yu2023center}}\label{lem-relation-51}
Under Assumption~\Rlabel{A},
in $\cS_q(\fS)$, for any $\mu,\nu\in\{\pm1\}$, we have
\begin{align}\label{eq-height-l}
	\begin{array}{c}
		\includegraphics[scale=0.7]{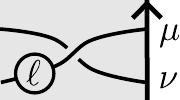}
	\end{array}
	=
	\eta
	\begin{array}{c}
		\includegraphics[scale=0.7]{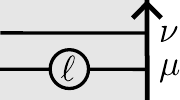}
	\end{array}.
\end{align}
\end{lemma}

\begin{lemma}
Under Assumption~\Rlabel{A},
the following relations hold in $\cS_q(\fS)$.

\begin{enumerate}
	\item [\rm (a)] We have 
	\begin{align}\label{eq-height-3}
		\begin{array}{c}
			\includegraphics[scale=0.7]{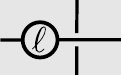}
		\end{array}
		=
		\begin{array}{c}
			\includegraphics[scale=0.7]{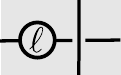}
		\end{array},
	\end{align}
	where the horizontal strand is a part of an essential arc (Definition~\ref{def-arcs}).
	
	\item [\rm (b)] We have \begin{align}\label{eq-height-33}
		\begin{array}{c}
			\includegraphics[scale=0.7]{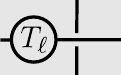}
		\end{array}
		=
		\begin{array}{c}
			\includegraphics[scale=0.7]{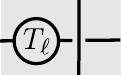}
		\end{array},
	\end{align}
	where the horizontal strand is a part of a loop or a circle arc (Definition~\ref{def-arcs}).
\end{enumerate}

\end{lemma}

\begin{proof}
$({\rm a})$ This follows from \cite[Theorem~4.10]{BL22} and \cite[Lemma~3.4]{karuo2026center}.

$({\rm b})$ This follows from \cite[Theorem 4.10]{BL22} and \cite[Theorem~5.1]{karuo2026center}.

\end{proof}

\begin{lemma}\label{lem-relation-53}
Under Assumption~\Rlabel{A},
the following relations hold in $\cS_q(\fS)$.

\begin{enumerate}
	\item [\rm(a)] We have
	\begin{align}
		\label{eq-height-2}
		\begin{array}{c}
			\includegraphics[scale=0.7]{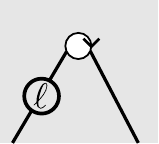}
		\end{array}
		=
		\eta
		\begin{array}{c}
			\includegraphics[scale=0.7]{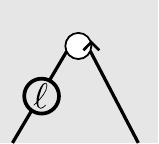}
		\end{array},
	\end{align}
	where the left strand is part of an edge-circle arc (Definition~\ref{def-arcs}).
	
	\item [\rm(b)] We have
	\begin{align}
		\label{eq-height-21}
		\begin{array}{c}
			\includegraphics[scale=0.7]{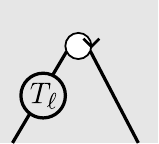}
		\end{array}
		=
		\eta
		\begin{array}{c}
			\includegraphics[scale=0.7]{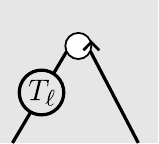}
		\end{array},
	\end{align}
	where the left strand is part of a circle arc (Definition~\ref{def-arcs}).
\end{enumerate}
\end{lemma}

\begin{proof}

$({\rm a})$  The proof of \cite[Lemma~3.4]{karuo2026center} 
implies that
\begin{align*}
	\begin{array}{c}
		\includegraphics[scale=0.7]{arc-circle1.pdf}
	\end{array}
	&=
	q^{\frac{n}{2}}X_1+q^{-\frac{n}{2}}X_2,
	&
	\begin{array}{c}
		\includegraphics[scale=0.7]{arc-circle2.pdf}
	\end{array}
	&=
	q^{-\frac{n}{2}}X_1+q^{\frac{n}{2}}X_2
\end{align*}
for some $X_1,X_2\in\cS_q(\fS)$.

Since $q^{2n}=1$, this yields
\[
\begin{array}{c}
	\includegraphics[scale=0.7]{arc-circle1.pdf}
\end{array}
=
q^n
\begin{array}{c}
	\includegraphics[scale=0.7]{arc-circle2.pdf}
\end{array}
=
\eta
\begin{array}{c}
	\includegraphics[scale=0.7]{arc-circle2.pdf}
\end{array},
\]
as desired.

$({\rm b})$ This follows from
\cite[Equation~(4.4)]{karuo2024positivity} by the same argument
as in part $({\rm a})$.
\end{proof}

Note that $q^\ell = \pm 1$. Since $\ell$ is odd, we have 
\begin{align}\label{eq-eta-ql}
\eta = q^{\ell^2} = q^\ell=\eta^\ell\in\{\pm 1\}.
\end{align}

\begin{remark}\label{rem-com-height1}
Under Assumption~\Rlabel{A},
Relations~\eqref{eq-height-l}--\eqref{eq-height-21} also hold in
$\cS_\eta(\fS)$ after removing the circle
$\begin{array}{c}
	\includegraphics[scale=0.7]{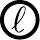}
\end{array}$ or 
$ \begin{array}{c}
	\includegraphics[scale=0.7]{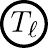}
\end{array}$
from the corresponding relations. 
\end{remark}

\begin{remark}\label{rem-com-height}
Under Assumption~\Rlabel{A}, we have
$\eta^\ell=\eta$. Hence, the height exchange relations in
$\cS_\eta(\fS)$ are compatible with those in $\cS_q(\fS)$ for
$\ell$ parallel copies of stated essential arcs and for
$T_\ell$-threadings of loops and circle arcs.
\end{remark}

\subsection{Construction of the Frobenius maps}
In this subsection, we establish the Frobenius map of the reduced L{\^e}–Roger–Yang skein algebra.
The following Corollary follows directly from
Corollary~\ref{cor-Fro-cluster} and
Theorem~\ref{thm:skein=A}.

\begin{corollary}\label{Cor-Fro-first}
Let $\fS$ be a quasi-triangulable punctured surface (Definition~\ref{def-quasitri}). Under Assumption~\Rlabel{A}, 
there exists a unique $\wR$-algebra (see \eqref{eq-Rtilde}) embedding 
$$\cF\colon \rd_\eta (\fS)\rightarrow \rdS$$
such that the following diagram commutes:
\begin{equation}\label{diag-Fro-cluster}
	\begin{tikzcd}[column sep=large, row sep=large]
		\rd_\eta(\fS)
		\arrow[r, "\cF"]
		\arrow[d, "\cong"']
		&
		\rd_q(\fS)
		\arrow[d, "\cong"]
		\\
		\mathscr A_\eta(\fS)
		\arrow[r, "\mathscr F"']
		&
		\mathscr A_q(\fS)
	\end{tikzcd}.
\end{equation}

\end{corollary}

\def\bF{\mathbb{F}}

The commutative diagram~\eqref{diag-Fro-cluster} implies that
\begin{align}\label{eq-image-fro-arc}
\cF(A_\alpha)=A_\alpha^\ell\qquad
\text{(see \eqref{eq-D-X-q})}
\end{align}
for every admissible ideal arc $\alpha$. We next determine the image of an arbitrary stated $\fS$-tangle under the Frobenius map.
The following lemma shows that the image of a
stated simple edge-circle arc $\alpha$ under the Frobenius map is given by $\alpha^{(\ell)}$.

\begin{lemma}\label{lem-Fro-arc-circle}
Let $\alpha$ be a stated simple edge-circle arc in $\fS$ (Definition~\ref{def-arcs}). Then
\[
\cF(\alpha)=\alpha^{(\ell)}. 
\]
\end{lemma}

\begin{proof}
We consider the two possible states of $\alpha$.

\noindent\textbf{Case 1: The state of $\alpha$ is $+$.}
The result follows directly from \eqref{eq-image-fro-arc}.

\noindent\textbf{Case 2: The state of $\alpha$ is $-$.}
The result follows from Case~1 together with
Lemma~\ref{lem-m-arc}.
\end{proof}

Before we continue to determine the images of general stated $\fS$-tangles under the Frobenius map, we establish its compatibility with the cutting homomorphisms. We then use this compatibility to compute $\cF(\alpha)$ when $\alpha$ is a simple circle arc; see Proposition~\ref{prop-Fro-arc-T}.

\vs

Let $\lambda$ be an ideal triangulation of $\fS$. The Frobenius
map for the $\overline{\mathcal {X}}$-quantum torus is defined by
\[
\mathbb F\colon
\overline{\mathcal {X}}_\eta(\fS,\lambda)
\longrightarrow
\overline{\mathcal {X}}_q(\fS,\lambda),
\qquad
x_e\longmapsto x_e^\ell,
\quad e\in\lambda.
\]
Similarly, the $\mathcal A$-quantum torus admits a Frobenius
map defined on its generators by taking their $\ell$-th
powers.

The construction of the Frobenius map $\mathscr F$, together with
the commutative diagram \eqref{diag-Fro-cluster}, shows that $\cF$ is
compatible with the $\mathcal A$-quantum trace map. Combining this
compatibility with \eqref{eq-com-traces-AX}, we obtain the following
result.

\begin{lemma}\label{lem-Fro-trace}
Under the assumption of Corollary~\ref{Cor-Fro-first},
the following
diagram commutes:
\begin{equation}
	\label{diag-Fro-trace}
	\begin{tikzcd}
		\rd_\eta(\fS)
		\arrow[r, "\cF"]
		\arrow[d, "\tr_\lambda"']
		&
		\rd_q(\fS)
		\arrow[d, "\tr_\lambda"]
		\\
		\overline{\mathcal {X}}_\eta(\fS,\lambda)
		\arrow[r, "\mathbb F"']
		&
		\overline{\mathcal {X}}_q(\fS,\lambda).
	\end{tikzcd}
\end{equation}
\end{lemma}

Since all the maps appearing in
\eqref{thm-eq-cut-trace} and \eqref{diag-Fro-trace} are injective,
the commutativity of these two diagrams yields the following result.

\begin{lemma}\label{lem-Fro-Cut}
Let $e$ be an ideal arc which is not a boundary edge. Under the assumption of Corollary~\ref{Cor-Fro-first}, the following
diagram commutes:
\begin{equation}
	\label{diag-Fro-cut}
	\begin{tikzcd}
		\rd_\eta(\fS)
		\arrow[r, "\cF"]
		\arrow[d, "\overline \Theta_e"']
		&
		\rd_q(\fS)
		\arrow[d, "\overline \Theta_e"]
		\\
		\rd_\eta(\Cut_e\fS)
		\arrow[r, "\cF"']
		&
		\rd_q(\Cut_e\fS).
	\end{tikzcd}
\end{equation} 
\end{lemma}

A direct computation yields
\begin{equation}
\label{eq-uminusplus}
\begin{array}{c}
	\includegraphics[scale=1]{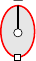}
\end{array}
\begin{array}{c}
	\includegraphics[scale=1]{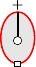}
\end{array}
=
v^{-1}
\in \rd_q(\mathfrak M),
\end{equation}
where $\mathfrak M$ is the punctured surface depicted in the figures
and $v$ denotes its boundary component shown therein.
The following Proposition shows that the image of a simple circle arc under the Frobenius map is given by threading it along the Chebyshev polynomial $T_\ell$.

\begin{proposition}\label{prop-Fro-arc-T}
Let $\alpha$ be a simple arc connecting two distinct boundary circles $v$ and $w$. Under the assumption of Corollary~\ref{Cor-Fro-first}, we have
\begin{align}\label{eq-Fro-arc-T}
	\cF(\alpha)
	= \alpha^{(T_\ell)}.
\end{align}
\end{proposition}

\begin{proof}
Let $e_1$ and $e_2$ be ideal arcs surrounding the circular boundary
components $v$ and $w$, respectively; see
Figure~\ref{fig:alpha-cut}. Set
\[
\Theta:=\overline\Theta_{e_1}\circ\overline\Theta_{e_2}
\qquad\text{and}\qquad
\fS':=\Cut_{e_1}\bigl(\Cut_{e_2}(\fS)\bigr).
\]
Then
\[
\Theta(\alpha)=A+B+C,
\]
where
\[
A=
\begin{array}{c}
	\includegraphics[scale=1]{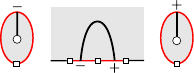}
\end{array},
\qquad
B=
\begin{array}{c}
	\includegraphics[scale=1]{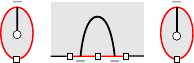}
\end{array},
\qquad
C=
\begin{array}{c}
	\includegraphics[scale=1]{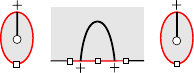}
\end{array}.
\]

Equation~\eqref{eq-uminusplus} implies that
\begin{equation}\label{eq-ABC}
	AB=q^{-2}BA
	\qquad\text{and}\qquad
	BC=CB=v^{-1}w^{-1}\in\rd_q(\fS').
\end{equation}
Regarding $T_l\!\left(\sqrt{vw}\,\alpha\right)$ as an element of
$\rd_q(\fS)$, we obtain
\begin{align}
	\label{eq-Theta-Tl}
	\Theta\left(T_l\!\left(\sqrt{vw}\,\alpha\right)\right)
	&=
	T_l\!\left(\sqrt{vw}\,\Theta(\alpha)\right)
	\nonumber\\
	&=
	T_l\!\left(
	\sqrt{vw}\,A+\sqrt{vw}\,B+\sqrt{vw}\,C
	\right) 
	\nonumber\\
	&=
	(vw)^{\frac{l}{2}}
	\left(A^l+B^l+C^l\right) \qquad \text{(by Lemma~\ref{lem-root-Tl} and \eqref{eq-ABC})}.
\end{align}

On the other hand,
\begin{align*}
	\Theta\bigl(\cF(\alpha)\bigr)
	&\overset{\eqref{diag-Fro-cut}}{=}
	\cF\bigl(\Theta(\alpha)\bigr)\\
	&=
	\cF(A+B+C)\\
	&\overset{\eqref{diag-Fro-cluster}}{=}
	(vw)^{\frac{l-1}{2}}
	\left(A^l+B^l+C^l\right)\\
	&\overset{\eqref{eq-Theta-Tl}}{=}
	\Theta\left(
	\frac{1}{\sqrt{vw}}\,
	T_l\!\left(\sqrt{vw}\,\alpha\right)
	\right) = \Theta\left(\alpha^{(T_\ell)}\right).
\end{align*}
Since $\Theta$ is injective, it follows that
$\cF(\alpha)
=\alpha^{(T_\ell)},$
which completes the proof.
\end{proof}

\begin{figure}
\centering
\includegraphics[width=0.2\linewidth]{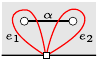}
\caption{The ideal arcs $e_1$ and $e_2$.}
\label{fig:alpha-cut}
\end{figure}

When $\fS$ has no boundary circle, \cite{BL22}
constructed the Frobenius map
\begin{align}\label{eq-Fro-stated}
\Phi\colon \rd_\eta(\fS) \rightarrow \rdS
\end{align}
such that, for any stated tangle $\left(\cup_{1\leq i\leq n}\alpha_i\right)\cup \left(\cup_{1\leq j\leq m}\beta_j\right)$
in $\widetilde{\fS}$ where each $\alpha_i$ is an arc and each $\beta_j$ a loop,
\begin{align}\label{eq-Phi-image}
\Phi\left(\left(\cup_{1\leq i\leq n}\alpha_i\right)\cup \left(\cup_{1\leq j\leq m}\beta_j\right) \right) = \left(\cup_{1\leq i\leq n}\alpha_i^{(\ell)}\right)\cup \left(\cup_{1\leq j\leq m}\beta_j^{(T_\ell)}\right).
\end{align}
Let $S$ be a subset of the set of interior punctures of $\fS$.
Define 
$$\rd_q^S(\fS):= \rd_q^{\bf w}(\fS),\qquad
\text{(see \eqref{eq-sliced-skein})},$$
where ${\bf w}=(w_v = q+q^{-1})_{v\in S}$.
Note that $T_l(q+q^{-1})=q^\ell + q^{-\ell}=\eta + \eta^{-1}$.
Then the Frobenius map in \eqref{eq-Fro-stated} induces the following Frobenius map
$$\Phi\colon \rd_\eta^S(\fS) \rightarrow \rd_q^S(\fS).$$

The following lemma establishes the compatibility between the two Frobenius maps $\cF$ and $\Phi$. In particular, it determines $\cF(\alpha)$ whenever $\alpha$ consists only of loops and stated edge arcs.

\begin{lemma}
Let $\fS$ be a punctured surface, and let
$C_1,\dots,C_n$ be all the boundary circles of $\fS$. Define
\[
\fS':=\fS\setminus\bigl(C_1\cup\cdots\cup C_n\bigr).
\]
For each $i$, let $u_i$ denote the interior puncture of $\fS'$ obtained
by removing the boundary circle $C_i$, and set
\[
S:=\{u_1,\dots,u_n\}.
\]
Under the assumption of Corollary~\ref{Cor-Fro-first},
then the following diagram commutes:
\begin{equation}\label{diag-Fro-sliced}
	\begin{tikzcd}
		\rd_\eta^S(\fS')
		\arrow[r, "\Phi"]
		\arrow[d, "P_S"']
		&
		\rd_q^S(\fS')
		\arrow[d, "P_S"]
		\\
		\rd_\eta(\fS)
		\arrow[r, "\cF"']
		&
		\rd_q(\fS),
	\end{tikzcd}
\end{equation}
where $P_S$ is induced by the embedding $\fS'\hookrightarrow \fS$.
\end{lemma}

\begin{proof}
Let $U$ denote the set of interior punctures of $\fS$, let
$\mathcal Q$ denote the set of all admissible ideal arcs of $\fS'$, and let
$\partial_E\fS'$ denote the set of boundary edges of $\fS'$.

By Corollary~\ref{thm-sliced-cluster}, the $\wR$-algebra
$\rd_\eta^S(\fS')$ is generated by
\[
Y
:=
\{\alpha_p\mid p\in U\}
\cup
\{A_e\mid e\in\mathcal Q\}
\cup
\{A_e^{-1}\mid e\in\partial_E\fS'\}.
\]
Under the natural identification, every admissible ideal arc of $\fS'$ determines
an admissible ideal arc of $\fS$. Consequently,
\[
P_S\bigl(\Phi(y)\bigr)
=
\cF\bigl(P_S(y)\bigr)
\qquad\text{for every }y\in Y.
\]
Since both compositions in \eqref{diag-Fro-sliced} are $\wR$-algebra
homomorphisms and agree on the generating set $Y$, they agree on
$\rd_\eta^S(\fS')$. Hence \eqref{diag-Fro-sliced} commutes.
\end{proof}

Next, we compute $\cF(\alpha)$ when $\alpha$ is a stated edge-circle arc (Definition~\ref{def-arcs}). We first establish Lemmas~\ref{lem-arc-sum-Tl}-\ref{lem-arc-resolution}, which will be used in the computation.

\begin{lemma}[\cite{BL22}]\label{lem-arc-sum-Tl}
Under Assumption~\Rlabel{A},
in $\rdS$, the following relations hold:
\begin{align*}
	\begin{array}{c}
		\includegraphics[scale=0.4]{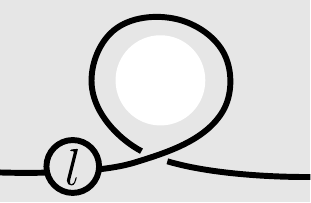}
	\end{array}
	&=
	\eta^{-1}
	\begin{array}{c}
		\includegraphics[scale=0.4]{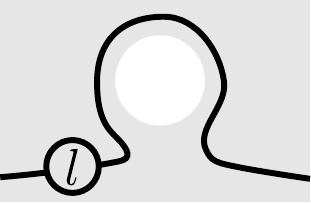}
	\end{array}
	+
	\eta
	\begin{array}{c}
		\includegraphics[scale=0.4]{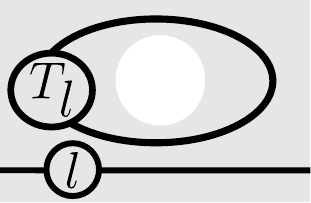}
	\end{array},\\
	\begin{array}{c}
		\includegraphics[scale=0.4]{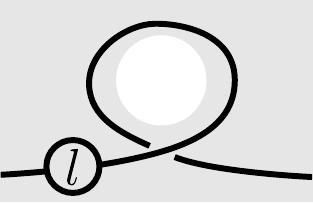}
	\end{array}
	&=
	\eta
	\begin{array}{c}
		\includegraphics[scale=0.4]{arc-12.pdf}
	\end{array}
	+
	\eta^{-1}
	\begin{array}{c}
		\includegraphics[scale=0.4]{arc-13.pdf}
	\end{array}.
\end{align*}
In each relation, the diagram on the left-hand side represents a portion
of a stated edge arc.
\end{lemma}

Lemmas \ref{lem-Fro-arc-circle} and \ref{lem-Fro-Cut} immediately yield the following result.

\begin{lemma}\label{lem-cut-N}
Under the assumption of Corollary~\ref{Cor-Fro-first},
we have
\[
\overline{\Theta}_c\left(
\begin{array}{c}
	\includegraphics[scale=0.4]{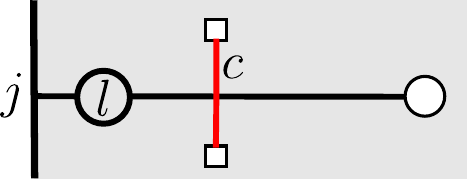}
\end{array}
\right)
=
\sum_{i\in\{-,+\}}
\begin{array}{c}
	\includegraphics[scale=0.4]{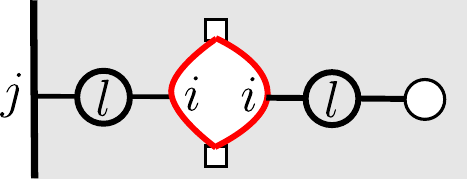}
\end{array}.
\]
\end{lemma}

\begin{lemma}\label{lem-arc-resolution}
Under the assumption of Corollary~\ref{Cor-Fro-first},
in $\rdS$, the following relations hold:
\begin{align}
	\label{eq-self-crossing1}
	\begin{array}{c}
		\includegraphics[scale=0.4]{arc-21.pdf}
	\end{array}
	&=
	\eta^{-1}
	\begin{array}{c}
		\includegraphics[scale=0.4]{arc-12.pdf}
	\end{array}
	+
	\eta
	\begin{array}{c}
		\includegraphics[scale=0.4]{arc-13.pdf}
	\end{array},\\
	\label{eq-self-crossing2}
	\begin{array}{c}
		\includegraphics[scale=0.4]{arc-11.pdf}
	\end{array}
	&=
	\eta
	\begin{array}{c}
		\includegraphics[scale=0.4]{arc-12.pdf}
	\end{array}
	+
	\eta^{-1}
	\begin{array}{c}
		\includegraphics[scale=0.4]{arc-13.pdf}
	\end{array}.
\end{align}
In each relation, the diagram on the left-hand side represents a portion
of a stated edge-circle arc (Definition~\ref{def-arcs}).
\end{lemma}

\begin{proof}
We prove only \eqref{eq-self-crossing1}, since the same argument applies
to \eqref{eq-self-crossing2}. Without loss of generality, assume that the
left endpoint of the diagram on the left-hand side lies on a boundary edge.

Add two punctures to $\fS$ as shown in
Figure~\ref{fig:adding-puncture}, and denote the resulting punctured surface
by $\fS'$. Let $c$ be the ideal arc in $\fS'$ connecting these two newly
added punctures, as depicted in Figure~\ref{fig:adding-puncture}.

Let $\alpha$ denote the stated tangle appearing on the left-hand side of
\eqref{eq-self-crossing1}. Let $C_1$ and $C_2$ be the boundary components
of $\overline{\fS'}$ such that $\alpha$ connects
$C_1\times(-1,1)$ to $C_2\times(-1,1)$. Let $U$ be an open regular
neighborhood of
\[
\alpha\cup\bigl(C_1\times(-1,1)\bigr)
\cup\bigl(C_2\times(-1,1)\bigr)
\cup\bigl(c\times(-1,1)\bigr)
\]
in $\widetilde{\fS'}=\fS'\times(-1,1)$. Then $U$ may be identified with
the thickening $\widetilde{\fS}_1$ of a punctured surface $\fS_1$.
Let
\[
f\colon\widetilde{\fS}_1\longrightarrow\widetilde{\fS'}
\]
be the resulting embedding.

There is an ideal arc $c_1$ in $\fS_1$ such that
\[
f\bigl(c_1\times\{t\}\bigr)=c\times\{t\},
\qquad t\in(-1,1).
\]
Moreover, $f$ induces an embedding
\[
g\colon
\widetilde{\Cut_{c_1}\fS_1}
\longrightarrow
\widetilde{\Cut_c\fS'}
\]
for which the following diagram commutes:
\begin{equation}\label{eq-fg-Theta}
	\begin{tikzcd}
		\rd_\eta(\fS_1)
		\arrow[r, "\overline\Theta_{c_1}"]
		\arrow[d, "f_*"']
		&
		\rd_\eta(\Cut_{c_1}\fS_1)
		\arrow[d, "g_*"]
		\\
		\rd_\eta(\fS')
		\arrow[r, "\overline\Theta_c"']
		&
		\rd_\eta(\Cut_c\fS').
	\end{tikzcd}
\end{equation}

We then have
\begin{align*}
	&\overline\Theta_c\left(
	\begin{array}{c}
		\includegraphics[scale=0.4]{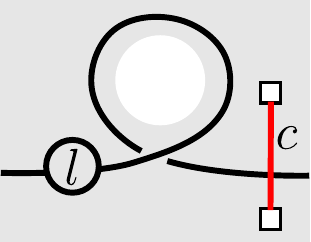}
	\end{array}
	\right)\\
	=&
	\sum_{i\in\{-,+\}}
	\begin{array}{c}
		\includegraphics[scale=0.4]{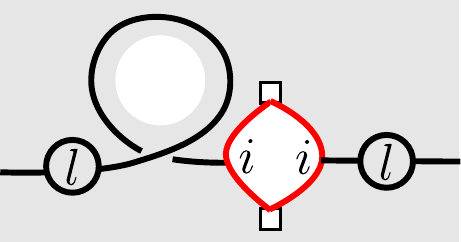}
	\end{array}
	&&\text{by Lemma~\ref{lem-cut-N} and \eqref{eq-fg-Theta}}
	\\
	=&
	\eta^{-1}
	\sum_{i\in\{-,+\}}
	\begin{array}{c}
		\includegraphics[scale=0.4]{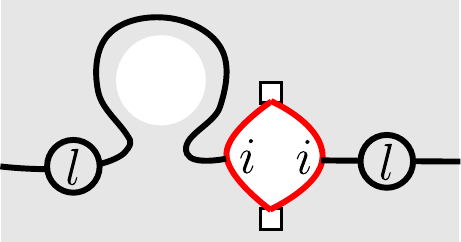}
	\end{array}
	+
	\eta
	\sum_{i\in\{-,+\}}
	\begin{array}{c}
		\includegraphics[scale=0.4]{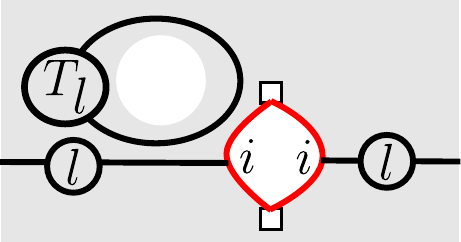}
	\end{array}
	&&\text{by Lemma~\ref{lem-arc-sum-Tl}}
	\\
	=&
	\overline\Theta_c\left(
	\eta^{-1}
	\begin{array}{c}
		\includegraphics[scale=0.4]{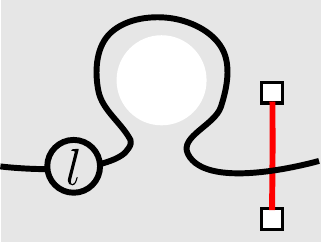}
	\end{array}
	+
	\eta
	\begin{array}{c}
		\includegraphics[scale=0.4]{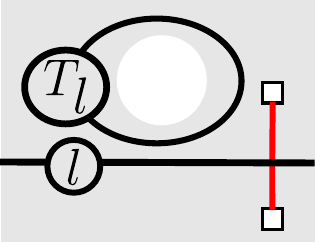}
	\end{array}
	\right)
	&&\text{by Lemma~\ref{lem-cut-N} and \eqref{eq-fg-Theta}}.
\end{align*}

Since $\overline\Theta_c$ is injective, it follows that
\begin{align}\label{lem-proof-fS}
	\begin{array}{c}
		\includegraphics[scale=0.4]{arc-31.pdf}
	\end{array}
	&=
	\eta^{-1}
	\begin{array}{c}
		\includegraphics[scale=0.4]{arc-03.pdf}
	\end{array}
	+
	\eta
	\begin{array}{c}
		\includegraphics[scale=0.4]{arc-04.pdf}
	\end{array}
	\qquad\text{in }\rd_\eta(\fS').
\end{align}

Finally, the inclusion $\fS'\hookrightarrow\fS$ induces an
$\wR$-module homomorphism
\[
\rd_\eta(\fS')\longrightarrow\rd_\eta(\fS).
\]
Applying this homomorphism to \eqref{lem-proof-fS} yields
\eqref{eq-self-crossing1}.
\end{proof}

\begin{figure}
\centering
\includegraphics[width=0.2\linewidth]{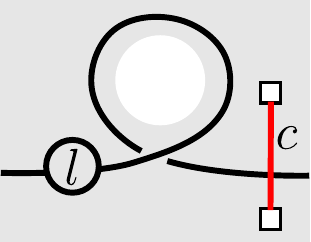}
\caption{The two newly added punctures of $\fS$. The red arc represents the ideal arc $c$.}
\label{fig:adding-puncture}
\end{figure}

The following lemma shows that the image of a stated essential arc (Definition~\ref{def-arcs}) under the Frobenius homomorphism is obtained by threading it by $x^\ell$.

\begin{lemma}\label{lem-image-arc}
Let $\alpha$ be a stated essential arc in $\fS$. Under the assumption of Corollary~\ref{Cor-Fro-first}, we have
\begin{align}\label{eq-image-arc}
	\cF(\alpha)
	= \alpha^{(\ell)}.
\end{align}
\end{lemma}

\begin{proof}
If $\alpha$ has no endpoint on a boundary circle, the result follows
from \eqref{eq-Fro-stated} and the commutative diagram
\eqref{diag-Fro-sliced}.

If $\alpha$ connects a boundary edge to a boundary circle $v$, the
result follows from Lemmas~\ref{lem-Fro-arc-circle}
and~\ref{lem-arc-resolution}.
\end{proof}

The following lemma computes $\cF(\alpha)$ when $\alpha$ consists of stated essential arcs (Definition~\ref{def-arcs}). It is a crucial step toward determining $\cF(\alpha)$ for an arbitrary stated $\fS$-tangle $\alpha$, as it establishes the compatibility between the Frobenius homomorphism and functoriality; see \eqref{eq-f-Fro-com}. This compatibility allows us to reduce the general problem to the case of crossingless stated $\fS$-tangle diagrams.

\begin{lemma}\label{lem-image-essential}
Let
$ \alpha=\bigcup_{i=1}^{n}\alpha_i$
be a stated $\fS$-tangle, where each $\alpha_i$ is an
essential arc in the sense of Definition~\ref{def-arcs}. Under the assumption of Corollary~\ref{Cor-Fro-first}, we have
\[          \cF(\alpha)=\bigcup_{i=1}^{n}\alpha_i^{(\ell)}.
\]

\end{lemma}

\begin{proof}
By Remark~\ref{rem-com-height}, there exist an integer $r$ such that
\[
\alpha
=
\eta^r\prod_{i=1}^{n}\alpha_i
\quad\text{in }\rd_\eta(\fS),
\qquad
\bigcup_{i=1}^{n}\alpha_i^{(\ell)}
=
\eta^r\prod_{i=1}^{n}\alpha_i^{(\ell)}
\quad\text{in }\rdS.
\]
The compatibility of the height-exchange relations established
in Remark~\ref{rem-com-height}, together with
Equation~\eqref{eq-image-arc}, gives
\begin{align*}
	\cF(\alpha)
	=
	\eta^r\prod_{i=1}^{n}\cF(\alpha_i) 
	=
	\eta^r\prod_{i=1}^{n}\alpha_i^{(\cF)} 
	=
	\bigcup_{i=1}^{n}\alpha_i^{(\cF)}.
\end{align*}

\end{proof}

By Proposition~\ref{prop-Fro-arc-T} and Lemma~\ref{lem-image-essential}, we obtain the following.

\begin{lemma}\label{imgea-circle-arc2}
Let $\alpha$ be either a loop or a circle arc in
$\widetilde{\fS}$; see Definition~\ref{def-arcs}. Under the assumption of Corollary~\ref{Cor-Fro-first}, we have
\[
\cF(\alpha)=\alpha^{(T_\ell)}.
\]
\end{lemma}

\begin{proof}
If $\alpha$ is a loop, the assertion follows from
Equations~\eqref{eq-Phi-image} and \eqref{diag-Fro-sliced}.

Suppose now that $\alpha$ is a circle arc, and write
$C(\alpha)=\{v_1,v_2\},$
as in Definition~\ref{def-arcs}. There exist $t_1,t_2\in(-1,1)$
such that the endpoints of $\alpha$ lie on
$v_1\times\{t_1\}$ and $v_2\times\{t_2\}$, respectively.

Since $\fS$ is essentially punctured, it has at least one boundary
edge $e$. Let $C$ be the boundary component of $\overline{\fS}$
containing $e$, and let
$ e_1=e,e_2,\ldots,e_n$
be all the boundary edges contained in $C$. Set
$ E:=\bigcup_{i=1}^{n}e_i.$
Choose an embedded path
$\beta\colon[0,1]\longrightarrow\widetilde{\fS}$
such that
\[
\beta(0)\in E\times(-1,1),
\qquad
\beta(1)\in\alpha,
\]
and
\[
\beta((0,1))
\cap
\left(
\bigl(\partial\fS\times(-1,1)\bigr)
\cup\alpha
\right)
=
\emptyset.
\]
Let $U$ be a regular open neighborhood of
\[
\alpha
\cup\beta([0,1])
\cup\bigl(E\times(-1,1)\bigr)
\cup\bigl(v_1\times\{t_1\}\bigr)
\cup\bigl(v_2\times\{t_2\}\bigr).
\]
Then $U$ is the thickening of a connected punctured surface
$\fS_1$. By construction, $\fS_1$ contains two boundary circles  on which the endpoints of $\alpha$ lie, as well as at least
one boundary edge. Hence $\fS_1$ is essentially punctured.

Let $ f\colon\widetilde{\fS}_1\hookrightarrow\widetilde{\fS}$
denote the inclusion. The algebra $\rd_\eta(\fS_1)$ is spanned by
stated $\partial\fS_1$-tangles whose components are essential arcs.
Lemma~\ref{lem-image-essential} therefore implies that the following
diagram commutes:
\begin{equation}
	\label{eq-f-Fro-com}
	\begin{tikzcd}
		\rd_\eta(\fS_1)
		\arrow[r, "\cF"]
		\arrow[d, "f_*"']
		&
		\rd_q(\fS_1)
		\arrow[d, "f_*"]
		\\
		\rd_\eta(\fS)
		\arrow[r, "\cF"']
		&
		\rd_q(\fS).
	\end{tikzcd}
\end{equation}
Equations~\eqref{eq-Fro-arc-T} and \eqref{eq-f-Fro-com} yield
\[
\cF(\alpha)=\alpha^{(T_\ell)},
\]
as required.
\end{proof}

Using the same technique as in Lemma~\ref{lem-image-essential}, and using Lemma~\ref{imgea-circle-arc2}, we can prove the following.

\begin{proposition}\label{prop-image-all-diagrams}
Let $\beta=
\left(\bigcup_{i=1}^{m}\beta_i\right)
\cup
\left(\bigcup_{j=1}^{k}\gamma_j\right)$
be a stated $\fS$-tangle, where each $\beta_i$
is an essential arc and each $\gamma_j$ is either a loop or a
circle arc. Under the assumption of Corollary~\ref{Cor-Fro-first}, we have
\begin{align}\label{eq-image-all-diagrams}
	\cF(\beta)
	=
	\left(\bigcup_{i=1}^{m}\beta_i^{(\ell)}\right)
	\cup
	\left(\bigcup_{j=1}^{k}\gamma_j^{(T_\ell)}\right).
\end{align}
\end{proposition}

The following theorem establishes the Frobenius map for the reduced
L{\^e}--Roger--Yang skein algebra and is the main result of this section.
We also prove that the Frobenius map is compatible with the cutting homomorphism
and the quantum trace map.
In \S\ref{sec-center}, we will use this map to describe the center of the
Muller--Roger--Yang skein algebra.

\begin{theorem}\label{thm-Fro-main}
Let $\fS$ be a punctured surface, and suppose that $q^2$ is a root
of unity of odd order $\ell$. Set
$\eta^{\frac12}=q^{\frac{\ell^2}{2}}.$
Then there exists a unique $\wR$-algebra embedding
\[
\cF\colon \rd_\eta(\fS)\longrightarrow\rdS
\]
such that it satisfies \eqref{eq-image-all-diagrams}.

This algebra embedding $\cF$ satisfies the following properties
\begin{enumerate}
	\item [\rm (a)] Let e be an ideal arc, the diagram \eqref{diag-Fro-cut} commutes.
	
	\item [\rm (b)] Suppose $\fS$ is triangulable with a triangulation $\lambda$. Then 
	the diagram \eqref{diag-Fro-trace} commutes.
\end{enumerate}
\end{theorem}

\begin{proof}

The uniqueness is straightforward since 
$\rd_\eta(\fS)$ is spanned by 
stated $\fS$-tangles.

We may assume that $\fS$ is connected.

\noindent\textbf{Existence.} 
Let $\fS_1$ be the punctured surface obtained from $\fS$ by removing three interior points $p_1,p_2$, and $p_3$. As in Remark~\ref{rem-blow-map}, let $\fS_1'$ be the punctured surface obtained from $\fS_1$ by blowing up one of the punctures, say $p=p_1$, into a boundary edge $c_p$. Then $\fS_1'$ is an quasi-triangulable. Lemmas~\ref{lem-Fro-trace} and \ref{lem-Fro-Cut}, together with Proposition~\ref{prop-image-all-diagrams}, imply that there exists an algebra embedding
\begin{equation}\label{thm-proof-Fro1}
	\cF\colon \rd_\eta(\fS_1')\longrightarrow\rd_q(\fS_1')
\end{equation}
satisfying \eqref{diag-Fro-trace}, \eqref{diag-Fro-cut}, and \eqref{eq-image-all-diagrams}.

By Remark~\ref{rem-blow-map}, $\rd(\fS_1)$ is a subalgebra of $\rd(\fS_1')$. Moreover, \eqref{eq-image-all-diagrams} implies that the algebra embedding in \eqref{thm-proof-Fro1} restricts to an algebra embedding
\begin{equation}\label{thm-proof-Fro2}
	\cF\colon \rd_\eta(\fS_1)\longrightarrow\rd_q(\fS_1).
\end{equation}
Since the embedding in \eqref{thm-proof-Fro1} satisfies \eqref{diag-Fro-cut} and \eqref{eq-image-all-diagrams}, so does its restriction in \eqref{thm-proof-Fro2}.

By Remark~\ref{rem-adding-puncture}, there exist surjective algebra homomorphisms
\begin{align*}
	{\bf P}_\eta\colon \rd_\eta(\fS_1)\longrightarrow\rd_\eta(\fS),\qquad
	{\bf P}_q\colon \rd_q(\fS_1)\longrightarrow\rd_q(\fS).
\end{align*}
It follows from Figure~\ref{eq-relation-point-p}, Remark~\ref{rem-adding-puncture}, and \eqref{eq-image-all-diagrams} that
$\cF\bigl(\ker {\bf P}_\eta\bigr)\subseteq \ker {\bf P}_q.$
Consequently, the algebra embedding in \eqref{thm-proof-Fro2} descends to an algebra homomorphism
\begin{equation}\label{thm-proof-Fro3}
	\cF\colon \rd_\eta(\fS)\longrightarrow\rd_q(\fS)
\end{equation}
satisfying \eqref{eq-image-all-diagrams}.

Let $e$ be an ideal arc of $\fS$. We may isotope $e$ so that it is disjoint from $p_1,p_2,p_3$ and is not a boundary edge of $\fS_1$. It is then straightforward to verify that the cutting homomorphism $\Theta_e$ commutes with the projection ${\bf P}$. Hence the algebra homomorphism in \eqref{thm-proof-Fro3} satisfies \eqref{diag-Fro-cut}, since the embedding in \eqref{thm-proof-Fro2} does.

Finally, since the quantum trace map ${\rm tr}_\lambda$ is compatible with the cutting homomorphisms, the commutativity of \eqref{diag-Fro-trace} follows from that of \eqref{diag-Fro-cut}.

\noindent\textbf{Injectivity.}
Suppose first that $\fS$ is quasi-triangulable. In this case,
injectivity follows from Corollary~\ref{Cor-Fro-first}.

Suppose that $\partial\fS\neq\emptyset$ and that $\fS$ is not
quasi-triangulable. Then $\fS$ is either the bigon
$
\begin{tikzpicture}[baseline=(current bounding box.center)]
	\draw[fill=gray!25,line width=0.8pt]
	(0,0) circle[radius=0.25];
	\draw[fill=white,line width=0.8pt]
	(-0.05,-0.30) rectangle (0.05,-0.20);
	\draw[fill=white,line width=0.8pt]
	(-0.05,0.20) rectangle (0.05,0.30);
\end{tikzpicture}
$
or the one--circle--monogon
$\begin{tikzpicture}[baseline=(current bounding box.center)]
	\draw[fill=gray!25,line width=0.8pt]
	(0,0) circle[radius=0.25];
	\draw[fill=white,line width=0.8pt]
	(-0.05,-0.30) rectangle (0.05,-0.20);
	\draw[fill=white,line width=0.8pt]
	(0,0.05) circle[radius=0.10];
\end{tikzpicture}.$
For the bigon, the injectivity of $\cF$ is well known; see
\cite{BL22}. For the one--circle--monogon, we have
$\rd_\eta(\fS)
\simeq R[x^{\pm1}]
\simeq \rd_q(\fS)$
by \cite[Theorem~5.6]{bloomquist2025degenerations}, and the Frobenius
homomorphism is given by
$\cF(x)=x^\ell.$
It follows immediately that $\cF$ is injective.

Finally, suppose that $\partial\fS=\emptyset$. A basis $\mathsf{B}$ of $\rd(\fS)$ is
defined in \cite[Equation~(4.2)]{karuo2024positivity}. Each element of
$\mathsf{B}$ is of the form
$\prod_{i=1}^{n}\alpha_i^{(T_{\ell_i})},$
where $\alpha_1,\ldots,\alpha_n$ are pairwise non-isotopic and mutually
disjoint loops or circle arcs, and each $\ell_i$ is a nonnegative
integer. Moreover,
$\cF\left(
\prod_{i=1}^{n}\alpha_i^{(T_{\ell_i})}
\right)
=
\prod_{i=1}^{n}\alpha_i^{(T_{\ell\ell_i})}.$
Consequently, $\cF$ maps $\mathsf{B}$ injectively into itself and is
therefore injective.
\end{proof}

\subsection{Frobenius maps for L{\^e}--Roger--Yang skein algebras}\label{sec-Fro-stated}
Let $\fS$ be a punctured surface.
Attach an ideal triangle $t_e$  to each boundary edge $e$ of $\fS$ as in Figure~\ref{fig:attaching-triangle}, and let $\fS^\ast$ denote the resulting punctured surface. 
There exists a smooth embedding
$\iota\colon \fS\hookrightarrow \fS^\ast$
that maps each boundary edge $e$ to $\hat{e}$ and restricts to the identity outside a small neighborhood of $e$; see Figure~\ref{fig:attaching-triangle}.
The embedding $\iota$ induces an algebra embedding
$\iota_*\colon \SS \rightarrow \rd_q(\fS^\ast)$ \cite{bloomquist2025degenerations}.

\begin{figure}[ht]
\centering
\begin{tikzpicture}[
	x=1cm,
	y=1cm,
	line cap=round,
	line join=round,
	surface/.style={
		fill=gray!20,
		draw=none
	},
	boundary/.style={
		draw=black,
		line width=0.85pt
	},
	tangle/.style={
		draw=black,
		line width=1.15pt
	},
	embedding arrow/.style={
		-{Latex[length=2.2mm,width=1.5mm]},
		line width=0.8pt
	},
	puncture/.style={
		rectangle,
		draw=black,
		fill=white,
		line width=0.75pt,
		minimum size=3.1mm,
		inner sep=0pt
	},
	every node/.style={
		font=\small
	}
	]
	
	\begin{scope}[shift={(0,0)}]
		\fill[surface]
		(-1.65,0) rectangle (1.65,1.20);
		
		\fill[gray!20]
		(-1.25,0) -- (1.25,0) -- (0,-1.10) -- cycle;
		
		\draw[boundary]
		(-1.65,0) -- (-1.25,0);
		\draw[boundary]
		(1.25,0) -- (1.65,0);
		\draw[boundary]
		(-1.25,0) -- (0,-1.10) -- (1.25,0);
		
		\draw[boundary]
		(-1.25,0) -- (1.25,0);
		
		\node[puncture] at (-1.25,0) {};
		\node[puncture] at ( 1.25,0) {};
		\node[puncture] at (0,-1.10) {};
		
		\node at (0,0.2) {$e$};
		\node at (0,-0.52) {$t_e$};
		\node at (0.9,-0.52) {$\hat e$};
		\node at (1.37,0.94) {$\fS$};
	\end{scope}
	
	\begin{scope}[shift={(4.40,0)}]
		\fill[surface]
		(-1.65,0) rectangle (1.65,1.20);
		
		\draw[boundary]
		(-1.65,0) -- (1.65,0);
		
		\node[puncture] at (-1.25,0) {};
		\node[puncture] at ( 1.25,0) {};
		
		\foreach \x in {-0.48,0,0.48}{
			\draw[tangle]
			(\x,0) -- (\x,0.88);
		}
		
		\draw[
		decorate,
		decoration={
			brace,
			amplitude=3.5pt
		},
		line width=0.65pt
		]
		(-0.64,1.00) -- (0.64,1.00)
		node[midway,yshift=10pt] {$\alpha$};
		
		\node at (1.37,0.94) {$\fS$};
	\end{scope}
	
	\draw[embedding arrow]
	(6.30,0.50) -- (7.05,0.50)
	node[midway,above=2pt] {$\iota$};
	
	\begin{scope}[shift={(8.90,0)}]
		\fill[surface]
		(-1.65,0) rectangle (1.65,1.20);
		
		\fill[gray!20]
		(-1.25,0) -- (1.25,0) -- (0,-1.10) -- cycle;
		
		\draw[boundary]
		(-1.65,0) -- (-1.25,0);
		\draw[boundary]
		(1.25,0) -- (1.65,0);
		\draw[boundary]
		(-1.25,0) -- (0,-1.10) -- (1.25,0);
		
		\draw[boundary]
		(-1.25,0) -- (1.25,0);
		
		\node[puncture] at (-1.25,0) {};
		\node[puncture] at ( 1.25,0) {};
		\node[puncture] at (0,-1.10) {};
		
		\draw[tangle]
		(-0.48,0.88)
		.. controls (-0.48,0.34) and (-0.48,-0.30) ..
		(0.24,-0.889);
		
		\draw[tangle]
		(0,0.88)
		.. controls (0,0.34) and (-0.04,-0.20) ..
		(0.50,-0.660);
		
		\draw[tangle]
		(0.48,0.88)
		.. controls (0.48,0.34) and (0.38,-0.12) ..
		(0.76,-0.431);
		
		\draw[
		decorate,
		decoration={
			brace,
			amplitude=3.5pt
		},
		line width=0.65pt
		]
		(-0.64,1.00) -- (0.64,1.00)
		node[midway,yshift=10pt] {$\iota(\alpha)$};
		
		\node at (1.37,0.94) {$\fS^*$};
	\end{scope}
	
\end{tikzpicture}

\caption{Left: attaching the ideal triangle $t_e$.
	Right: the embedding
	$\iota\colon\fS\hookrightarrow\fS^*$.}
\label{fig:attaching-triangle}
\end{figure}
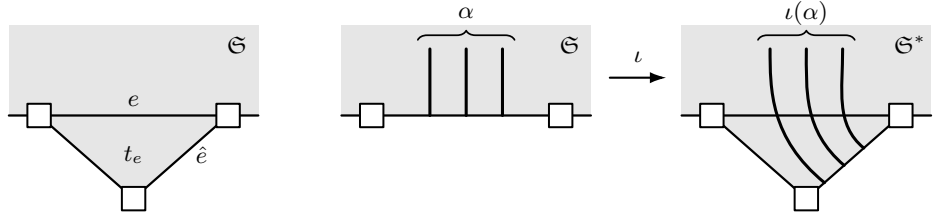

Under Assumption~\Rlabel{A}, there exists an algebra embedding
$\cF\colon \rd_\eta (\fS^\ast) \rightarrow \rd_q(\fS^\ast)$.
Equation~\eqref{eq-image-all-diagrams} implies that
there exists a unique $\wR$-algebra embedding 
\begin{align}\label{eq-Fro-stated}
\cF\colon \cS_\eta(\fS) \rightarrow \SS
\end{align}
such that the following diagram commutes
\begin{equation}
\label{diag-Fro-stated}
\begin{tikzcd}
	\cS_\eta(\fS)
	\arrow[r, "\cF"]
	\arrow[d, "\iota_\ast"']
	&
	\cS_q(\fS)
	\arrow[d, "\iota_\ast"]
	\\
	\rd_\eta(\fS^\ast)
	\arrow[r, "\cF"']
	&
	\rd_q(\fS^\ast)
\end{tikzcd}.
\end{equation} 
It is straightforward that the Frobenius map in \eqref{eq-Fro-stated} also satisfies \eqref{eq-image-all-diagrams}.

\begin{remark}\label{rem-trace}
Suppose that $\fS$ is triangulable and let $\lambda$ be an ideal
triangulation of $\fS$. Then $\lambda$ induces an ideal triangulation
$\lambda^\ast$ of $\fS^\ast$. The quantum trace map for the
L{\^e}--Roger--Yang skein algebra is defined by the composition
\cite[Equation~(45)]{bloomquist2025degenerations}
\[
\cS_q(\fS)
\xrightarrow{\ \iota_\ast\ }
\rd_q(\fS^\ast)
\xrightarrow{\ \tr_{\lambda^\ast}\ }
\overline{\mathcal{X}}_q(\fS^\ast,\lambda^\ast).
\]
Theorem~\ref{thm-Fro-main}(b) implies that this quantum trace map
commutes with the Frobenius map in \eqref{eq-Fro-stated}.
\end{remark}

The following will be used to show that the Frobenius map in \eqref{eq-Fro-stated} commutes with the cutting homomorphism. 

\begin{lemma}\label{lem-quantum-binomial-uplus-uminus}
Suppose that $q^2$ is a root of unity of odd order $\ell$. Let
$u_+,u_-,a,b$ satisfy
\[
q u_+u_- - q^{-1}u_-u_+ = q-q^{-1},
\qquad
ab=q^{-2}ba.
\]
Then
\[
\left(u_+\otimes a+u_-\otimes b\right)^\ell
=
u_+^\ell\otimes a^\ell
+
u_-^\ell\otimes b^\ell.
\]
\end{lemma}

\begin{proof}
Set $r=q^{-2}$. Then $r$ is a primitive $\ell$-th root of
unity, and the defining relations can be rewritten as
\[
u_+u_-=r\,u_-u_+ +(1-r),
\qquad
ab=r\,ba.
\]

Work temporarily in the localization in which $u_-$ is
invertible, and set
\[
K:=1-u_-u_+.
\]
The relation between $u_+$ and $u_-$ implies
\[
u_+K=rKu_+,
\qquad
Ku_-=r\,u_-K.
\]
Moreover,
\[
u_+=u_-^{-1}(1-K).
\]

Define
\[
A:=u_-^{-1}\otimes a,\qquad
B:=-u_-^{-1}K\otimes a,\qquad
C:=u_-\otimes b.
\]
Then
\[
u_+\otimes a+u_-\otimes b=A+B+C.
\]
A direct calculation gives
\[
AB=r\,BA,\qquad
AC=r\,CA,\qquad
BC=r^2CB.
\]
In particular,
\[
A(B+C)=r(B+C)A.
\]
Since $\ell$ is odd, $r^2$ is also a primitive $\ell$-th root
of unity. Applying Lemma~\ref{lem:expan1} twice yields
\[
(A+B+C)^\ell
=A^\ell+(B+C)^\ell
=A^\ell+B^\ell+C^\ell.
\]

Since $\ell$ is odd and $r^\ell=1$, we have
\[
A^\ell+B^\ell
=
u_-^{-\ell}(1-K^\ell)\otimes a^\ell.
\]
On the other hand, the expression
\[
u_+=u_-^{-1}-u_-^{-1}K
\]
is a sum of two $r$-commuting elements. Thus another application
of Lemma~\ref{lem:expan1} gives
\[
u_+^\ell=u_-^{-\ell}(1-K^\ell).
\]
It follows that
\[
A^\ell+B^\ell=u_+^\ell\otimes a^\ell,
\qquad
C^\ell=u_-^\ell\otimes b^\ell.
\]
Therefore,
\[
\left(u_+\otimes a+u_-\otimes b\right)^\ell
=
u_+^\ell\otimes a^\ell
+
u_-^\ell\otimes b^\ell.
\]
Although the calculation was carried out after adjoining
$u_-^{-1}$, the resulting identity contains no negative powers
of $u_-$ and hence holds in the original setting.
\end{proof}

\begin{figure}
\centering
\includegraphics[width=0.25\linewidth]{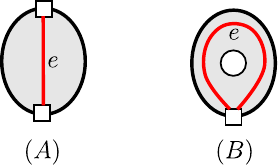}
\caption{The pair $(\fS,e)$.}
\label{fig:cuttinge}
\end{figure}

\begin{remark}\label{rem-cutting}
We show that the Frobenius map in~\eqref{eq-Fro-stated} also
commutes with the cutting homomorphism. If $\fS$ has no boundary
edges, this follows from Theorem~\ref{thm-Fro-main}.

Suppose now that $\fS$ has at least one boundary edge.
Relation~\eqref{hight} and Remark~\ref{rem-com-height1} imply that
the $\widetilde{\mathbb C}$-algebra $\cS_\eta(\fS)$ is generated
by stated essential arcs; see Definition~\ref{def-arcs}. Let $e$
be an ideal arc of $\fS$. To prove
$\cF\circ\Theta_e=\Theta_e\circ\cF,$
it therefore suffices to verify that
$\cF\bigl(\Theta_e(\alpha)\bigr)
=
\Theta_e\bigl(\cF(\alpha)\bigr)$
for every stated essential arc $\alpha$.

By~\cite[Lemma 3.6]{BL22}, we may assume that $\alpha$ intersects
$e\times(-1,1)$ exactly once. After restricting to a regular open
neighborhood of $\alpha\cup\bigl(e\times(-1,1)\bigr)$, we may
further assume that the pair $(\fS,e)$ is one of the two cases
depicted in Figure~\ref{fig:cuttinge}. The case in
Figure~\ref{fig:cuttinge}(A) follows from~\cite[Theorem 1.2]{BL22}, while
the case in Figure~\ref{fig:cuttinge}(B) follows from
Lemma~\ref{lem-quantum-binomial-uplus-uminus}.
\end{remark}

\section{Centers and ranks of the Muller--Roger--Yang and reduced L{\^e}--Roger--Yang skein algebras}\label{sec-center}

In this section, we compute the centers of the Muller--Roger--Yang
skein algebra and the reduced L{\^e}--Roger--Yang skein algebra
(Theorem~\ref{thm-center}, Remark~\ref{rem-center-reduced}), and use
these descriptions to determine their ranks as modules over the centers
(Theorem~\ref{prop-rank}, Remark~\ref{rem-rankkk}). These ranks
coincide with the squares of the maximal dimensions of the irreducible
representations. By Lemma~\ref{lem-positive-reduced}
(\cite[Proposition~3.1]{karuo2025cluster}), it suffices to carry out
these computations for the Muller--Roger--Yang skein algebra.

\subsection{Centers of the Muller--Roger--Yang and reduced L{\^e}--Roger--Yang skein algebras}

To describe the center $\mathcal Z(\cS_q^+(\fS))$ of the Muller--Roger--Yang skein algebra $\cS_q^+(\fS)$ when $q$ is a root of unity of even order, we introduce a subalgebra of $\cS_q^+(\fS)$.

Let $\alpha$ be a $\fS$-tangle diagram, and let $e$ be a boundary component of $\fS$. Denote by $e(\alpha)$ the number of endpoints of $\alpha$ lying on $e$. Define
\begin{align}\label{eqqq-Beven}
B^{\rm ev}(\fS)
:=
\left\{
\alpha\in B^+(\fS)
\,\middle|\,
e(\alpha)\equiv c(\alpha)\pmod 2
\text{ for all boundary components }c,e\text{ of }\fS
\right\}.
\end{align}

The following lemma gives an equivalent characterization of
$B^{\rm ev}(\fS)$ in terms of coordinates (see \eqref{eq-coordinate}). This characterization will be
used to compute the center of $\cS_q^+(\fS)$.

\begin{lemma}\label{lem-even}
Let $\alpha\in B^+(\fS)$. Then
\[
\alpha\in B^{\rm ev}(\fS)
\quad\Longleftrightarrow\quad
{\bf n}_\alpha\overline Q_\lambda{\bf n}_\beta^T
\equiv 0\pmod 2
\quad\text{for every }\beta\in B^+(\fS).
\]
\end{lemma}

\begin{proof}
Set $\eta=-1$. Remark~\ref{rem-com-height1} implies that
\[
\alpha\in B^{\rm ev}(\fS)
\quad\Longleftrightarrow\quad
\alpha\in\mathcal Z\bigl(\cS^+_\eta(\fS)\bigr).
\]
It therefore suffices to prove that
\begin{align}\label{eq-center-parity}
	\alpha\in\mathcal Z\bigl(\cS^+_\eta(\fS)\bigr)
	\quad\Longleftrightarrow\quad
	{\bf n}_\alpha\overline Q_\lambda{\bf n}_\beta^T
	\equiv 0\pmod 2
	\quad\text{for every }\beta\in B^+(\fS).
\end{align}

The forward implication in \eqref{eq-center-parity} follows immediately from
Lemma~\ref{lem-leading-deg}.

Conversely, suppose that
\[
{\bf n}_\alpha\overline Q_\lambda{\bf n}_\beta^T
\equiv 0\pmod 2
\quad\text{for every }\beta\in B^+(\fS).
\]
Since $B^+(\fS)$ is a basis of $\cS^+_\eta(\fS)$, it suffices to show that
\[
\alpha\beta=\beta\alpha
\qquad\text{for every }\beta\in B^+(\fS).
\]
By Lemma~\ref{lem-leading-term}, we may write
\[
\tr_\lambda(\alpha)
=
x^{{\bf n}_\alpha}
+
\sum_{{\bf t}\in S_1}r_{\bf t}x^{\bf t},
\qquad
\tr_\lambda(\beta)
=
x^{{\bf n}_\beta}
+
\sum_{{\bf k}\in S_2}c_{\bf k}x^{\bf k},
\]
where $S_1,S_2\subset\mathbb Z^\lambda$ are finite subsets such that
${\bf n}_\alpha\notin S_1$ and ${\bf n}_\beta\notin S_2$, and
$r_{\bf t},c_{\bf k}\in R\setminus\{0\}$. Moreover,
\[
{\bf n}_\alpha-{\bf t}\in(2\mathbb N)^\lambda
\quad\text{for every }{\bf t}\in S_1,
\qquad
{\bf n}_\beta-{\bf k}\in(2\mathbb N)^\lambda
\quad\text{for every }{\bf k}\in S_2.
\]
Since $\eta=-1$, the defining commutation relations of the quantum torus (\eqref{eq-def-X-torus})
therefore imply that every monomial occurring in $\tr_\lambda(\alpha)$
commutes with every monomial occurring in $\tr_\lambda(\beta)$. Hence
\[
\tr_\lambda(\alpha\beta)
=
\tr_\lambda(\alpha)\tr_\lambda(\beta)
=
\tr_\lambda(\beta)\tr_\lambda(\alpha)
=
\tr_\lambda(\beta\alpha).
\]
The injectivity of $\tr_\lambda$ now gives
$\alpha\beta=\beta\alpha$. Therefore
$\alpha\in\mathcal Z\bigl(\cS^+_\eta(\fS)\bigr)$, completing the proof.
\end{proof}

We make the following definition.
Recall that $\wR$ is defined in \eqref{eq-Rtilde}.
Define
\begin{align}
\cS^{\rm ev}_q(\fS)
:=
{\rm Span}_{\wR}
\left\{
\alpha \,\middle|\,
\begin{array}{l}
	\alpha \text{ is a positively stated }\partial\fS\text{-tangle diagram, and}\\
	e(\alpha)\equiv c(\alpha)\pmod 2
	\text{ for all boundary components }c,e\text{ of }\fS
\end{array}
\right\}.
\end{align}
Then we have the following.

\begin{lemma}\label{lem-sub-even}
The $\wR$-module $\cS^{\rm ev}_q(\fS)$ is a subalgebra of
$\cS^+_q(\fS)$, and $B^{\rm ev}(\fS)$ is an $\wR$-basis of
$\cS^{\rm ev}_q(\fS)$.
\end{lemma}

\begin{proof}

It is straightforward that $\cS^{\rm ev}_q(\fS)$ is closed under multiplication and
hence is a subalgebra of $\cS^+_q(\fS)$.

Since the parity of the number of endpoints on each boundary component
is preserved by all the defining relations, $\cS^{\rm ev}_q(\fS)$ is
spanned by those elements of $B^+(\fS)$ satisfying
$e(\alpha)\equiv c(\alpha)\pmod 2$
for all boundary components $c$ and $e$ of $\fS$. These elements are
precisely those in $B^{\rm ev}(\fS)$. Finally, because
$B^{\rm ev}(\fS)\subseteq B^+(\fS)$ and $B^+(\fS)$ is $\wR$-linearly
independent, $B^{\rm ev}(\fS)$ is an $\wR$-basis of
$\cS^{\rm ev}_q(\fS)$.
\end{proof}

Define
\begin{align}\label{eq-image-qF}
\im_q\cF
:=
\begin{cases}
	\cF\bigl(\cS^+_\eta(\fS)\bigr),
	& \text{if the order of $q$ is odd},\\[2pt]
	\cF\bigl(\cS^{\rm ev}_\eta(\fS)\bigr),
	& \text{if the order of $q$ is even}.
\end{cases}
\end{align}

Let $E$ be a boundary component of $\overline{\fS}$ containing at least one puncture, and let $e_1,\dots,e_n$ be the boundary edges contained in $E$. Define
\begin{align}\label{boundary-central}
A_E
:=
\left[
\prod_{i=1}^n A_{e_i}
\right]_{\mathrm{norm}}\qquad \text{(see \eqref{eq-D-X-q})}.
\end{align}
It was shown in \cite{korinman2021unicity,karuo2026center} that
$A_E\in \mathcal Z\bigl(\cS_q^+(\fS)\bigr).$

\begin{remark}
Recall that we use $\mathfrak{M}$ to denote the punctured surface 
obtained from a closed annulus by removing a puncture from one of its boundary component
(see Figure~\ref{fig:uplus-a}).
Let $\mathfrak{M}_1$  denote the punctured surface 
obtained from a closed annulus by removing a puncture from its interior. 
It is well-known that $\cS_q^+(\fS)$ is a commutative algebra when $\fS=\mathfrak{M}$ or $\mathfrak{M}_1$.
\end{remark}

The following theorem describes the center of the Muller--Roger--Yang skein algebra $\cS_q^+(\fS)$ when $q^2$ is a root of unity of odd
order. It is the main result of this section.
In \S\ref{sec-rank}, we will use the center formulated here to calculate the rank of $\cS_q^+(\fS)$ over its center.

\begin{theorem}\label{thm-center}
Let $\fS\neq \mathfrak{M}, \mathfrak{M}_1$ be a triangulable punctured surface, and suppose that $q^2$ is a root of unity of odd order. Then, as a $\widetilde{R}$-subalgebra of $\cS_q^+(\fS)$ (see \eqref{eq-Rtilde}),  its center $\mathcal Z\bigl(\cS_q^+(\fS)\bigr)$
is generated by $\im_q\cF$, the peripheral loops, and the boundary central elements defined in~\eqref{boundary-central}.
\end{theorem}

\begin{proof}
Let $\ell'$ denote the order of $q$, and let $\mathcal A$ be the $\cR$-subalgebra of
$\cS_q^+(\fS)$ generated by $\im_q\cF$ (see \eqref{eq-image-qF}), the peripheral loops, and the boundary central elements defined in~\eqref{boundary-central}. Equations~\eqref{eq-height-l}--\eqref{eq-height-21} imply that
$\mathcal A\subseteq \mathcal Z\bigl(\cS_q^+(\fS)\bigr).$

We prove the reverse inclusion. Let
$0\neq X\in\mathcal Z\bigl(\cS_q^+(\fS)\bigr),$
and fix an ideal triangulation $\lambda$ of $\fS$. We show that $X\in\mathcal A$ by induction on $\deg(X)$ with respect to the order introduced in Definition~\ref{def-leading-term-degree}.

If $\deg(X)={\bf 0}$, then
$X=r\emptyset\in\mathcal A$
for some $r\in R\setminus\{0\}$, where $\emptyset$ denotes the empty tangle.

Now suppose that
$\lt(X)=s\alpha'$
for some $s\in R\setminus\{0\}$ and $\alpha'\in B^+(\fS)$. Since $X$ is central, Lemmas~\ref{lem-basis-coordinate-monoid} and~\ref{lem-leading-deg} imply that
\begin{align}\label{eq-center-ell-p}
	{\bf n}_{\alpha'}\overline Q_\lambda{\bf k}^{T}=0
	\qquad\text{in }\mathbb Z_{\ell'}
\end{align}
for every ${\bf k}\in\Gamma_\lambda$. Since
$2\mathbb N^\lambda\subseteq\Gamma_\lambda$, it follows that
\begin{align}\label{eq-center-ell-ppp}
	{\bf n}_{\alpha'}\overline Q_\lambda={\bf 0}
	\qquad\text{in }(\mathbb Z_\ell)^\lambda.
\end{align}

Write $\alpha'=\alpha\beta,$
where $\beta$ is a product of peripheral loops and boundary central elements, and $\alpha$ contains neither peripheral loops nor boundary central elements. We have
${\bf n}_{\alpha'}={\bf n}_\alpha+{\bf n}_\beta.$
Since
${\bf n}_\beta\overline Q_\lambda={\bf 0},$
we obtain
\begin{align}\label{eq-center-ell-pppp}
	{\bf n}_\alpha\overline Q_\lambda
	=
	{\bf n}_{\alpha'}\overline Q_\lambda
	=
	{\bf 0}
	\qquad\text{in }(\mathbb Z_\ell)^\lambda.
\end{align}

Let $\mathrm C$ denote the set of all corners of the ideal triangles of $\lambda$. For each
${\bf m}=(m_e)_{e\in\lambda}\in\Gamma_\lambda,$
define
${\bf f}_{\bf m}
=
(f_{\mathrm c})_{\mathrm c\in\mathrm C}
\in\mathbb N^{\mathrm C}$
as follows. If $\mathrm c_1$ is a corner bounded by the edges $e_1$ and $e_3$, with $e_2$ opposite to $\mathrm c_1$, set
\[
f_{\mathrm c_1}
=
\frac{m_{e_1}+m_{e_3}-m_{e_2}}{2}.
\]
Equations~\eqref{eq-tri-1} and \eqref{eq-tri-2} ensure that
$f_{\mathrm c_1}\in\mathbb N$. Moreover, since every
$\begin{tikzpicture}[baseline=(current bounding box.center)]
	\draw[fill=gray!25,line width=0.8pt]
	(0,0) circle[radius=0.25];
	\draw[fill=white,line width=0.8pt]
	(-0.05,-0.30) rectangle (0.05,-0.20);
	\draw[fill=white,line width=0.8pt]
	(0,0.05) circle[radius=0.10];
\end{tikzpicture}$
is adjacent to an ideal triangle, the vector ${\bf f}_{\bf m}$ determines ${\bf m}$. More precisely, if $e_1$ separates the corners $\mathrm c_1$ and $\mathrm c_2$, then
\begin{align}\label{center-recover}
	m_{e_1}=f_{\mathrm c_1}+f_{\mathrm c_2}.
\end{align}

Set
\[
{\bf n}:={\bf n}_\alpha=(n_e)_{e\in\lambda}
\qquad\text{and}\qquad
{\bf f}_{\bf n}
=
(g_{\mathrm c})_{\mathrm c\in\mathrm C}
\in\mathbb N^{\mathrm C}.
\]

Let $e$ be an interior edge situated as in Figure~\ref{fig:corners}$(A)$. Then
\begin{align}\label{eq-center1}
	({\bf n}\overline Q_\lambda)(e)
	=
	n_{e_1}+n_{e_2}-n_{e_3}-n_{e_4} 
	=
	2(g_{\mathrm c_3}-g_{\mathrm c_1})
	=
	2(g_{\mathrm c_2}-g_{\mathrm c_4})
	=
	0
	\qquad\text{in }\mathbb Z_\ell.
\end{align}

Let $a$ be an interior edge situated as in Figure~\ref{fig:corners}$(B)$. Then
\begin{align}\label{eq-center2}
	({\bf n}\overline Q_\lambda)(a)
	=
	n_{a_1}-n_{a_2}
	=
	g_{\mathrm c_5}-g_{\mathrm c_6}
	=
	0
	\qquad\text{in }\mathbb Z_\ell.
\end{align}
Similarly, if $b$ is a boundary edge situated as in Figure~\ref{fig:corners}$(C)$, then
\begin{align}\label{eq-center3}
	({\bf n}\overline Q_\lambda)(b)
	=
	n_{b_1}-n_{b_2}
	=
	g_{\mathrm c_7}-g_{\mathrm c_8}
	=
	0
	\qquad\text{in }\mathbb Z_\ell.
\end{align}

For each puncture $p$, let $\mathrm C_p$ denote the set of corners surrounding $p$. Since $\alpha$ contains no peripheral loops, for every interior puncture $p$ there exists at least one corner
$\mathrm c\in\mathrm C_p$ such that $g_{\mathrm c}=0$. Equations~\eqref{eq-center1} and~\eqref{eq-center2} therefore imply that
$g_{\mathrm c}=0\text{ in }\mathbb Z_\ell$
for every $\mathrm c\in\mathrm C_p$ (recall that $\ell$ is odd).

Now let $E$ be a boundary component of $\overline{\fS}$ containing the punctures
$p_1,\ldots,p_m$, where $m\geq1$. Since $\alpha$ contains no boundary central elements, there exists some
$p_i\in\{p_1,\ldots,p_m\}$ such that
$g_{\mathrm c}=0$
for every $\mathrm c\in\mathrm C_{p_i}$. Equations~\eqref{eq-center1}--\eqref{eq-center3} then imply that
$g_{\mathrm c}=0\text{ in }\mathbb Z_\ell$
for every $\mathrm c\in\mathrm C_{p_j}$ and every $1\leq j\leq m$.

Consequently,
\begin{align}\label{eq-generic}
	{\bf f}_{\bf n}={\bf 0}
	\qquad\text{in }(\mathbb Z_\ell)^{\mathrm C}.
\end{align}
Hence
${\bf f}'
:=
\frac{1}{\ell}{\bf f}_{\bf n}
\in\mathbb N^{\mathrm C}.$
Using~\eqref{center-recover}, we obtain an element
${\bf m}'\in\Gamma_\lambda$ satisfying
${\bf f}_{{\bf m}'}={\bf f}'.$
In particular,
$\ell{\bf m}'={\bf n}_\alpha.$

Choose $\gamma\in B^+(\fS)$ such that
${\bf n}_\gamma={\bf m}'.$
By Lemma~\ref{lem-leading-deg}(a) and~\eqref{eq-image-all-diagrams},
$\deg\bigl(\cF(\gamma)\bigr)
=
\ell{\bf m}'
=
{\bf n}_\alpha.$
It follows that
\[
\deg\bigl(\cF(\gamma)\beta\bigr)
=
{\bf n}_\alpha+{\bf n}_\beta
=
{\bf n}_{\alpha'}
=
\deg(X).
\]
Therefore either
$X-s\cF(\gamma)\beta=0$
or
$\deg\bigl(X-s\cF(\gamma)\beta\bigr)<\deg(X).$
The induction hypothesis gives
$X-s\cF(\gamma)\beta\in\mathcal A.$

It remains to show that
$\cF(\gamma)\in\im_q(\cF).$
If $\ell'$ is odd, this follows immediately from the definition of
$\im_q(\cF)$. Suppose that $\ell'$ is even. Equation~\eqref{eq-center-ell-p} implies that
\[
{\bf n}_\gamma\overline Q_\lambda={\bf 0}
\qquad\text{in }(\mathbb Z_2)^\lambda.
\]
Hence, by Lemma~\ref{lem-even}, $\gamma\in\cS_\eta^{\rm ev}(\fS)$, and therefore
$\cF(\gamma)\in\im_q(\cF).$
Since $\beta\in\mathcal A$, we conclude that
$s\cF(\gamma)\beta\in\mathcal A$, and hence $X\in\mathcal A$.

\end{proof}

\begin{figure}
\centering
\includegraphics[width=0.6\linewidth]{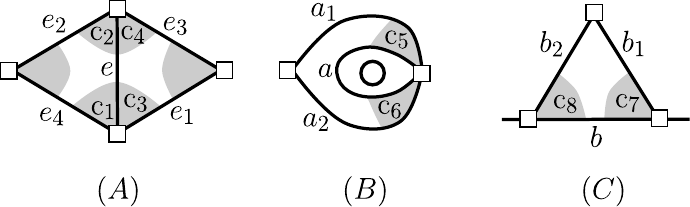}
\caption{Pictures of corners.}
\label{fig:corners}
\end{figure}

\begin{remark}\label{rem-center}
We explain how the preceding argument shows that, when $q$ is not a root of unity, the center
$\mathcal Z\bigl(\cS_q^+(\fS)\bigr)$
is generated by the peripheral loops and the boundary central elements defined in~\eqref{boundary-central}.

Indeed, in this case, the congruences in
\eqref{eq-center-ell-p}--\eqref{eq-generic}
become equalities in $\mathbb Z$ or $\mathbb Z^\lambda$.
In particular, \eqref{eq-generic} becomes
\[
{\bf f}_{\bf n}={\bf 0}
\qquad\text{in }\mathbb Z^{\rm C}.
\]
It follows that ${\bf n}_\alpha={\bf 0}$, and hence
$\alpha=\emptyset$.
Applying the induction argument from the proof of
Theorem~\ref{thm-center}, we conclude that
$\mathcal Z\bigl(\cS_q^+(\fS)\bigr)$
is generated by the peripheral loops and the boundary central elements defined in~\eqref{boundary-central}.
\end{remark}

\begin{remark}\label{rem-center-reduced}
Lemma~\ref{lem-positive-reduced} and Theorem~\ref{thm-center} imply that the center $\mathcal Z (\rdS)$ is generated by $\mathcal Z(\cS_q^+(\fS))$ and the inverses of the boundary central elements defined in~\eqref{boundary-central}.
\end{remark}

\subsection{Ranks of the Muller--Roger--Yang and reduced L{\^e}--Roger--Yang skein algebras}\label{sec-rank}
All algebras mentioned in this subsection are over $\mathbb C$. We will assume $R=\mathbb C$ throughout this subsection. 

In this subsection, we compute the rank of the Muller--Roger--Yang skein algebra over its center. We begin with the
following definition.

\def\AA{\mathcal{A}}

\begin{definition}[\cite{korinman2021unicity}]\label{def-Azu}
A $\mathbb C$-algebra $\AA$ is \textbf{affine almost Azumaya} if $\AA$ satisfies the following conditions; 
\begin{enumerate}[label={\rm (Az\arabic*)}]\itemsep0,3em
	\item\label{aaz1} $\AA$ is finitely generated as a $\mathbb C$-algebra, 
	\item\label{aaz2} $\AA$ is a domain, 
	\item\label{aaz3} $\AA$ is finitely generated as a module over its center. 
\end{enumerate}
We will use $\mathcal Z(\AA)$ to denote the center of $\AA$.
\end{definition}

The following lemma will be used to show that $\cS_q^+(\fS)$ is an affine almost Azumaya algebra when $q^2$ is a root of unity of odd order.

\begin{lemma}\label{lem-finite-center}
Suppose that $q^2$ is a root of unity of odd order. Then
$\cS_q^+(\fS)$ is finitely generated as a module over its center.
\end{lemma}

\begin{proof}
Let $\beta_1,\ldots,\beta_m\in B^+(\fS)$ be the basis elements
given by Lemma~\ref{lem-finite-order}. After relabeling, we may assume
that $\beta_1,\ldots,\beta_n$ are loops or circle arcs, whereas
$\beta_{n+1},\ldots,\beta_m$ are essential arcs.
Equation~\eqref{eq-span-finite} implies that
\begin{align}
	\label{eq-span-finite1}
	\cS_q^+(\fS)
	=
	{\rm Span}_R
	\left\{
	T_{k_1}(\beta_1)\cdots T_{k_n}(\beta_n)
	\beta_{n+1}^{k_{n+1}}\cdots\beta_m^{k_m}
	\;\middle|\;
	k_i\in\mathbb N,\ 1\leq i\leq m
	\right\}.
\end{align}
Since
$\im_q(\cF)\subset\mathcal Z\bigl(\cS_q^+(\fS)\bigr),$
the standard reduction argument shows that $\cS_q^+(\fS)$ is generated,
as a module over its center, by the finite set
\[
\left\{
T_{k_1}(\beta_1)\cdots T_{k_n}(\beta_n)
\beta_{n+1}^{k_{n+1}}\cdots\beta_m^{k_m}
\;\middle|\;
0\leq k_i\leq 2\ell,\ 1\leq i\leq m
\right\}.
\]
For details of this reduction argument, see
\cite[Theorem~5.1]{unicity}. Therefore,
$\cS_q^+(\fS)$ is finitely generated as a module over its center.
\end{proof}

It is shown in \cite{bloomquist2025degenerations} that $\cS_q^+(\fS)$ is a domain. Then Lemmas~\ref{lem-finite-order} and \ref{lem-finite-center} imply the following.

\begin{proposition}\label{prop-Azumaya}
Suppose that $q^2$ is a root of unity of odd order. Then $\cS_q^+(\fS)$ is affine almost Azumaya.
\end{proposition}

\begin{definition}\label{def-rank}
Suppose $\AA$ is an affine almost Azumaya algebra. We use $\widetilde{\mathcal Z(\AA)}$ to denote the fraction field of $\mathcal Z(\AA)$. Then $\widetilde{\AA}=\AA\otimes_{\mathcal Z(\AA)}\widetilde{\mathcal Z(\AA)}$ is a vector space over the field $\widetilde{\mathcal Z(\AA)}$.  
We call the dimension of $\widetilde{\AA}$ over $\widetilde{\mathcal Z(\AA)}$ as the {\bf rank} of $\AA$ over $\mathcal Z(\AA)$, denoted as $\rank_{\mathcal Z} \mathcal A$.
\end{definition}

\begin{remark}\label{rem-unicty}
The unicity theorem for affine almost Azumaya algebras was established
in \cite{unicity}. It states that
$\rank_{\mathcal Z}\mathcal A$ is the square of the maximal dimension
of  irreducible representations of $\mathcal A$. Moreover, the
irreducible representations of maximal dimension are classified by
the so-called Azumaya locus; see \cite{unicity}.
\end{remark}

Proposition~\ref{prop-Azumaya} shows that the unicity theorem applies to
$\cS_q^+(\fS)$ when $q^2$ is a root of unity of odd order. Therefore, by
Remark~\ref{rem-unicty}, computing
$\rank_{\mathcal Z}\cS_q^+(\fS)$ is essential for understanding the
representation theory of $\cS_q^+(\fS)$. We carry out this computation in
the remainder of this section.

\vs

We first introduce some notations. 
Let $\fS$ be a triangulable punctured surface with an ideal triangulation $\lambda$.
For any subset $S\subset\mathbb Z^\lambda$, let
$\langle S \rangle$ denote the subgroup of $\mathbb Z^\lambda$ generated by $S$.

Define
\[
S_1
:=
\left\{
{\bf n}_\alpha
\;\middle|\;
\alpha \text{ is either a peripheral loop or a boundary central element
defined in~\eqref{boundary-central}}
\right\},
\]
and set
$\Gamma^\circ:=\langle S_1 \rangle.$
Let $\mathcal M$ be the set of all monomials in the peripheral loops and
boundary central elements with nonnegative integer exponents. Then
$\mathcal M\subset B^+(\fS)$, and hence $\mathcal M$ is linearly independent.
It follows that $S_1$ is linearly independent over $\mathbb Z$. Therefore,
\begin{align}\label{rank-circ}
\rank_{\mathbb Z}\Gamma^\circ
=
a+d,
\end{align}
where $a$ is the number of interior punctures and $d$ is the number of components of
$\partial\overline{\fS}$ that contain at least one puncture.
Then proof of \cite[Lemma 7.7(a)]{kim2024unicity} and 
Remark~\ref{rem-center} implies that 
\begin{align}\label{eq-summand}
\text{$\Gamma^\circ$ is a direct summand of $\langle\Gamma_\lambda \rangle$,}
\end{align}
where $\Gamma_\lambda$ is defined in Lemma~\ref{lem-basis-coordinate-monoid}(a).

For each boundary circle $e$ of $\fS$, we also use $e$, by abuse of
notation, to denote the ideal arc of $\lambda$ surounding $e$ that cuts off a subsurface
of the form
$\begin{tikzpicture}[baseline=(current bounding box.center)]
\draw[fill=gray!25,line width=0.8pt]
(0,0) circle[radius=0.25];
\draw[fill=white,line width=0.8pt]
(-0.05,-0.30) rectangle (0.05,-0.20);
\draw[fill=white,line width=0.8pt]
(0,0.05) circle[radius=0.10];
\end{tikzpicture}.$
Thus, every boundary component of $\fS$ is identified with an ideal arc
of $\lambda$. Let $E$ denote the set of boundary components of $\fS$,
and set
\[
\Gamma_{\mathrm{ev}}:=\kappa\bigl(B^{\mathrm{ev}}\bigr),
\]
where $\kappa$ is defined in Lemma~\ref{lem-basis-coordinate-monoid}(a) and $B^{\mathrm{ev}}$ is defined in 
\eqref{eqqq-Beven}.
We have the following lemma.

\begin{lemma}\label{lem-gamma-even}
We have
\[
\left|
\frac{\langle\Gamma_\lambda\rangle}
{\langle\Gamma_{\mathrm{ev}}\rangle}
\right|
=
2^{2\left\lfloor\frac{b-1}{2}\right\rfloor}, 
\]
where $b$ is the number of connected components of
$\partial\fS$ and
\begin{align}\label{eq-def-lrfloor}
	\left\lfloor\frac{b-1}{2}\right\rfloor = \begin{cases}
		\frac{b-1}{2} & b \text{ is odd},\\
		\frac{b-2}{2} & b \text{ is even}.
	\end{cases}
\end{align}
\end{lemma}

\begin{proof}
Define a group homomorphism
\[
f\colon \langle\Gamma_\lambda\rangle\longrightarrow\mathbb Z_2^E,
\qquad
\mathbf n=(n_e)_{e\in\lambda}
\longmapsto
(n_e)_{e\in E}\pmod 2.
\]
Let $\mathbf 1=(1,1,\ldots,1)\in\mathbb Z_2^E$. We first show that,
for every $\mathbf n\in\langle\Gamma_\lambda\rangle$,
\begin{equation}\label{eq-Gamma-ev-characterization}
	\mathbf n\in\langle\Gamma_{\mathrm{ev}}\rangle
	\quad\Longleftrightarrow\quad
	f(\mathbf n)\in\{\mathbf 0,\mathbf 1\}.
\end{equation}

The forward implication follows immediately from the definition of
$B^{\mathrm{ev}}$. Conversely, suppose that
$f(\mathbf n)\in\{\mathbf 0,\mathbf 1\}$. Let
$\mathbf m\in\Gamma_{\mathrm{ev}}$ be the vector all of whose entries
are equal to $2$. For all sufficiently large $k\in\mathbb N$, the vector
$\mathbf n+k\mathbf m$
has nonnegative entries and satisfies~\eqref{eq-tri-2}. Consequently,
$\mathbf n+k\mathbf m\in\Gamma_\lambda$. Moreover,
$f(\mathbf n+k\mathbf m)=f(\mathbf n)
\in\{\mathbf 0,\mathbf 1\}.$
It therefore follows from the definition of $B^{\mathrm{ev}}$ that
$\mathbf n+k\mathbf m\in\Gamma_{\mathrm{ev}}$. Since
$k\mathbf m\in\Gamma_{\mathrm{ev}}$, we conclude that
$\mathbf n\in\langle\Gamma_{\mathrm{ev}}\rangle$. This proves
\eqref{eq-Gamma-ev-characterization}.

It is straightforward to verify that
\[
{\rm im}f
=
\left\{
(n_e)_{e\in E}\in\mathbb Z_2^E
\;\middle|\;
\sum_{e\in E}n_e=0
\right\}.
\]
In particular, $|{\rm im}f|=2^{b-1}$. If $b$ is odd, then
$\mathbf 1\notin {\rm im}f$, whereas if $b$ is even, then
$\mathbf 1\in {\rm im}f$. Hence
\[
\left|
\frac{\langle\Gamma_\lambda\rangle}
{\langle\Gamma_{\mathrm{ev}}\rangle}
\right|
=
\begin{cases}
	2^{b-1}, & b \text{ is odd},\\[2pt]
	2^{b-2}, & b \text{ is even},
\end{cases}
=
2^{2\left\lfloor\frac{b-1}{2}\right\rfloor}.
\]
\end{proof}

Let $\ell'$ denote the order of $q$. Define
\begin{align}
\Gamma_q
:=
\begin{cases}
	\ell\Gamma_\lambda+{\rm Span}_{\mathbb N}(S_1),
	& \text{if $\ell'$ is odd},\\
	\ell\Gamma^{\mathrm{ev}}+{\rm Span}_{\mathbb N}(S_1),
	& \text{if $\ell'$ is even}.
\end{cases}
\end{align}
The proof of Theorem~\ref{thm-center} implies that
\begin{align}\label{eq-Gamma-q-characterization}
\Gamma_q
=
\left\{
\mathbf n\in\Gamma_\lambda
\;\middle|\;
\mathbf n\overline Q_\lambda\mathbf k^T=0
\text{ in }\mathbb Z_{\ell'}
\text{ for every }
\mathbf k\in\langle\Gamma_\lambda\rangle
\right\}.
\end{align}

\begin{lemma}\label{lem-center-torus}
We have
\begin{align}
	\langle\Gamma_q\rangle
	=
	\left\{
	\mathbf n\in\langle\Gamma_\lambda\rangle
	\;\middle|\;
	\mathbf n\overline Q_\lambda\mathbf k^T=0
	\text{ in }\mathbb Z_{\ell'}
	\text{ for every }
	\mathbf k\in\langle\Gamma_\lambda\rangle
	\right\}.
\end{align}
\end{lemma}

\begin{proof}
Equation~\eqref{eq-Gamma-q-characterization} immediately gives
\[
\langle\Gamma_q\rangle
\subseteq
\left\{
\mathbf n\in\langle\Gamma_\lambda\rangle
\;\middle|\;
\mathbf n\overline Q_\lambda\mathbf k^T=0
\text{ in }\mathbb Z_{\ell'}
\text{ for every }
\mathbf k\in\langle\Gamma_\lambda\rangle
\right\}.
\]

Conversely, suppose that $\mathbf n\in\langle\Gamma_\lambda\rangle$
satisfies
$\mathbf n\overline Q_\lambda\mathbf k^T=0
\text{ in }\mathbb Z_{\ell'}$
for every $\mathbf k\in\langle\Gamma_\lambda\rangle$. Let
$\mathbf m\in\Gamma_q$ be the vector all of whose entries are equal to
$2\ell$. For all sufficiently large $r\in\mathbb N$, we have
$\mathbf n+r\mathbf m\in\Gamma_\lambda.$
Since both $\mathbf n$ and $\mathbf m$ satisfy the required
orthogonality condition,
\[
(\mathbf n+r\mathbf m)\overline Q_\lambda\mathbf k^T=0
\quad\text{in }\mathbb Z_{\ell'}
\]
for every $\mathbf k\in\langle\Gamma_\lambda\rangle$. It follows from
\eqref{eq-Gamma-q-characterization} that
$\mathbf n+r\mathbf m\in\Gamma_q.$
Since $r\mathbf m\in\langle\Gamma_q\rangle$, we obtain
\[
\mathbf n
=
(\mathbf n+r\mathbf m)-r\mathbf m
\in\langle\Gamma_q\rangle.
\]
This proves the reverse inclusion.
\end{proof}

Define
\[
\overline{\mathcal {X}}_q^{\mathrm{bl}}(\fS,\lambda)
:=
{\rm Span}_R
\left\{
x^{\mathbf n}
\;\middle|\;
\mathbf n\in\langle\Gamma_\lambda\rangle
\right\}
\subseteq
\overline{\mathcal {X}}_q(\fS,\lambda).
\]
It is straightforward to verify that
$\overline{\mathcal {X}}_q^{\mathrm{bl}}(\fS,\lambda)$ is a subalgebra of
$\overline{\mathcal {X}}_q(\fS,\lambda)$.

Lemma~\ref{lem-center-torus} and
\cite[Lemma~3.2(b)]{karuo2024center} imply that
\begin{align}\label{rank-torus}
{\rm rank}_{\mathcal Z}
\overline{\mathcal {X}}_q^{\mathrm{bl}}(\fS,\lambda)
=
\left|
\frac{\langle\Gamma_\lambda\rangle}
{\langle\Gamma_q\rangle}
\right|.
\end{align}
Moreover, Proposition~\ref{prop-Azumaya} and
\cite[Lemmas~6.1 and~6.2]{korinman2021unicity} imply that
\begin{align}\label{rank-torus1}
{\rm rank}_{\mathcal Z}\cS_q^+(\fS)
=
{\rm rank}_{\mathcal Z}
\overline{\mathcal {X}}_q^{\mathrm{bl}}(\fS,\lambda).
\end{align}

The following Theorem formulates the rank of the Muller--Roger--Yang skein $\cS_q^+(\fS)$ over its center when $q^2$ is a root of unity of odd order. 

\begin{theorem}\label{prop-rank}
Under the assumption of Theorem~\ref{thm-center},
we have
\[
{\rm rank}_{\mathcal Z}\cS_q^+(\fS)
=
\begin{cases}
	\ell^{\,2a+2b+6g+2d-6},
	& \text{if $\ell'$ is odd},\\[2pt]
	2^{2\left\lfloor\frac{b-1}{2}\right\rfloor}
	\ell^{\,2a+2b+6g+2d-6},
	& \text{if $\ell'$ is even},
\end{cases}
\]
where $a$ is the number of interior punctures,
$g$ is the genus of $\fS$, $b$ is the number of connected components of
$\partial\fS$, and $d$ is the number of components of
$\partial\overline{\fS}$ that contain at least one puncture.
\end{theorem}

\begin{proof}
By~\eqref{rank-torus} and~\eqref{rank-torus1}, it suffices to prove that
\[
\left|
\frac{\langle\Gamma_\lambda\rangle}
{\langle\Gamma_q\rangle}
\right|
=
\begin{cases}
	\ell^{\,2a+2b+6g+2d-6},
	& \text{if $\ell'$ is odd},\\[2pt]
	2^{2\left\lfloor\frac{b-1}{2}\right\rfloor}
	\ell^{\,2a+2b+6g+2d-6},
	& \text{if $\ell'$ is even}.
\end{cases}
\]

Suppose first that $\ell'$ is odd. We have
\begin{align}
	&\left|
	\frac{\langle\Gamma_\lambda\rangle}
	{\langle\Gamma_q\rangle}
	\right|
	=
	\left|
	\frac{\langle\Gamma_\lambda\rangle}
	{\ell\langle\Gamma_\lambda\rangle+\Gamma^\circ}
	\right|
	=
	\left|
	\frac{\langle\Gamma_\lambda\rangle/\Gamma^\circ}
	{\bigl(\ell\langle\Gamma_\lambda\rangle
		+\Gamma^\circ\bigr)/\Gamma^\circ}
	\right|
	\nonumber\\
	=&
	\label{eq-rank-odd}
	\left|
	\frac{\langle\Gamma_\lambda\rangle/\Gamma^\circ}
	{\ell\bigl(\langle\Gamma_\lambda\rangle/
		\Gamma^\circ\bigr)}
	\right|\overset{\eqref{eq-summand}}{=}
	\ell^{
		{\rm rank}_{\mathbb Z}\langle\Gamma_\lambda\rangle
		-
		{\rm rank}_{\mathbb Z}\Gamma^\circ
	}\overset{\eqref{eq-rank-lambda},\,\eqref{rank-circ}}{=}
	\ell^{\,2a+2b+6g+2d-6}.
\end{align}

Now suppose that $\ell'$ is even. Since
$\Gamma^\circ\subseteq\langle\Gamma^{\mathrm{ev}}\rangle$, we obtain
\begin{align*}
	\left|
	\frac{\langle\Gamma_\lambda\rangle}
	{\langle\Gamma_q\rangle}
	\right|
	=
	\left|
	\frac{\langle\Gamma_\lambda\rangle}
	{\langle\Gamma^{\mathrm{ev}}\rangle}
	\right|
	\left|
	\frac{\langle\Gamma^{\mathrm{ev}}\rangle}
	{\ell\left(\langle\Gamma^{\mathrm{ev}}\rangle+\Gamma^\circ\right)}
	\right|=
	2^{2\left\lfloor\frac{b-1}{2}\right\rfloor}
	\ell^{
		{\rm rank}_{\mathbb Z}
		\langle\Gamma^{\mathrm{ev}}\rangle
		-
		{\rm rank}_{\mathbb Z}\Gamma^\circ
	}=
	2^{2\left\lfloor\frac{b-1}{2}\right\rfloor}
	\ell^{\,2a+2b+6g+2d-6},
\end{align*}
where the second equality follows from
Lemma~\ref{lem-gamma-even} and the same argument as in
\eqref{eq-rank-odd}, while the last equality follows from the fact that
$\langle\Gamma^{\mathrm{ev}}\rangle$ has finite index in
$\langle\Gamma_\lambda\rangle$, and hence
${\rm rank}_{\mathbb Z}
\langle\Gamma^{\mathrm{ev}}\rangle
=
{\rm rank}_{\mathbb Z}
\langle\Gamma_\lambda\rangle.$
\end{proof}

\begin{remark}\label{rem-rankkk}
Remark \ref{rem-center-reduced} implies 
$\widetilde{\mathcal Z(\cS_q^+(\fS))} = \widetilde{\mathcal Z(\rdS)}$.
It is well-known that every nonzero element in $\widetilde{\cS_q^+(\fS)}$ is invertible. This shows 
$\widetilde{\cS_q^+(\fS)}=\widetilde{\rdS}$. Then we have 
$\rank_{\mathcal Z} \rdS = \rank_{\mathcal Z} \cS_q^+(\fS)$.
\end{remark}

\bibliography{ref.bib}

\end{document}